\documentclass[a4paper,10pt]{amsart}
\usepackage[utf8]{inputenc}
\usepackage{lmodern}
\usepackage[T1]{fontenc}
\usepackage{microtype} 	
\usepackage[all]{xy}
\usepackage{amsmath,amssymb,amsfonts,amsthm,latexsym,mathrsfs}
\usepackage{mathtools}
\usepackage{ stmaryrd } 
\usepackage{tikz-cd}
\usepackage{leftindex}
\usepackage{upgreek}
\usepackage[shortlabels]{enumitem}
\usepackage{xcolor}
\usepackage{hyperref}
\usepackage{cleveref}
\hypersetup{colorlinks=true,linkcolor=gray,citecolor=gray}

\usepackage[colorinlistoftodos]{todonotes}

\usepackage[style=alphabetic,doi=false,isbn=false,url=false,maxnames=10,maxalphanames=10,sorting=nyt,backref=true]{biblatex}
\DeclareDelimFormat{multicitedelim}{\addcomma\space}

\renewbibmacro*{journal+issuetitle}{%
  \usebibmacro{journal}%
  \iffieldundef{volume}
    { 
      \iffieldundef{year}
        {}
        {\setunit{\addspace}\printtext[parens]{\printfield{year}}}%
    }
    { 
      \setunit{\addspace}\printfield{volume}%
      \iffieldundef{year}
        {}
        {\setunit{\addspace}\printtext[parens]{\printfield{year}}}%
    }%
  \setunit*{\addcomma\space}%
  \printfield{number}%
  \newunit}

\DeclareFieldFormat[article, thesis, incollection, inbook, unpublished, inproceedings]{title}{\mkbibemph{#1}}

\DeclareFieldFormat[article]{journaltitle}{#1}

\DeclareFieldFormat[article]{volume}{\textbf{#1}}

\DeclareFieldFormat[article]{number}{no.~#1}
 
\DeclareFieldFormat[incollection, inbook, inproceedings]{booktitle}{#1}

\renewbibmacro*{in:}{}

\DeclareFieldFormat{pages}{#1}

\newtheorem{lemma}{Lemma}[section]
\newtheorem{proposition}[lemma]{Proposition}
\newtheorem{theorem}[lemma]{Theorem}
\newtheorem{corollary}[lemma]{Corollary}

\newtheorem{maintheorem}{Theorem}

\theoremstyle{definition}

\newtheorem{conjecture}[lemma]{Conjecture}
\newtheorem*{conjecture*}{Conjecture}

\theoremstyle{remark}
\newtheorem{remark}[lemma]{Remark}
\newtheorem{example}[lemma]{Example}

\newtheorem*{remark*}{Remark}
\newtheorem*{problem*}{Problem}
\newtheorem*{lremark*}{Literature remark}

\DeclareMathOperator{\im}{im}
\newcommand\Ob{\mathrm{Ob}}

\newcommand\End{\mathrm{End}}

\newcommand\Hom{\mathrm{Hom}}
\newcommand\Ext{\mathrm{Ext}}
\DeclareMathOperator{\id}{id}

\DeclarePairedDelimiter\abs{\lvert}{\rvert}
\newcommand{\Modcat}{\operatorname{Mod}}
\newcommand\modcat{\operatorname{mod}} 

\newcommand\ind{\operatorname{ind}}

\newcommand{\gldim}{\operatorname{gldim}}

\newcommand{\ten}{\otimes}
\newcommand{\lten}{\ten^{\mathrm{\mathbf{L}}}}
\newcommand{\RHom}{\mathbf{R}\Hom}
\newcommand{\cone}{\mathrm{cone}}

\DeclareMathOperator{\add}{add}
\newcommand{\pvd}{\mathrm{pvd}}
\newcommand{\per}{\mathrm{per}}

\newcommand{\dd}{\mathfrak{d}}
\newcommand{\DD}{\mathscr{D}}

\newcommand\mc{\mathcal}
\newcommand\ov{\overline}
\newcommand\un{\underline}

\newcommand\mf{\mathfrak}
\newcommand\wt{\widetilde}

\DeclareFontFamily{U}{mathx}{}
\DeclareFontShape{U}{mathx}{m}{n}{<-> mathx10}{}
\DeclareSymbolFont{mathx}{U}{mathx}{m}{n}
\DeclareMathAccent{\widecheck}{0}{mathx}{"71}
\newcommand{\N}{\mathbb{N}}
\newcommand{\Z}{\mathbb{Z}}
\newcommand{\F}{\mathbb{F}}

\newcommand{\Cbb}{\mathbb{C}}
\newcommand{\Qbb}{\mathbb{Q}}

\newcommand{\bG}{\mathbf{\Gamma}}
\newcommand{\bGh}{\widehat{\mathbf{\Gamma}}}
\newcommand{\bP}{\mathbf{\Pi}}

 \newcommand\restr[2]{{
   \left.\kern-\nulldelimiterspace 
   #1 
   \right|_{#2} 
   }}

\makeatletter
\providecommand\@dotsep{5}
\renewcommand{\listoftodos}[1][\@todonotes@todolistname]{%
  \@starttoc{tdo}{#1}}
\makeatother

\let\oldtocsubsection=\tocsubsection

\renewcommand{\tocsubsection}[2]{\hspace{2em}\oldtocsubsection{#1}{#2}}

\usepackage[margin=1.4in]{geometry}

\allowdisplaybreaks

\begin{document}

\title{Additive categorification of the monoidal $\Lambda$-invariant}
\author{Ricardo Canesin}
\author{Peigen Cao}
\author{Geoffrey Janssens}

\address{(Ricardo Canesin) \newline Université Paris Cité and Sorbonne Université, CNRS, IMJ-PRG, F-75013 Paris, France \newline E-mail address: {\tt ricardo.canesin@imj-prg.fr}}

\address{(Peigen Cao) \newline School of Mathematical Sciences, University of Science and Technology of China, Hefei, 230026, People's Republic of China \newline E-mail address: {\tt peigencao@126.com}}

\address{(Geoffrey Janssens) \newline Departement Wiskunde, Vrije Universiteit Brussel,
Pleinlaan $2$, 1050 Elsene, Belgium \newline E-mail address: {\tt geofjans@vub.ac.be}}

\begin{abstract}
In this paper, we contribute to the broad aim of relating invariants of additive and monoidal categorifications of cluster algebras. Specifically, in the setting of representations of a quantum affine algebra $U_q'(\mathfrak{g})$, Kashiwara--Kim--Oh--Park proved that the Hernandez--Leclerc categories form a monoidal categorification of their Grothendieck rings. Furthermore, these rings are $\Uplambda$-cluster algebras, meaning they are equipped with a compatible Poisson structure, constructed via the $\Lambda$-invariant. Under certain natural conditions, where $U_q'(\mathfrak{g})$ is of untwisted simply-laced type, we provide an additive interpretation of the $\Lambda$-invariant within the framework of Higgs categories. More precisely, there is an ice quiver with potential associated with these cluster algebras, and a key ingredient of our work consists in proving that its relative Ginzburg algebra is proper. More generally, if the relative Ginzburg algebra associated with an arbitrary ice quiver with potential is proper, we prove that the corresponding cluster algebra admits the structure of a $\Uplambda$-cluster algebra defined in terms of negative extensions in the Higgs category. Moreover, we provide a homological formula to compute the corresponding tropical and $F$-invariants introduced by Cao.
\end{abstract}
\maketitle

\newcommand\blfootnote[1]{%
  \begingroup
  \renewcommand\thefootnote{}\footnote{#1}%
  \addtocounter{footnote}{-1}%
  \endgroup
}


\blfootnote{The second author is partially supported by the National Key R\&D Program of China (2024YFA1013801). The third author is grateful to Fonds Wetenschappelijk Onderzoek vlaanderen - FWO (grant 88258), and le Fonds de la Recherche Scientifique - FNRS (grant 1.B.239.22) for financial support.}

\vspace{-0,7cm}
\tableofcontents

\section{Introduction}

\subsection{Background and outline}\addtocontents{toc}{\protect\setcounter{tocdepth}{1}}

Cluster algebras are a class of commutative algebras introduced by Fomin--Zelevinsky \cite{FominZelevinskyI} in the early 2000s. By definition, they possess a distinguished set of generators constructed via an iterative procedure called mutation. A fruitful approach to understanding their intricate combinatorics has been to interpret them using certain categories, whose rich structure provides more conceptual tools for their study. There are two types of categorification of cluster algebras: additive and monoidal. Although they employ very different techniques, recent work by many authors has revealed correspondences between them that go well beyond what was initially expected. These notably include Fujita \cite{Fujformula,FujE}, Fujita--Oh \cite{FujitaOh}, Fujita--Murakami \cite{FM}, Duan--Schiffler \cite{DS,DuanSchifflerClassificationPreprint}, Contu \cite{ContuMonoidalAdditive}, Baur--Fu--Li \cite{BCJ}, Cao \cite{CaoFinvariant}, and Casbi \cite{Casbi}. In this paper, we contribute to this connection by providing an additive reinterpretation of the $\Lambda$-invariant of Kashiwara--Kim--Oh--Park \cite{KKOPmonI}, an important tool in their study of monoidal categorification \cite{KKOPmonII,KKOPmonIII}.

The \emph{additive} categorification program was initiated soon after the introduction of cluster algebras, starting from the observation of similarities between the combinatorics of mutations and those of tilting theory for finite-dimensional hereditary algebras (cf. \cite{MRZ}). After the invention of cluster categories by Caldero--Chapoton--Schiffler \cite{CCS} and Buan--Marsh--Reineke--Reiten--Todorov \cite{BMRRT}, it was further developed and expanded in numerous articles, such as \cite{CC, GLSinv, CalKel1, CalKel2, BuanIyamaReitenScott, FuKeller, DerksenWeymanZelevinskyI, DerksenWeymanZelevinskyII, AmiotClusterCat, Palu, PlamondonCharacter, PlamondonCluster, PresslandInternallyCY}, culminating in the categorification of skew-symmetric cluster algebras with coefficients by Yilin Wu \cite{YilinWuJacobifinite} and Keller--Wu \cite{KellerWuJacobiinfinite}. In these two papers, to an ice quiver with potential $(Q,F,W)$, one associates a differential graded algebra $\bGh(Q,F,W)$ called the \emph{completed relative Ginzburg dg algebra}. From its derived category, the authors of \cite{YilinWuJacobifinite,KellerWuJacobiinfinite} construct the \emph{Higgs category} $\mc{H}(Q,F,W)$ together with a \emph{cluster character} which, under mild assumptions, categorify the cluster algebra $\mc{A}_{Q,F}$ (with non-invertible coefficients) associated with the ice quiver $(Q,F)$. For instance, the cluster character establishes a bijection between the isomorphism classes of reachable rigid indecomposable objects in $\mc{H}(Q,F,W)$ and the cluster variables in $\mc{A}_{Q,F}$.

In another direction, Hernandez--Leclerc \cite{HLduke} introduced monoidal categorifications of cluster algebras, motivated by character identities in the representation theory of quantum affine algebras. A monoidal abelian length category $\mc{C}$ is a \emph{monoidal categorification} of a cluster algebra $\mc{A}$ if it is endowed with an isomorphism of $\Z$-algebras from the Grothendieck ring $K_0(\mc{C})$ to $\mc{A}$ that identifies the set of cluster monomials of $\mc{A}$ with a subset of the set of isomorphism classes of simple objects in $\mc{C}$, whose elements are called the \emph{reachable} simple objects. For a complex finite-dimensional simple Lie algebra $\mf{g}_0$, Hernandez--Leclerc \cite{HLduke,HLcluster} defined monoidal subcategories $\mathscr{C}_{\ell}$ ($\ell \geq 1$) and $\mathscr{C}^-$ of the category of finite-dimensional modules over the untwisted quantum affine algebra $U_q'(\widehat{\mf{g}}_0)$, and showed that their Grothendieck rings admit natural cluster algebra structures. They conjectured that these categories provide monoidal categorifications of these cluster algebras. After pioneering work by Hernandez--Leclerc and Nakajima \cite{Nakcluster} (see also Kimura--Qin \cite{KimQin}), a major breakthrough came with the work of Kang--Kashiwara--Kim--Oh \cite{KKKOjams} and Qin \cite{Qinduke}. Building on the tools developed in \cite{KKKOjams}, Hernandez--Leclerc's conjectures were completely proved by Kashiwara--Kim--Oh--Park \cite{KKOPmonI,KKOPmonII,KKOPmonIII}, who generalized them to a much larger family of monoidal subcategories $\mathscr{C}_{\mf{g}}^{[a,b],\mc{D},\un{w}_0}$ of the category of representations of an arbitrary quantum affine algebra $U'_q(\mf{g})$.

For a pair $(V,W)$ of finite-dimensional integrable simple modules over $U'_q(\mf{g})$, Kashiwara--Kim--Oh--Park \cite{KKOPmonI} defined the \emph{$\Lambda$-invariant} $\Lambda(V,W) \in \Z$ and the \emph{$\dd$-invariant} $\dd(V,W) \in \Z_{\geq 0}$, inspired by similar invariants constructed in \cite{KKKOjams} for representations of quiver Hecke algebras. These invariants are extracted from the commutation morphism $V \ten W \to W \ten V$ arising from the universal $R$-matrix, and they encode important information about the structure of these tensor products. They play a fundamental role in the proofs of \cite{KKOPmonI,KKOPmonII,KKOPmonIII} and allow Kashiwara--Kim--Oh--Park to establish a more refined result: the category $\mathscr{C}_{\mf{g}}^{[a,b],\mc{D},\un{w}_0}$ is a \emph{$\Uplambda$-monoidal categorification}. This implies, in particular, that the commutative ring $K_0(\mathscr{C}_{\mf{g}}^{[a,b],\mc{D},\un{w}_0})$ has a compatible Poisson structure \cite{GekhtmanShapiroVainshtein} and is a \emph{$\Uplambda$-cluster algebra}. A $\Uplambda$-cluster algebra is a cluster algebra in which each seed with extended exchange matrix $\widetilde{B}$ is endowed with an additional skew-symmetric integer matrix $\Lambda$, called the \emph{Poisson coefficient matrix}, such that the pair $(\widetilde{B},\Lambda)$ is compatible in the sense of Berenstein--Zelevinsky \cite{BerensteinZelevinsky}. Moreover, the compatible pairs associated with two neighboring seeds are related by a mutation as defined in loc.\ cit. In the case of the category $\mathscr{C}_{\mf{g}}^{[a,b],\mc{D},\un{w}_0}$, if $\{V_i\}_{i \in K}$ is the family of simple modules corresponding to the cluster variables of a seed, then the associated Poisson coefficient matrix $\Lambda = (\lambda_{ij})_{i,j \in K}$ is given by $\lambda_{ij} = \Lambda(V_i,V_j)$ for $i,j \in K$.

Motivated in part by these results, the second-named author \cite{CaoFinvariant} (see also \cite{CaoLi}) introduced the \emph{tropical invariant} and the \emph{$F$-invariant} for pairs of cluster monomials (or, more generally, of \emph{good elements}) in an arbitrary (upper) $\Uplambda$-cluster algebra. These invariants can be computed using $F$-polynomials and extended $g$-vectors, and he showed that they recover the $\Lambda$-invariant and the $\dd$-invariant for pairs of reachable simple modules in the $\Uplambda$-cluster algebras of Kashiwara--Kim--Oh--Park. Moreover, he proved that the $F$-invariant agrees with the $E$-invariant of decorated representations of quivers with potential in the additive categorification framework of Derksen--Weyman--Zelevinsky \cite{DerksenWeymanZelevinskyI, DerksenWeymanZelevinskyII} (see also \cite{DeFe}), confirming a conjecture by Fujita \cite{FujE}.

Our goal in this paper is to complete this picture by providing an additive interpretation of the tropical invariant (and hence of the $\Lambda$-invariant) in the setting of Higgs categories. For a non-degenerate ice quiver with potential $(Q,F,W)$ such that $\bGh(Q,F,W)$ is a \emph{proper} dg algebra, we prove that the associated cluster algebra $\mc{A}_{Q,F}$ admits a structure of $\Uplambda$-cluster algebra defined in terms of \emph{negative extensions} in the Higgs category $\mc{H}(Q,F,W)$, and we give a homological formula to compute the corresponding tropical and $F$-invariants. Inspired by the work of Contu \cite{ContuMonoidalAdditive}, we upgrade the initial exchange matrix of Kashiwara--Kim--Oh--Park for $\mathscr{C}_{\mf{g}}^{[a,b],\mc{D},\un{w}_0}$ into an ice quiver with potential when the interval $[a,b]$ is finite, and consider the associated Higgs category. When $U_q'(\mf{g})$ is untwisted of simply-laced type and the PBW-pair $(\mc{D},\un{w}_0)$ is \emph{adapted}, we show that the corresponding relative Ginzburg dg algebra is proper and that our $\Uplambda$-cluster algebra structure coincides with the one defined by Kashiwara--Kim--Oh--Park, thereby providing an additive categorification of their $\Lambda$- and $\dd$-invariants.

We now explain our main results in more detail.

\subsection{A \texorpdfstring{$\Uplambda$}{}-cluster algebra structure through additive categorification}\addtocontents{toc}{\protect\setcounter{tocdepth}{1}}

Let $(Q,F)$ be a finite ice quiver. We allow arrows between frozen vertices, which may be either unfrozen or frozen. Assume that the sets of vertices are $Q_0 = \{1,2,\dots,m\}$ and $F_0 = \{n+1,\dots,m\}$ for $1 \leq n \leq m$. We consider the $m \times m$ matrix $\widehat{B} = (b_{ij})$ defined by
\[
b_{ii} = \begin{cases}
    0 &\textrm{if } 1 \leq i \leq n,\\
    1 &\textrm{otherwise,}
\end{cases}
\]
and, for $i \neq j$,
\[
b_{ij} = \#\{\textrm{unfrozen arrows } i\to j\textrm{ in } Q\} - \#\{\textrm{arrows } j\to i \textrm{ in } Q\}.
\]
The first $n$ columns form the extended exchange matrix $\widetilde{B}$ of the ice quiver $(Q,F)$. If $\widehat{B}$ is invertible over $\Qbb$, we define
\[
\Lambda = \abs{\det \widehat{B}} \cdot (\widehat{B}^{-T} - \widehat{B}^{-1}).
\]
In this case, we show that the pair $(\widetilde{B},\Lambda)$ is compatible (see Proposition \ref{compatible pair in inv case}) and call it the \emph{compatible pair associated with} $(Q,F)$. This construction is compatible with mutation of ice quivers (taking into account the arrows between frozen vertices!) as defined by Pressland \cite{PresslandMutation} (see Lemma \ref{proposition:compatible pair of an ice quiver respects mutation}).

If one equips $(Q,F)$ with a potential $W$, then the matrix $\widehat{B}$ is the matrix of the Euler form of the perfectly valued derived category $\pvd(\bGh)$ with respect to the basis of simple dg modules, where $\bGh = \bGh(Q,F,W)$ is the completed relative Ginzburg dg algebra. When $\bGh(Q,F,W)$ is proper, it follows that $\widehat{B}$ is invertible (see Lemma \ref{lemma:inverse of Euler matrix}), and hence the compatible pair associated with $(Q,F)$ is defined. We will interpret its Poisson coefficient matrix using the Higgs category $\mc{H} = \mc{H}(Q,F,W)$. To this end, we define
\begin{equation}\label{eq:bracket on the Higgs category (introduction)}
[M,N]_{\mc{H}} = \sum_{p \geq 0}(-1)^p(\dim\Ext^{-p}_{\mc{H}}(M,N) - \dim\Ext^{-p}_{\mc{H}}(N,M)),
\end{equation}
for $M,N \in \mc{H}$, where the negative extensions are computed in the relative cluster category (see Section \ref{section:relative cluster and Higgs categories}). Under the assumption that $\bGh(Q,F,W)$ is proper, this sum is finite (see Lemma \ref{lemma:bracket is defined}). 

To any vertex $t$ of the labeled $n$-regular tree $\mathbb{T}_n$, we associate an ice quiver with potential $(Q_t,F_t,W_t)$ obtained from $(Q,F,W)$ by iterated mutation along the path in $\mathbb{T}_n$ from the root to $t$. Our first main result is the following.

\begin{maintheorem}[Theorem \ref{thm:canonical quantum structure}]\label{thm:canonical quantum structure (introduction)}
    Let $(Q,F,W)$ be a non-degenerate ice quiver with potential such that $\bGh(Q,F,W)$ is proper. Let $\mc{H} = \mc{H}(Q,F,W)$ be the associated Higgs category. Then the compatible pair associated with $(Q_t,F_t)$ is defined for all $t \in \mathbb{T}_n$, and its Poisson coefficient matrix $\Lambda_t = (\lambda_{ij}^t)$ is given by
    \[
    \lambda_{ij}^t = [M_i^t,M_j^t]_{\mc{H}},
    \]
    where $M_i^t \in \mc{H}$ is the reachable rigid indecomposable object corresponding to the $i$-th cluster variable in the seed of $\mc{A}_{Q,F}$ associated with $t$. In particular, this endows $\mc{A}_{Q,F}$ with a $\Uplambda$-cluster algebra structure.
\end{maintheorem}

If $\bGh(Q,F,W)$ is moreover concentrated in degree zero, then the Higgs category is a Frobenius exact category by a result of Yilin Wu \cite[Theorem 6.2]{YilinWuJacobifinite}. As observed in Remark \ref{rem: Frobenis exact case}, it follows that
\[
[M,N]_{\mc{H}} = \dim\Hom_{\mc{H}}(M,N) - \dim\Hom_{\mc{H}}(N,M)
\]
for $M,N \in \mc{H}$. Consequently, the $\Uplambda$-cluster algebra structure on $\mc{A}_{Q,F}$ arising from Theorem \ref{thm:canonical quantum structure (introduction)} coincides with that of Gei\ss--Leclerc--Schröer \cite{GLSQuantumstructure} and Grabowski--Pressland \cite[Theorem 6.22]{GrabowskiPressland}.

The next result gives a homological interpretation of Cao's tropical and $F$-invariants for pairs of cluster monomials in $\mc{A}_{Q,F}$, computed using the $\Uplambda$-cluster algebra structure of Theorem \ref{thm:canonical quantum structure (introduction)}.

\begin{maintheorem}[Proposition \ref{prop:additive interpretation of the tropical and F invariants}]
    Let $(Q,F,W)$ be a non-degenerate ice quiver with potential such that $\bGh(Q,F,W)$ is proper. Let $\mc{H} = \mc{H}(Q,F,W)$ be the associated Higgs category. For reachable rigid objects $M,N \in \mc{H}$, we have
    \[
    \langle CC(M),CC(N)\rangle_{\rm{trop}} = \dim\Ext^1_{\mc{H}}(M,N) + [M,N]_{\mc{H}}
    \]
    and
    \[
    (CC(M) \mid\mid CC(N))_F = 2 \cdot \dim\Ext^1_{\mc{H}}(M,N),
    \]
    where $CC$ denotes the canonical cluster character on $\mc{H}$.
\end{maintheorem}

As explained in Remark \ref{rem: recovering sym E invariant}, a result of Plamondon \cite{PlamondonCharacter} allows us to express $\dim\Ext^1_{\mc{H}}(M,N)$ in the theorem above as an $E$-invariant between the decorated representations corresponding to $M$ and $N$. Hence, we recover, in our setting, the link between the $E$-invariant and the $F$-invariant established in \cite[Theorem 6.11]{CaoFinvariant}.

\subsection{Additive categorification of the monoidal \texorpdfstring{$\Lambda$}{}-matrix}\addtocontents{toc}{\protect\setcounter{tocdepth}{1}}

Let us recall the main ingredients in the definition of the categories $\mathscr{C}_{\mf{g}}^{[a,b],\mc{D},\un{w}_0}$ and of their initial monoidal seeds, introduced by Kashiwara--Kim--Oh--Park. Let $\mf{g}$ be an affine Kac--Moody Lie algebra and consider the associated quantum affine algebra $U_q'(\mf{g})$. By \cite[Theorem 3.6]{KKOPSimplyLacedRootSystems}, one can attach to $U_q'(\mf{g})$ a simply-laced Dynkin diagram $\Delta$ of finite type whose root system controls the block decomposition of the category of finite-dimensional integrable $U_q'(\mf{g})$-modules. For example, if $\mf{g}$ is of type $\mathsf{A}^{(1)}$, $\mathsf{D}^{(1)}$, or $\mathsf{E}^{(1)}$, then $\Delta$ is the corresponding Dynkin diagram of finite type. In \cite{KKOP_PBWQuantumAffine}, the notion of a \emph{complete PBW-pair} is introduced. It consists of a pair $(\mc{D},\un{w}_0)$, where $\mc{D}$ is a \emph{complete duality datum} and $\un{w}_0$ is a reduced expression of the longest element of the Weyl group of $\Delta$. The category $\mathscr{C}_{\mf{g}}^{[a,b],\mc{D},\un{w}_0}$ depends on the choice of an integer interval $[a,b]$, where $a,b \in \Z \cup \{\pm\infty\}$, and on a complete PBW-pair $(\mc{D},\un{w}_0)$.

When $b \in \Z$, Kashiwara--Kim--Oh--Park describe the ice quiver of an initial seed for the cluster algebra structure on $K_0(\mathscr{C}_{\mf{g}}^{[a,b],\mc{D},\un{w}_0})$ by adapting earlier constructions from \cite{BerensteinFominZelevinsky,BuanIyamaReitenScott,GLSKacMoodyGroups}. We denote it by $(Q^{[a,b]}(\un{w}_0),F^{[a,b]}(\un{w}_0))$, as it does not depend on $\mc{D}$. When the interval $[a,b]$ is finite, we endow this ice quiver with a potential $W^{[a,b]}(\un{w}_0)$, building on \cite{BuanIyamaReitenSmith,HLcluster,ContuMonoidalAdditive}. The definition of the resulting ice quiver with potential is given in Section \ref{section:initial monoidal seed}. This allows us to study the Higgs category \cite{YilinWuJacobifinite} that categorifies the corresponding cluster algebra. We remark that Contu \cite{ContuMonoidalAdditive} has already considered an additive counterpart to the monoidal categorification of Kashiwara--Kim--Oh--Park. However, for his purposes, it was sufficient to categorify the cluster algebra \emph{without} coefficients, whereas retaining coefficients is crucial for our results.

Let $\bGh^{[a,b]}(\un{w}_0)$ be the completed relative Ginzburg dg algebra of the ice quiver with potential $(Q^{[a,b]}(\un{w}_0),F^{[a,b]}(\un{w}_0),W^{[a,b]}(\un{w}_0))$. We conjecture that it is always a proper dg algebra. 

\begin{conjecture*}[\Cref{conj: ginzburg proper}]
The relative Ginzburg dg algebra $\bGh^{[a,b]}(\underline{w}_0)$ is proper. 
\end{conjecture*}

The next result confirms this conjecture in two particular cases.

\begin{maintheorem}[Corollaries \ref{cor:adapted words give proper Ginzburgs} and \ref{cor:properness if length of w0 divides length of [a,b]}]\label{thm:proper Ginzburg algebra (introduction)}
    Let $[a,b]$ be a finite integer interval and $\un{w}_0$ a reduced expression of the longest element $w_0$ of the Weyl group of $\Delta$. The completed relative Ginzburg dg algebra $\bGh^{[a,b]}(\un{w}_0)$ is a proper dg algebra if one of the following conditions holds:
    \begin{enumerate}[(1)]
        \item $\un{w}_0$ is a source sequence for some orientation of $\Delta$;
        \item $b-a+1$ is a multiple of the length of $w_0$.
    \end{enumerate}
\end{maintheorem}

To prove (1), we first show that $\bGh^{[a,b]}(\un{w}_0)$ can be replaced by the \emph{non-completed} relative Ginzburg dg algebra $\bG^{[a,b]}(\un{w}_0)$ (see Lemma \ref{lemma:complete/noncomplete Ginzburg are qiso when Jacobi-finite}). By a relative version of a result of Keller \cite{KellerCY} (see Theorem \ref{theorem:Ginzburg algebras as completion}), $\bG^{[a,b]}(\un{w}_0)$ is a relative Calabi--Yau completion in the sense of Yeung \cite{YeungRelativeCY}. This completion is proper if a certain relative inverse dualizing bimodule is nilpotent. We prove that this is indeed the case by decomposing it and reducing the argument to the case of $\Delta = \mathsf{A}_1$. Along the way, we establish a ``gluing'' result for relative inverse dualizing bimodules which may be of independent interest (see Proposition \ref{proposition:gluing of bimodules}). To prove (2), we study the effect of commutation and braid moves on $\un{w}_0$, following ideas of Contu \cite[Section 6.1]{ContuMonoidalAdditive}. See also Proposition \ref{prop:words that we can prove give a proper Ginzburg} for a slight strengthening of the theorem.

We now state the main result of this paper. We assume that $\mf{g}$ is untwisted of simply-laced type, that is, of type $\mathsf{A}^{(1)}$, $\mathsf{D}^{(1)}$, or $\mathsf{E}^{(1)}$. We further assume that the PBW-pair $(\mc{D},\un{w}_0)$ is \emph{adapted to an orientation of $\Delta$}, meaning that $\mc{D}$ is the complete duality datum associated with a height function on $\Delta$ (see Section \ref{section:initial monoidal seed}) and $\un{w}_0$ is a source sequence for the corresponding orientation. Let $[a,b]$ be a finite integer interval. By Theorem \ref{thm:proper Ginzburg algebra (introduction)}, $\bGh^{[a,b]}(\un{w}_0)$ is proper, so the bracket \eqref{eq:bracket on the Higgs category (introduction)} is defined for objects in the associated Higgs category $\mc{H}^{[a,b]}(\un{w}_0)$. Let $CC$ be the canonical cluster character from $\mc{H}^{[a,b]}(\un{w}_0)$ to the corresponding cluster algebra, denoted $\mc{A}^{[a,b]}(\un{w}_0)$, and let $\varphi: K_0(\mathscr{C}_{\mf{g}}^{[a,b],\mc{D},\un{w}_0}) \to \mc{A}^{[a,b]}(\un{w}_0)$ be the isomorphism defining the $\Uplambda$-monoidal categorification of Kashiwara--Kim--Oh--Park.

\begin{maintheorem}[Theorem \ref{thm:coincidence of Lambda matrices}]\label{thm:coincidence of Lambda matrices (introduction)}
    Suppose $\mf{g}$ is untwisted of simply-laced type. Let $(\mc{D},\un{w}_0)$ be a complete PBW-pair adapted to an orientation of $\Delta$, and let $a,b \in \Z$ with $a \leq b$. For reachable simple objects $V,W \in \mathscr{C}_{\mf{g}}^{[a,b],\mc{D},\un{w}_0}$, we have
    \[
    \Lambda(V,W) = \dim\Ext^1_{\mc{H}}(M,N) + [M,N]_{\mc{H}},
    \]
    where $\mc{H} = \mc{H}^{[a,b]}(\un{w}_0)$ and $M,N \in \mc{H}$ are the corresponding reachable rigid objects such that $\varphi(V) = CC(M)$ and $\varphi(W) = CC(N)$ in $\mc{A}^{[a,b]}(\un{w}_0)$.
\end{maintheorem}

We prove this theorem by showing that the $\Uplambda$-cluster algebra structures on $\mc{A}^{[a,b]}(\un{w}_0)$ arising from Theorem \ref{thm:canonical quantum structure (introduction)} and from the $\Uplambda$-monoidal categorification coincide. To this end, we use a formula of Kashiwara--Kim--Oh--Park expressing the $\Lambda$-invariant in terms of the $\dd$-invariant and the left duality functor $\DD^{-1}$ (see Proposition \ref{prop: lambda as sum of delta}), namely
\[
\Lambda(V,W) = \dd (V,W) +  \sum_{n \geq 1} (-1)^{n-1} \left[\dd(V, \DD^{-n}(W))  - \dd(W, \DD^{-n}(V)) \right].
\]
This expression closely resembles the definition of the bracket $[M,N]_{\mc{H}}$ given in \eqref{eq:bracket on the Higgs category (introduction)}. A comparison of the two formulas shows that it is enough to prove that
\[
\dd(V,\DD^{-n}(W)) = \dim\Ext^{1-n}_{\mc{H}}(M,N)
\]
for $n \geq 0$. This is established in Proposition \ref{prop:d-invariant with dual as a negative extension} when $V$ and $W$ belong to the initial monoidal seed, which was previously described by Hernandez--Leclerc in \cite{HLcluster}. By exploiting an infinite sequence of mutations considered in loc.\ cit., we show in Proposition \ref{proposition:mutation sequence the gives left dual} that $\DD^{-1}(W)$ can be computed via a maximal green sequence. Up to some technical details, the desired equality then follows from results of Keller \cite{KellerQuantumDilog,KellerClusterDerivedSurvey}, together with the additive and monoidal categorifications of the $F$-invariant. We remark that the connection between the duality functor on the monoidal side and the suspension functor on the additive side has been observed before; see, for instance, \cite[Section 3.6]{KashiwaraKimLaurent} and \cite[Table 3]{QinClusterAlgebrasBases}.

Keeping in mind that $\bGh^{[a,b]}(\un{w}_0)$ is expected to be proper for arbitrary finite intervals $[a,b]$ and reduced expressions $\un{w}_0$, we expect that the hypotheses on $\mf{g}$ and $(\mc{D},\un{w}_0)$ in Theorem \ref{thm:coincidence of Lambda matrices (introduction)} are not necessary and can be removed.

\subsection{Organization}\addtocontents{toc}{\protect\setcounter{tocdepth}{1}}

This paper is organized as follows. In Section \ref{section:review on cluster algebras}, we briefly review the theory of $\Uplambda$-cluster algebras and introduce the tropical and $F$-invariants from \cite{CaoFinvariant}. In Section \ref{section:additive categorification of cluster algebras}, we recall the additive categorification of cluster algebras with coefficients following \cite{YilinWuJacobifinite,KellerWuJacobiinfinite}, focusing on the Jacobi-finite case. In Section \ref{section additive lambda}, we establish the $\Uplambda$-cluster algebra structure arising from proper relative Ginzburg dg algebras and provide an additive interpretation of the tropical invariant. Section \ref{section:background on monoidal categorification} contains the necessary background on quantum affine algebras and the work of Kashiwara--Kim--Oh--Park. In Section \ref{section:maximal green sequences}, we present results on maximal green sequences and explain how they can be used to realize the left duality functor $\DD^{-1}$ on certain $U_q'(\mf{g})$-modules. In Section \ref{section:inverse dualizing bimodules}, we develop the required results on inverse dualizing bimodules and relative derived preprojective algebras, including the proof of the gluing theorem (Proposition \ref{proposition:gluing of bimodules}). The proof of Theorem \ref{thm:proper Ginzburg algebra (introduction)} is given in Section \ref{section:properness of the Ginzburg algebra}. A substantial part of this section is devoted to Theorem \ref{thm:coincidence of extensions}, which provides a key technical result on the coincidence of extensions for relative Ginzburg dg algebras in the adapted case. Finally, in Section \ref{section:additive categorification of the Lambda-matrix}, we prove Theorem \ref{thm:coincidence of Lambda matrices (introduction)} using the results developed in the previous sections.

\vspace{0,2cm}
\noindent \textbf{Acknowledgments.} The authors are grateful to Bernhard Keller for numerous interesting discussions and his invaluable insight into Ginzburg algebras. The first author also thanks him for his guidance as PhD advisor. We thank Fan Qin for helpful email exchanges concerning the relationship between the duality and suspension functors.

\section{Review on cluster algebras and Cao's invariants} \label{section:review on cluster algebras}

We give a brief reminder on the theory of $\Uplambda$-cluster algebras (also known as cluster algebras with compatible Poisson structures), following mainly \cite{CaoFinvariant}. We then recall the definition and main properties of the tropical invariant and the $F$-invariant introduced in loc.\ cit.

Let $K = K^{\rm{un}} \sqcup K^{\rm{fr}}$ be a countable index set which decomposes into a non-empty subset $K^{\rm{un}}$ of \emph{unfrozen} indices and a subset $K^{\rm{fr}}$ of \emph{frozen} indices. Following \cite{BerensteinZelevinsky} (and \cite[Section 5]{KKOPmonI}), we say that a pair $(\wt{B},\Lambda)$ is a \emph{compatible pair} if
\begin{enumerate}[(1)]
    \item $\wt{B} = (b_{ij})$ is a $K \times K^{\rm{un}}$ integer matrix such that, for any $j \in K^{\rm{un}}$, there are finitely many $i \in K$ such that $b_{ij} \neq 0$;
    \item $\Lambda = (\lambda_{ij})$ is a $K \times K$ skew-symmetric integer matrix;
    \item we have $\wt{B}^T\Lambda = (S \mid {\bf 0})$, where $S$ is a $K^{\rm{un}} \times K^{\rm{un}}$ diagonal matrix whose diagonal entries are strictly positive integers and ${\bf 0}$ is the $K^{\rm{un}} \times K^{\rm{fr}}$ zero matrix.
\end{enumerate}
We call $\wt{B}$ the \emph{extended exchange matrix}, $\Lambda$ the \emph{Poisson coefficient matrix}, and $S$ the \emph{type} of the compatible pair. If $B$ denotes the submatrix of $\wt{B}$ indexed by $K^{\rm{un}} \times K^{\rm{un}}$, we call $B$ the \emph{principal part} of $\wt{B}$. By \cite[Proposition 3.3]{BerensteinZelevinsky}, $B$ is a skew-symmetrizable matrix with skew-symmetrizer $S$ and, if $K$ is finite, $\wt{B}$ has full rank.

\begin{remark}\label{rmk:matrices and quivers}
    In most of this paper, we will work with compatible pairs $(\wt{B},\Lambda)$ such that the principal part of $\wt{B}$ is skew-symmetric. In this case, the data of $\wt{B}$ is essentially the same as that of a pair $(Q,F)$ where $Q = (Q_0,Q_1)$ is a quiver without loops or $2$-cycles and $F \subseteq Q_0$ is a subset of \emph{frozen vertices}. We also require that any unfrozen vertex is incident to finitely many arrows. Given such a pair and relabeling the vertices so that $Q_0 = K$ and $F = K^{\rm{fr}}$, the corresponding matrix $\wt{B} = (b_{ij})$ is defined by
    \[
    b_{ij} = \#\{i \to j \textrm{ in } Q_1\} - \#\{j \to i \textrm{ in } Q_1\}
    \]
    for $i \in K$ and $j \in K^{\rm{un}}$. Notice that the number of arrows between frozen vertices is irrelevant. However, when $K$ is finite, we will need to upgrade $F$ to a \emph{subquiver} of $Q$  for the additive categorification of cluster algebras with coefficients. The \emph{frozen arrows} will play an important role (see Section \ref{section:additive categorification of cluster algebras}). As we will see in Section \ref{section:compatible pair of ice quiver}, this extra data also encodes a matrix $\Lambda$ and yields a compatible pair $(\widetilde{B},\Lambda)$.
\end{remark}

Let $\F$ be the field of rational functions on $|K|$ variables over $\Qbb$. We call a triple $({\bf x},\wt{B},\Lambda)$ a \emph{$\Uplambda$-seed} in $\F$ if
\begin{enumerate}[(1)]
    \item ${\bf x} = \{x_i\}_{i \in K}$ is a \emph{free generating set} of $\F$, that is, a transcendence basis of $\F$ that generates it;
    \item $(\widetilde{B},\Lambda)$ is a compatible pair.
\end{enumerate}
We call ${\bf x}$ the \emph{cluster} of the $\Uplambda$-seed, and its elements the \emph{cluster variables}. We say that the cluster variables $x_i$ for $i \in K^{\rm{un}}$ are \emph{unfrozen}, while the remaining ones are \emph{frozen}. A monomial on the cluster variables of a $\Uplambda$-seed is called a \emph{cluster monomial}.

Let $({\bf x},\wt{B},\Lambda)$ be a $\Uplambda$-seed in $\F$ and write ${\bf x} = \{x_i\}_{i \in K}$, $\wt{B} = (b_{ij})$ and $\Lambda = (\lambda_{ij})$. For $v \in K^{\rm{un}}$, we define the \emph{mutation} $\mu_v({\bf x},\wt{B},\Lambda)$ of $({\bf x},\wt{B},\Lambda)$ \emph{in the direction of $v$} as a $\Uplambda$-seed $({\bf x'},\mu_v(\wt{B}),\mu_v(\Lambda))$ where ${\bf x'} = \{x_i'\}_{i \in K}$, $\mu_v(\wt{B}) = (b_{ij}')$, and $\mu_v(\Lambda) = (\lambda_{ij}')$ are defined by
\begin{align*}
x_i' &= \begin{cases}
    x_i &\textrm{if } i\neq v,\\
    x_v^{-1}\cdot(\prod_{j\in K}x_j^{[b_{jv}]_+} + \prod_{j\in K}x_j^{[-b_{jv}]_+}) &\textrm{otherwise},
\end{cases}\\[1em]
b_{ij}' &= \begin{cases}
    -b_{ij} &\textrm{if } i = v \textrm{ or } j=v,\\
    b_{ij} + [b_{iv}]_+[b_{vj}]_+ - [-b_{iv}]_+[-b_{vj}]_+ &\textrm{otherwise,}
\end{cases}\\[1em]
\lambda_{ij}' &= \begin{cases}
    -\lambda_{iv} + \sum_{l\in K}[ b_{lv}]_+\lambda_{il} &\textrm{if } i \neq v \textrm{ and } j = v,\\
    -\lambda_{vj} + \sum_{l\in K}[ b_{lv}]_+\lambda_{lj} &\textrm{if } i = v \textrm{ and } j \neq v,\\
    \lambda_{ij} &\textrm{otherwise}.
\end{cases}
\end{align*}
Here $[a]_+$ denotes $\max\{a,0\}$ for $a \in \Z$. Note that the sums and products above are all finite by the definition of the extended exchange matrix. We remark that $(\mu_v(\wt{B}),\mu_v(\Lambda))$ is indeed a compatible pair and has the same type as $(\wt{B},\Lambda)$ by \cite[Proposition 3.4]{BerensteinZelevinsky}. Moreover, mutation of $\Uplambda$-seeds in a fixed direction is an involution.

Let $\mathbb{T}_{|K^{\rm{un}}|}$ denote the $|K^{\rm{un}}|$-regular tree. Label its edges by $K^{\rm{un}}$ so that edges incident to the same vertex have different labels. A \emph{$\Uplambda$-seed pattern}
\[
\mc{S} = \{({\bf x}_t,\wt{B}_t,\Lambda_t) \mid t \in \mathbb{T}_{|K^{\rm{un}}|}\}
\]
is an assignment of $\Uplambda$-seeds to the vertices of $\mathbb{T}_{|K^{\rm{un}}|}$ such that $({\bf x}_{t'},\wt{B}_{t'},\Lambda_{t'}) = \mu_v({\bf x}_t,\wt{B}_t,\Lambda_t)$ for any vertices $t$ and $t'$ connected by an edge labeled by $v$. If we fix a root vertex $t_0 \in \mathbb{T}_{|K^{\rm{un}}|}$, then $\mc{S}$ is determined by the $\Uplambda$-seed $({\bf x},\wt{B},\Lambda)$ associated with $t_0$, which we call the \emph{initial} $\Uplambda$-seed. The \emph{$\Uplambda$-cluster algebra (with non-invertible coefficients)} $\mc{A} = \mc{A}_{\wt{B}}$ associated with $\mc{S}$ is the $\Z$-subalgebra of $\F$ generated by all the cluster variables of all the $\Uplambda$-seeds in $\mc{S}$. If $\wt{B}$ comes from a pair $(Q,F)$ as in Remark \ref{rmk:matrices and quivers}, we will denote $\mc{A}$ by $\mc{A}_{Q,F}$. If $K$ is finite (resp. infinite), we say that $\mc{A}$ has \emph{finite} (resp. \emph{infinite}) \emph{rank}.

\begin{remark}
    As the notation suggests, the $\Uplambda$-cluster algebra does not depend a priori on the chosen initial matrix $\Lambda$. In fact, as originally done by Fomin--Zelevinsky \cite{FominZelevinskyI}, by disregarding the Poisson coefficient matrices in all the definitions above, we can define a \emph{cluster algebra} for any matrix $\wt{B}$ satisfying the finiteness condition above and whose principal part is skew-symmetrizable. However, we remark that this extra data can be used to define a compatible Poisson structure \cite{GekhtmanShapiroVainshtein} on the $\Uplambda$-cluster algebra, and the corresponding quantization is called a \emph{quantum cluster algebra} \cite{BerensteinZelevinsky}. Although we do not work with quantum cluster algebras, we still want to keep track of the Poisson coefficient matrices of the $\Uplambda$-seeds of $\mc{A}$. This is the reason for the definition of $\Uplambda$-cluster algebras given above.
\end{remark}

For the rest of the section, we assume that $\mc{A}$ has finite rank. In this case, we will always take $K = \{1,\dots,m\}$ and $K^{\rm{un}} = \{1,\dots,n\}$ for fixed integers $m \geq n \geq 1$.

Each $\Uplambda$-seed $({\bf x}_t,\wt{B}_t,\Lambda_t)$ of $\mc{A}$ determines a partial order $\preceq_t$ on $\Z^m$ defined as follows: for ${\bf g},{\bf g'} \in \Z^m$, we have ${\bf g} \preceq_t {\bf g'}$ if there is ${\bf v} \in \N^n$ such that ${\bf g'} = {\bf g} + \wt{B}_t{\bf v}$. This is indeed a partial order because $\wt{B}_t$ has full rank. We write ${\bf g}\prec {\bf g'}$ if we additionally have ${\bf g} \neq {\bf g'}$. For ${\bf g} = (g_1,\dots,g_m) \in \Z^m$, we denote ${\bf x}^{\bf g}_t = x_{1;t}^{g_1}\dotsb x_{m;t}^{g_m}$, where ${\bf x}_t = (x_{1;t},\dots,x_{m;t})$. Notice that ${\bf g'} = {\bf g} + \wt{B}_t{\bf v}$ is equivalent to ${\bf x}^{\bf g'}_t = {\bf x}^{\bf g}_t \cdot \widehat{{\bf y}}_t^{\bf v}$, where $\widehat{{\bf y}}_t^{\bf v} = {\bf x}^{\wt{B}_t{\bf v}}_t$ is a monomial in the variables $(\widehat{y}_{1;t},\dots,\widehat{y}_{n;t})$ defined by $\widehat{y}_{t;j} = \prod_{i = 1}^mx_{i;t}^{b_{ij}}$.

\begin{theorem}[{\cite{FominZelevinskyIV,GrossHackingKeelKontsevich}}]\label{thm:g-vectors and F-polynomials}
For $t \in \mathbb{T}_n$, any cluster monomial $u \in \mc{A}$ can be written (uniquely) as
\[
u = {\bf x}_t^{{\bf g}^t_u} + \sum_{{\bf h \prec {\bf g}^t_u}}b_{\bf h}{\bf x}_t^{\bf h} = {\bf x}_t^{{\bf g}^t_u} \cdot F^t_u(\widehat{y}_{1;t},\dots,\widehat{y}_{n;t}),
\]
where $b_{\bf h} \in \Z_{\geq 0}$ and $F^t_u \in \Z[y_1,\dots,y_n]$ is a polynomial with constant term $1$.    
\end{theorem}

The vector ${\bf g}^t_u \in \Z^m$ is called the \emph{(extended) $g$-vector} of $u$ with respect to $t$, and $F^t_u$ is the \emph{$F$-polynomial} of $u$ with respect to $t$. We remark that ${\bf g}^t_{uu'} = {\bf g}^t_u + {\bf g}^t_{u'}$ for two cluster monomials $u,u' \in \mc{A}$ associated to the same $\Uplambda$-seed by \cite[Lemma 3.2.6]{QinJEMS24}. In particular, $g$-vectors are determined by the $g$-vectors of the cluster variables. We give a categorical interpretation for them in Section \ref{section:categorification after Yilin}.

For a nonzero polynomial
\[
F = \sum_{{\bf v} \in \N^n}c_{\bf v}{\bf y}^{\bf v} \in \Z[y_1,\dots,y_n],
\]
its \emph{tropical evaluation} on ${\bf r} \in \Z^n$ is defined as
\[
F[{\bf r}] = \max\{{\bf v}^T{\bf r} \mid c_{\bf v} \neq 0\} \in \Z.
\]
If $F$ has a nonzero constant term, the integer above is non-negative. We can finally define the invariants of \cite{CaoFinvariant}. For two cluster monomials $u,u' \in \mc{A}$, their \emph{tropical invariant} is defined as
\begin{equation}\label{def tropical invariant}
\langle u,u'\rangle_{\rm{trop}} = ({\bf g}^t_u)^T\Lambda_t{\bf g}^t_{u'} + F^t_u[(S \mid {\bf 0}){\bf g}^t_{u'}] \in \Z
\end{equation}
for any $t \in \mathbb{T}_n$, where $\wt{B}_t^T\Lambda_t = (S \mid {\bf 0})$. By \cite[Theorem 1.3]{CaoFinvariant}, it does not depend on the choice of $t$. The \emph{$F$-invariant} is given by
\begin{equation}\label{def F-invariant}
(u \mid\mid u')_F = \langle u,u'\rangle_{\rm{trop}} + \langle u',u\rangle_{\rm{trop}} = F^t_u[(S \mid {\bf 0}){\bf g}^t_{u'}] + F^t_{u'}[(S \mid {\bf 0}){\bf g}^t_{u}] \in \Z_{\geq 0}
\end{equation}
for any $t \in \mathbb{T}_n$. The last equality follows from the definition of the tropical invariant and the fact that $\Lambda_t$ is skew-symmetric. Consequently, observe that the $F$-invariant does not depend on the Poisson coefficient matrices. In fact, it can also be defined for finite rank cluster algebras with trivial coefficients without a compatible Poisson structure (see \cite[Section 4.4]{CaoFinvariant}).

\begin{theorem}[{\cite[Theorem 1.5]{CaoFinvariant}}]
    Let $u,u' \in \mc{A}$ be two cluster monomials. They are cluster monomials associated with the same ($\Uplambda$-)seed if and only if $(u \mid\mid u')_F = 0$.
\end{theorem}

\begin{remark}
The original definitions of the tropical and $F$-invariants are not the ones given above, but they coincide by \cite[Theorem 1.3]{CaoFinvariant}. Moreover, they can be defined for a larger class of elements in the \emph{upper} $\Uplambda$-cluster algebra called \emph{good elements} in loc.\ cit.\ (building on the notion of \emph{compatibly pointed elements} of \cite{QinJEMS24}). These are certain elements of the upper $\Uplambda$-cluster algebra that have a decomposition similar to that of Theorem \ref{thm:g-vectors and F-polynomials}.
\end{remark}

\section{Additive categorification of cluster algebras with coefficients}\label{section:additive categorification of cluster algebras}

We present the necessary background on the additive categorification of cluster algebras with coefficients using relative Ginzburg dg algebras and their Higgs categories, following mainly \cite{YilinWuJacobifinite} and \cite{KellerWuJacobiinfinite}. We will focus on the Jacobi-finite case, even though the theory can be developed more generally.

From now on, $k$ will denote an algebraically closed field. We assume $k = \Cbb$ in Section \ref{section:categorification after Yilin} and in the results that depend on the cluster character that is introduced there.

\subsection{Ice quivers with potential}\addtocontents{toc}{\protect\setcounter{tocdepth}{2}} 
In this section, we give the necessary background on ice quivers with potential and their mutations. More details can be found in Sections 2, 3, and 4 of \cite{PresslandMutation}, which builds on the foundational work of \cite{DerksenWeymanZelevinskyI}.

An \emph{ice quiver} is a pair $(Q,F)$ consisting of a quiver $Q$ and a subquiver $F$. In this case, we refer to $F$ as the \emph{frozen} subquiver of $Q$, and its vertices and arrows are called \emph{frozen}.

Let $Q$ be a finite quiver and denote by $\widehat{kQ}$ the completion of its path algebra with respect to the ideal generated by all arrows. Let $HH_0(\widehat{kQ})$ be the \emph{zeroth continuous Hochschild cohomology} of $\widehat{kQ}$, that is, the quotient of $\widehat{kQ}$ by the closure of its subspace generated by all commutators. It has a topological basis given by all cycles in $Q$ up to cyclic permutation. A \emph{potential} $W$ on a quiver $Q$ is an element of $HH_0(\widehat{kQ})$ that is a (possibly infinite) linear combination of cycles of length at least two such that any term involving a loop has degree at least three. An \emph{ice quiver with potential} is a triple $(Q,F,W)$ where $(Q,F)$ is a \emph{finite} ice quiver and $W$ is a potential in $Q$. If $F$ is empty, we call $(Q,W)$ a \emph{quiver with potential}. We say that $W$ is \emph{irredundant} if each term of $W$ contains at least one unfrozen arrow. We define $(Q,F,W)$ to be \emph{reduced} if $W$ is irredundant and none of its terms is a $2$-cycle.

For a cyclic path $p = \alpha_1\dotsb\alpha_n$ in $Q$ (with $\alpha_i \in Q_1$) and an arrow $\alpha \in Q_1$, the \emph{cyclic derivative} of $p$ with respect to $\alpha$ is defined by
\[
\partial_{\alpha}(p) = \sum_{\alpha_i = \alpha}\alpha_{i+1}\dotsb\alpha_n\alpha_1\dotsb\alpha_{i-1}.
\]
We can extend $\partial_{\alpha}$ by linearity and continuity to a well-defined linear map $\partial_{\alpha}: HH_0(\widehat{kQ}) \to \widehat{kQ}$. For an ice quiver with potential $(Q,F,W)$, we define the \emph{relative Jacobian algebra} as the quotient of $\widehat{kQ}$ by the closure of the ideal generated by the elements $\partial_{\alpha}(W)$ for all \emph{unfrozen} arrows $\alpha$. We call $(Q,F,W)$ \emph{Jacobi-finite} if $J(Q,F,W)$ is finite-dimensional.

By \cite[Theorem 3.6]{PresslandMutation}, for any ice quiver with potential $(Q,F,W)$, there exists a \emph{reduced} ice quiver with potential $(Q_{\rm{red}},F_{\rm{red}},W_{\rm{red}})$ such that
\[
J(Q,F,W) \cong J(Q_{\rm{red}},F_{\rm{red}},W_{\rm{red}}).
\]
If $W$ is irredundant, then $(Q_{\rm{red}},F_{\rm{red}},W_{\rm{red}})$ is unique up to right equivalence (see Definition 3.7 and Proposition 3.15 in \cite{PresslandMutation}). In this case, $(Q_{\rm{red}},F_{\rm{red}},W_{\rm{red}})$ is called the \emph{reduction} of $(Q,F,W)$. It is an important ingredient for the mutation of ice quivers with potential, which we now define. From now on, suppose $W$ is irredundant. Let $v \in Q_0 \setminus F_0$ be an unfrozen vertex that is not incident to any loop or $2$-cycle. Define an ice quiver with potential $\widetilde{\mu}_v(Q,F,W)$ by the following procedure:
\begin{enumerate}[(1)]
    \item For each pair of arrows $\alpha: u \to v$ and $\beta: v \to w$, add an unfrozen arrow $[\beta\alpha]: u \to w$.
    \item Replace each arrow $\alpha$ incident with $v$ by an arrow $\alpha^*$ going in the opposite direction. This gives a new ice quiver $(\widetilde{\mu}_v(Q),F)$.
    \item Choose a representative for $W$ in $\widehat{kQ}$ such that no term begins at $v$. For each pair of arrows as in (1), replace each occurrence of $\beta\alpha$ by the newly introduced arrow $[\beta\alpha]$, and add the term $[\beta\alpha]\alpha^*\beta^*$. This gives a new potential $\widetilde{\mu}_v(W)$.
    \item Set $\widetilde{\mu}_v(Q,F,W) = (\widetilde{\mu}_v(Q),F,\widetilde{\mu}_v(W))$.
\end{enumerate}
The \emph{mutation} $\mu_v(Q,F,W) = (\mu_v(Q),\mu_v(F),\mu_v(W))$ of $(Q,F,W)$ is defined to be the reduction of $\widetilde{\mu}_v(Q,F,W)$. We say that $(Q,F,W)$ is \emph{non-degenerate} if any iterated sequence of mutations of $(Q,F,W)$ does not produce $2$-cycles containing unfrozen arrows in the resulting ice quiver. By \cite[Proposition 4.10]{PresslandMutation}, this holds for example if $(Q,F,W)$ is reduced and \emph{rigid}, that is, any cycle in $Q$ containing an unfrozen arrow is zero in $J(Q,F,W)$ up to cyclic equivalence (see also Definition 4.8 in loc.\ cit.).

\begin{remark}
     If $(\ov{Q},\ov{W})$ denotes the quiver with potential obtained from $(Q,F,W)$ by deleting the frozen vertices, then $(\ov{Q},\ov{W})$ is non-degenerate and $\mu_v(\ov{Q},\ov{W}) = (\ov{\mu_v(Q)},\ov{\mu_v(W)})$ for any unfrozen vertex $v$.
\end{remark}

\begin{remark}
By \cite[Proposition 4.6]{PresslandMutation}, if $(\mu_v(Q),\mu_v(F))$ does not have $2$-cycles containing unfrozen arrows, then it coincides with the \emph{extended Fomin--Zelevinsky mutation} $\mu_v^{\rm{FZ}}(Q,F)$ of the ice quiver $(Q,F)$, which is constructed by the following procedure:
\begin{enumerate}[(1)]
    \item  For each pair of arrows $\alpha: u \to v$ and $\beta: v \to w$, add an unfrozen arrow $[\beta\alpha]: u \to w$.
    \item Replace each arrow $\alpha$ incident with $v$ by an arrow $\alpha^*$ going in the opposite direction.
    \item Remove a maximal collection of $2$-cycles involving only unfrozen arrows.
    \item Choose a maximal collection of half-frozen $2$-cycles, that is, involving exactly one frozen arrow. Replace each $2$-cycle in this collection by a frozen arrow pointing in the same direction as the unfrozen arrow of the $2$-cycle.
\end{enumerate}
The resulting quiver is defined up to isomorphism. If we disregard the frozen arrows, then the procedure above gives the classical mutation rule (as defined in Section \ref{section:review on cluster algebras} in terms of extended exchange matrices).
\end{remark}

\subsection{Relative Ginzburg dg algebras} Let $(Q, F, W)$ be a finite ice quiver with potential. We define $\widetilde{Q}$ to be the graded quiver whose vertices are the same as those of $Q$ and whose arrows are
\begin{itemize}
\item a copy of each arrow of $Q$, with degree $0$;
\item an arrow $\alpha^*: j \rightarrow i$ of degree $-1$ for every \emph{unfrozen} arrow $\alpha: i \to j$ of $Q$;
\item a loop $t_i: i \rightarrow i$ of degree $-2$ for every \emph{unfrozen} vertex $i$ of $Q$.
\end{itemize}
We define the \emph{completed relative Ginzburg dg algebra} $\bGh(Q, F, W)$ as follows. Its underlying graded algebra is the completion of the graded path algebra $k\widetilde{Q}$ with respect to the ideal generated by the arrows. The completion is taken in the category of graded algebras, so that the $n$-th component of $\bGh(Q, F, W)$ has as a topological basis the set of all paths in $\widetilde{Q}$ of degree $n$. The differential is the unique continuous linear endomorphism of degree $1$ which satisfies $d(uv)=du\cdot v+(-1)^p u \cdot dv$ for all homogeneous elements $u$ of degree $p$ and all $v$, and takes the following values on the arrows of $\widetilde{Q}$:
\begin{itemize}
\item $d\alpha = 0$ for every $\alpha \in Q_1$;
\item $d\alpha^* = \partial_{\alpha}(W)$ for every $\alpha \in Q_1 \setminus F_1$;
\item $dt_i = e_i\left(\sum_{\alpha \in Q_1}\left[\alpha, \alpha^*\right]\right)e_i$ for every $i \in Q_0 \setminus F_0$, where $e_i$ denotes the lazy path of length zero at vertex $i$.
\end{itemize}
Note that $J(Q,F,W)$ is isomorphic to $H^0(\bGh(Q,F,W))$.

When $W$ involves finitely many cycles, we also define the \emph{non-completed relative Ginzburg dg algebra} $\bG(Q,F,W)$ as the dg subalgebra of $\bGh(Q,F,W)$ whose underlying algebra is the non-completed path algebra $k\widetilde{Q}$. We have the following result linking the two of them.

\begin{lemma}\label{lemma:complete/noncomplete Ginzburg are qiso when Jacobi-finite}
    Let $(Q,F,W)$ be an ice quiver with potential that is Jacobi-finite and such that $W$ involves finitely many cycles. Suppose that we can give a non-negative Adams grading to the arrows of $\widetilde{Q}$ such that the differential $d$ defined above is homogeneous of Adams degree zero and no cycle in $\widetilde{Q}$ has Adams degree zero. Then the canonical morphism
    \[
    \bG(Q,F,W) \longrightarrow \bGh(Q,F,W)
    \]
    is a quasi-isomorphism.
\end{lemma}

\begin{proof}
We start with the following observation. Since $(Q,F,W)$ is Jacobi-finite and $\bG(Q,F,W)$ is concentrated in non-negative degrees and smooth (by \cite[Corollary 5.4(4)]{YeungRelativeCY}), we can apply \cite[Proposition 2.5]{KalckYangI} to conclude that all the cohomologies of $\bG(Q,F,W)$ are finite-dimensional. Now, both $\bGh(Q,F,W)$ and $\bG(Q,F,W)$ can be equipped with an Adams grading on top of the cohomological one using the extra grading on $\widetilde{Q}$. By our hypothesis, there are only finitely many paths on $\widetilde{Q}$ of a given Adams degree, so the component $(k\widetilde{Q})^{r,p}$ of $k\widetilde{Q}$ of cohomological degree $r$ and Adams degree $p$ has finite dimension. This allows us to write the $r$-th cohomological components of $\bGh(Q,F,W)$ and $\bG(Q,F,W)$ as
\[
\bGh(Q,F,W)^r = \prod_{p \geq 0}(k\widetilde{Q})^{r,p} \quad \textrm{and} \quad \bG(Q,F,W)^r = \bigoplus_{p \geq 0}(k\widetilde{Q})^{r,p}.
\]
Since $d$ is homogeneous of Adams degree zero, we may restrict it to a map $d^{r,p}$ from $(k\widetilde{Q})^{r,p}$ to $(k\widetilde{Q})^{r+1,p}$. If we denote the restriction of $d$ to $\bG(Q,F,W)$ as $d_{\bG}$, we deduce that
\[
\ker d^r = \prod_{p \geq 0}\ker d^{r,p}  \quad \textrm{and} \quad \ker d^r_{\bG}  = \bigoplus_{p \geq 0}\ker d^{r,p}. 
\]
But $H^r(\bG(Q,F,W))$ is finite-dimensional, so there is $p_r \geq 0$ such that the image of $d_{\bG}$ contains $\ker d^{r,p}$ for all $p \geq p_r$. By the continuity of $d$, we also deduce that the image of $d$ contains $\prod_{p \geq p_r} \ker d^{r,p}$. We conclude that both $H^r(\bG(Q,F,W))$ and $H^r(\bGh(Q,F,W))$ are given by the finite direct sum
\[
\bigoplus_{0 \leq p < p_r} \frac{\ker d^{r,p}}{\im d^{r,p}}.
\]
This gives the desired quasi-isomorphism.
\end{proof}

\begin{remark}\label{rmk:defining an Adams grading}
Suppose we can give a strictly positive grading to the arrows of $Q$ such that the potential $W$ is homogeneous of degree $n$ and all arrows have degree strictly less than $n$. Then we can extend it to an Adams grading on the arrows of $\widetilde{Q}$ as follows. For $\alpha \in Q_1$, we set the Adams degree of $\alpha^*$ as $n - |\alpha|$, and we set each $t_i$ to be of degree $n$. Each arrow in $\widetilde{Q}$ then has positive Adams degree, and the conditions in the lemma above are satisfied. For example, if all cycles in $W$ have the same length, we can set each arrow of $Q$ to have degree $1$ and then the resulting Adams grading satisfies the required conditions.
\end{remark}

\subsection{Relative cluster categories and Higgs categories}\label{section:relative cluster and Higgs categories} In this section, we fix a Jacobi-finite ice quiver with potential $(Q,F,W)$ and recall some of the constructions and results of \cite{YilinWuJacobifinite}. We remark that they can be generalized to the Jacobi-infinite setting \cite{KellerWuJacobiinfinite}, but we will not need this level of generality.

\begin{remark}
    Our hypothesis that $(Q,F,W)$ is Jacobi-finite ensures that the technical condition \cite[Assumption 1]{KellerWuJacobiinfinite} holds. This guarantees that the next results and the results in Section \ref{section:categorification after Yilin} are true without extra assumptions.
\end{remark}

Let $\bGh = \bGh(Q,F,W)$. Denote by $\mc{D}(\bGh)$ its derived category and by $\Sigma$ the suspension functor on $\mc{D}(\bGh)$. We denote by $\per(\bGh)$ the \emph{perfect derived category} of $\bGh$, that is, the thick subcategory of $\mc{D}(\bGh)$ generated by $\bGh$. We let $\pvd(\bGh)$ be the \emph{perfectly valued derived category} of $\bGh$, that is, the full subcategory of $\mc{D}(\bGh)$ whose objects are the dg $\bGh$-modules with finite-dimensional total cohomology. By the proof of \cite[Theorem 2.19(a)]{KellerYang}, $\pvd(\bGh)$ is a thick subcategory of $\per(\bGh)$ and is generated by all simple $H^0(\bGh)$-modules. Note that there is one simple $H^0(\bGh)$-module for each vertex of $Q$.

The \emph{relative cluster category} $\mc{C}(Q,F,W)$ is defined as the Verdier quotient of $\per(\bGh)$ by the thick subcategory generated by all simple $H^0(\bGh)$-modules associated with \emph{unfrozen} vertices. It is a $\Hom$-finite, idempotent complete, and Krull--Schmidt triangulated category by \cite[Lemma 2.17]{KellerYang} and \cite[Section 4.4]{YilinWuJacobifinite}. The \emph{Higgs category} $\mc{H}(Q,F,W)$ is its full subcategory whose objects are the $X \in \mc{C}(Q,F,W)$ such that
\[
\Ext^n_{\mc{C}(Q,F,W)}(X,e_i\bGh) = \Ext^n_{\mc{C}(Q,F,W)}(e_i\bGh,X) = 0
\]
for all $n \geq 0$ and $i \in F_0$. We remark that $\bGh \in \mc{H}(Q,F,W)$. The Higgs category is closed under extensions in $\mc{C}(Q,F,W)$ and hence canonically inherits from it the structure of an extriangulated category \cite{NakaokaPalu} equipped with a bivariant $\delta$-functor (as in \cite[Example 4.9(2)]{GorskyNakaokaPalu}). In particular, we have positive and negative (!) extension groups between objects in $\mc{H}(Q,F,W)$, which are computed as extension groups in $\mc{C}(Q,F,W)$. Moreover, $\mc{H}(Q,F,W)$ is a stably $2$-Calabi--Yau Frobenius extriangulated category whose subcategory $\mc{P}$ of projective-injective objects is the additive subcategory generated by the objects of the form $e_i\bGh$ for a frozen vertex $i \in F_0$. This implies, in particular, that there is a natural isomorphism
\[
\Ext^1_{\mc{H}(Q,F,W)}(M,N) \cong D\Ext^1_{\mc{H}(Q,F,W)}(N,M)
\]
for $M,N \in \mc{H}(Q,F,W)$, where $D$ denotes duality over $k$.

\begin{proposition}[{\cite[Lemma 3.30]{XiaofaChenExactdgcatII}}]\label{proposition:negative extensions can be computed in per Gamma}
    For $i,j \in Q_0$ and $n \leq 1$, the canonical map
    \[
    \Ext^n_{\bGh}(e_i\bGh,e_j\bGh) \longrightarrow \Ext^n_{\mc{C}(Q,F,W)}(e_i\bGh,e_j\bGh) = \Ext^n_{\mc{H}(Q,F,W)}(e_i\bGh,e_j\bGh)
    \]
    induced by the quotient functor $\per(\bGh) \to \mc{C}(Q,F,W)$ is an isomorphism.
\end{proposition}

Let $(\ov{Q},\ov{W})$ be the quiver with potential obtained from $(Q,F,W)$ by deleting the frozen vertices. We have a dg quotient morphism
\[
p: \bGh(Q,F,W) \longrightarrow \bGh(\ov{Q},\ov{W})
\]
and, in particular, $(\ov{Q},\ov{W})$ is again Jacobi-finite by \cite[Proposition 4.5]{YilinWuJacobifinite}. The \emph{(absolute) cluster category} $\mc{C}(\ov{Q},\ov{W})$ is defined as above if we regard $(\ov{Q},\ov{W})$ as an ice quiver with potential with empty frozen part, and it agrees with Amiot's cluster category \cite{AmiotClusterCat}. The morphism $p$ above induces a triangulated quotient functor
\[
p^*: \mc{C}(Q,F,W) \longrightarrow \mc{C}(\ov{Q},\ov{W}),
\]
which in turn yields an equivalence of triangulated categories
\begin{equation}\label{eq:stable Higgs is absolute cluster}
\mc{H}(Q,F,W)/[\mc{P}] \xrightarrow{\sim} \mc{C}(\ov{Q},\ov{W})
\end{equation}
by \cite[Proposition 4.17]{KellerWuJacobiinfinite}. As a consequence, we have the following result.

\begin{proposition}[{\cite[Proposition 5.36]{YilinWuJacobifinite}}]\label{prop:Ext1 is the same in Higgs and absolute cluster categories}
    For $X,Y \in \mc{H}(Q,F,W)$ and $n \geq 1$, the canonical map
    \[
    \Ext^n_{\mc{H}(Q,F,W)}(X,Y) \longrightarrow \Ext^n_{\mc{C}(\ov{Q},\ov{W})}(p^*(X),p^*(Y))
    \]
    is an isomorphism.
\end{proposition}

We conclude this subsection with the effect of the mutation of ice quivers with potential in the associated categories. Suppose $(Q,F,W)$ is reduced and let $v \in Q_0 \setminus F_0$ be an unfrozen vertex that is not incident to any loop or $2$-cycle. We denote $\bGh' = \bGh(Q',F',W')$, where $\mu_v(Q,F,W) = (Q',F',W')$. We remark that $(Q',F',W')$ is again Jacobi-finite. By \cite[Theorem 1.1(b)]{YilinWuMutation}, there are two triangle equivalences $\Phi_{v,\pm}: \mc{D}(\bGh') \to \mc{D}(\bGh)$ that restrict to equivalences $\per(\bGh') \to \per(\bGh)$ and $\pvd(\bGh') \to \pvd(\bGh)$. By \cite[Proposition 3.29]{KellerWuJacobiinfinite}, they both descend to equivalences between the relative cluster categories and, on this level, they become naturally isomorphic functors. We denote this common equivalence by $\Phi_v: \mc{C}(Q,F,W) \to \mc{C}(Q',F',W')$. In turn, $\Phi_v$ restricts to an equivalence between the corresponding Higgs categories and preserves the classes of projective-injective objects. Consequently, we also have an induced triangle equivalence $\mc{C}(\ov{Q},\ov{W}) \to \mc{C}(\ov{Q'},\ov{W'})$, which coincides with the one defined in \cite{KellerYang}.

For the next result, recall that a dg algebra is \emph{proper} if its total cohomology is finite-dimensional.

\begin{lemma}\label{lemma:mutation preserves properness}
    The completed relative Ginzburg dg algebra $\bGh$ is proper if and only if its mutation $\bGh'$ is proper.
\end{lemma}

\begin{proof}
    This follows from the discussion above since a completed relative Ginzburg dg algebra is proper if and only if its perfect and perfectly valued derived categories coincide.
\end{proof}

\subsection{The categorification}\label{section:categorification after Yilin} For this section, the base field is $k = \Cbb$. Let $(Q,F,W)$ be an ice quiver with potential. We assume that it is Jacobi-finite and non-degenerate. We also require that $Q$ has no loops and no $2$-cycles containing unfrozen arrows, and we relabel the vertices so that $Q_0 = \{1,\dots,m\}$ and $F_0 = \{n+1,\dots,m\}$ for $1 \leq n \leq m$. Let $\mathbb{F}$ be the field of rational functions in $m$ variables over $\Qbb$ and let $\mc{A}_{Q,F} \subseteq \mathbb{F}$ be cluster algebra associated with $(Q,F)$ (see Section \ref{section:review on cluster algebras}).

In \cite{KellerWuJacobiinfinite}, building on previous works such as \cite{Palu}, \cite{FuKeller}, and \cite{PlamondonCharacter}, the authors construct a canonical \emph{cluster character} $CC: \Ob(\mc{H}) \longrightarrow \mathbb{F}$ on the Higgs category $\mc{H} = \mc{H}(Q,F,W)$. This means that
\begin{enumerate}[(1)]
    \item $CC$ is constant on isomorphism classes;
    \item $CC(L \oplus M) = CC(L) \cdot CC(M)$ for $L,M \in \mc{H}$;
    \item if $L,M \in \mc{H}$ are such that $\Ext^1_{\mc{H}}(L,M)$ is one-dimensional (hence $\Ext^1_{\mc{H}}(M,L)$ is also one-dimensional), and
    \[
    L \rightarrowtail E \twoheadrightarrow M \dashrightarrow \quad \textrm{and} \quad M \rightarrowtail E' \twoheadrightarrow L \dashrightarrow
    \]
    are non-split conflations in $\mc{H}$, then $CC(L) \cdot CC(M) = CC(E) + CC(E')$.
\end{enumerate}
We say that $M \in \mc{H}$ is \emph{rigid} if $\Ext^1_{\mc{H}}(M,M) = 0$. It is \emph{reachable} if it is a direct summand of a direct sum of copies of the image of $\bGh(Q',F',W')$ under the equivalence
\[
\Phi_{v_1}\Phi_{v_2}\dotsb\Phi_{v_r}: \mc{H}(Q',F',W') \longrightarrow \mc{H}
\]
induced by a sequence of mutations at unfrozen vertices $v_1,\dots,v_r$. We call $T \in \mc{H}$ a \emph{cluster-tilting object} if it is rigid and, if $L \in \mc{H}$ satisfies $\Ext^1_{\mc{H}}(L,T) = 0$, then $L \in \add T$, where $\add T \subseteq \mc{H}$ denotes the smallest full subcategory containing $T$ and closed under direct sums and summands. If $T$ is also \emph{basic}, that is, its indecomposable direct summands are pairwise non-isomorphic, then it is isomorphic to the image of $\bGh(Q',F',W')$ under an equivalence as above.

\begin{theorem}[{\cite[Theorem 5.9]{KellerWuJacobiinfinite}}]
    The canonical cluster character $CC: \Ob(\mc{H}) \to \mathbb{F}$ induces a bijection between the sets of isomorphism classes of reachable rigid indecomposable objects in $\mc{H}$ and cluster variables in $\mc{A}_{Q,F}$. Under this bijection, basic reachable cluster-tilting objects correspond to clusters.
\end{theorem}

Let $T \in \mc{H}$ be a reachable cluster-tilting object. Its corresponding cluster is associated with a vertex $t \in \mathbb{T}_n$ in the $n$-regular tree. By considering the sequence of mutations from the initial root to $t$, we can also attach to it an ice quiver with potential $(Q_t,F_t,W_t)$. The split Grothendieck group $K_0(\add T)$ is then isomorphic to the Grothendieck group of $\per(\bGh(Q_t,F_t,W_t))$, which we identify with $\Z^m$ by sending $e_i\bGh(Q_t,F_t,W_t)$ to the $i$-th vector of the canonical basis.

Now take $M \in \mc{H}$. By the construction of the Higgs category in \cite{KellerWuJacobiinfinite}, there is a conflation
\[\begin{tikzcd}
	{T_1} & {T_0} & M & {}
	\arrow[tail, from=1-1, to=1-2]
	\arrow[two heads, from=1-2, to=1-3]
	\arrow[dashed, from=1-3, to=1-4]
\end{tikzcd}\]
where $T_0,T_1 \in \add T$. By \cite[Proposition 3.2]{GrabowskiPressland} (see also the references mentioned therein), the element $[T_0]-[T_1] \in K_0(\add T)$ does not depend on the chosen conflation. Via the identification above, we can then attach a vector $\ind_T(M) \in \Z^m$ that depends only on $T$ and $M$, called the \emph{index} of $M$ with respect to $T$. It is a categorical version of the $g$-vector by the following result.

\begin{theorem}[{\cite[Theorem 4.39]{GrabowskiPressland}}]\label{thm:g-vectors as indices}
    Let $T \in \mc{H}$ be a reachable cluster-tilting object corresponding to $t \in \mathbb{T}_n$. If $M \in \mc{H}$ is a reachable rigid object, then $\ind_T(M)$ is the $g$-vector ${\bf g}^t_u$ of the cluster monomial $u = CC(M) \in \mc{A}_{Q,F}$ with respect to $t$.
\end{theorem}

\section{An additive \texorpdfstring{$\Lambda$}{}-invariant}\label{section additive lambda}

We start this section by attaching a compatible pair to any finite ice quiver whenever the so-called Euler matrix is invertible. For an ice quiver with potential $(Q,F,W)$, we prove the Euler matrix is invertible whenever $\bGh(Q,F,W)$ is proper, and give a homological interpretation for the corresponding Poisson coefficient matrix using the Higgs category. Our main result is Theorem \ref{thm:canonical quantum structure}, which gives a $\Uplambda$-cluster algebra structure for $\mc{A}_{Q,F}$ and generalizes previous results by \cite{GLSQuantumstructure} and \cite{GrabowskiPressland}. We then interpret the associated tropical invariant categorically. Our formula should be thought of as an ``additive'' $\Lambda$-invariant, since it will be used in Theorem \ref{thm:coincidence of Lambda matrices} to recover the monoidal $\Lambda$-invariant in certain cases.

\subsection{The compatible pair associated with an ice quiver}\label{section:compatible pair of ice quiver} Let $(Q,F)$ be a finite ice quiver. We label its vertices so that $Q_0 = \{1,\dots,m\}$ and $F_0 = \{n+1,\dots,m\}$ for $1 \leq n \leq m$. We define the \emph{Euler matrix} of $(Q,F)$ to be the $m \times m$ integer matrix $\widehat{B} = (b_{ij})$ given by
\[
b_{ii} = \begin{cases}
    0 &\textrm{if } 1 \leq i \leq n,\\
    1 &\textrm{otherwise,}
\end{cases}
\]
and, for $i \neq j$,
\[
b_{ij} =
    \#\{i\to j\textrm{ in } Q_1 \setminus F_1\}-\#\{j\to i \textrm{ in } Q_1\}.
\]
In the next section, we give an alternative description of this matrix using the Euler form of the perfectly valued derived category of a relative Ginzburg dg algebra. 

Note that $b_{ij} = -b_{ji}$ if $i$ or $j$ is unfrozen, but $\widehat{B}$ is not skew-symmetric if $n < m$. Moreover, the submatrix of $\widehat{B}$ defined as
\begin{equation}\label{definition B tilde submatrix}
    \widetilde{B} \coloneqq \text{ first $n$ columns of }\widehat{B}
\end{equation}
is the extended exchange matrix associated with $(Q,F_0)$ as in Remark \ref{rmk:matrices and quivers}. If $\det \widehat{B} \neq 0$, we define
\begin{equation}\label{Lambda if det invertible}
\Lambda \coloneqq \abs{\det \widehat{B}} \cdot (\widehat{B}^{-T} - \widehat{B}^{-1}).
\end{equation}
This is an $m \times m$ skew-symmetric integer matrix.

\begin{proposition}\label{compatible pair in inv case}
    Assume $\det \widehat{B} \neq 0$. The pair $(\widetilde{B},\Lambda)$ defined in \eqref{definition B tilde submatrix} -- \eqref{Lambda if det invertible} is compatible. Its type is $2\abs{\det \widehat{B}} \cdot I_n$, where $I_n$ is the $n \times n$ identity matrix.
\end{proposition}

\begin{proof}
   We have $\widetilde{B}^T\widehat{B}^{-T} = (\widehat{B}^{-1}\widetilde{B})^T = (I_n \mid {\bf 0})$ because $\widetilde{B}$ consists of the first $n$ columns of $\widehat{B}$. Moreover, since $b_{ij} = -b_{ji}$ if $i$ or $j$ is not frozen, the submatrix of $\widehat{B}$ consisting of the first $n$ rows is $-\widetilde{B}^T$, so $\widetilde{B}^T\widehat{B}^{-1} = (-I_n \mid {\bf 0})$. We conclude that $\widetilde{B}^T\Lambda = 2\abs{\det \widehat{B}} \cdot (I_n \mid {\bf 0})$, as desired.
\end{proof}

We refer to $(\widetilde{B},\Lambda)$ as \emph{the compatible pair associated with the ice quiver} $(Q,F)$, which is defined whenever the Euler matrix is invertible. Let us study its behavior under mutations. 

Let $1 \leq v \leq n$ be an unfrozen vertex not incident to loops or $2$-cycles. Let $E_v = (e_{ij})$ be the $m \times m$ integer matrix given by
\[
e_{ij} \coloneqq \begin{cases}
    \delta_{ij} &\textrm{if } j \neq v,\\
    [b_{iv}]_+ &\textrm{if } i \neq v \textrm{ and } j = v,\\
    -1 &\textrm{if } i = j = v.
\end{cases}
\]
We remark that $E_v$ squares to the identity and the mutation rule for Poisson coefficient matrices can be written as $\mu_v(\Lambda) = E_v^T\Lambda E_v$. Indeed, $E_v$ is the matrix denoted by $E_{-}$ in \cite{BerensteinZelevinsky}.

\begin{lemma}\label{lemma:mutation of Euler matrix}
    Let $\mu_v^{\rm{FZ}}(Q,F)$ be the extended Fomin--Zelevinsky mutation of $(Q,F)$ and denote by $\mu_v(\widehat{B})$ its Euler matrix. Then $\mu_v(\widehat{B}) = E_v\widehat{B}E_v^T$.
\end{lemma}

\begin{proof}
Since $b_{vv} = 0$, the entries of the product above are given by
\[
b_{ij}' = \begin{cases}
    -b_{ij} &\textrm{if } i =v \textrm{ or } j = v,\\
    b_{ij} + b_{iv}[b_{jv}]_+ +[b_{iv}]_+b_{vj} &\textrm{otherwise}.
\end{cases}
\]
Using that $b_{jv} = -b_{vj}$ and dividing in cases depending on the signs of $b_{iv}$ and $b_{vj}$, one shows that the second line above is equal to
\[
b_{ij} + [b_{iv}]_+[b_{vj}]_+ -[-b_{iv}]_+[-b_{vj}]_+.
\]
If $i = v$ or $j = v$, then it is clear that $b_{ij}'$ is also the entry $(i,j)$ of $\mu_v(\widehat{B})$, since the step (2) in the extended Fomin--Zelevinsky mutation reverses arrows incident to $v$ (which are all unfrozen) and the other steps do not influence this entry. If $i$ and $j$ are different from $v$, at most one of the terms $[b_{iv}]_+[b_{vj}]_+$ and $[-b_{iv}]_+[-b_{vj}]_+$ in the expression above is nonzero. Since $v$ is not incident to $2$-cycles, they count the number of arrows added between $i$ and $j$ at step (2) of the mutation. Therefore, it is clear that $E_v\widehat{B}E_v^T$ is the Euler matrix of the ice quiver $(\widetilde{\mu}_v(Q),F)$. One can easily check that the steps (3) and (4) of the extended Fomin--Zelevinsky mutation do not change the Euler matrix, so we conclude the proof.
\end{proof}

\begin{proposition}\label{proposition:compatible pair of an ice quiver respects mutation}
Assume $\det \widehat{B} \neq 0$ and let $(\widetilde{B},\Lambda)$ be the compatible pair associated with the finite ice quiver $(Q,F)$. If $v \in Q_0 \setminus F_0$ is an unfrozen vertex not incident to loops or $2$-cycles, then $(\mu_v(\widetilde{B}),\mu_v(\Lambda))$ is the compatible pair associated with $\mu_v^{\rm{FZ}}(Q,F)$.    
\end{proposition}

\begin{proof}
    By Lemma \ref{lemma:mutation of Euler matrix}, we have that $\mu_v(\widehat{B}) = E_v\widehat{B}E_v^T$ is the Euler matrix of $\mu_v^{\rm{FZ}}(Q,F)$. In particular, $\mu_v(\widehat{B})$ is invertible if $\widehat{B}$ is invertible, and they have the same determinant. We have
    \begin{align*}
    \mu_v(\Lambda) = E_v^T\Lambda E_v &= \abs{\det\widehat{B}} \cdot (E_v^T\widehat{B}^{-T}E_v - E_v^T\widehat{B}^{-1}E_v)\\
    &= \abs{\det\widehat{B}} \cdot ((E_v\widehat{B}E_v^T)^{-T} - (E_v\widehat{B}E_v^T)^{-1})\\
    &= \abs{\det\mu_v(\widehat{B})} \cdot (\mu_v(\widehat{B})^{-T} - \mu_v(\widehat{B})^{-1}),
    \end{align*}
    where we used that $E_v^{-1} = E_v$ in the third equality. This concludes the proof.
\end{proof}

\subsection{Homological interpretation via the Euler form} Let $(Q,F,W)$ be an ice quiver with potential. As before, we label the vertices as $Q_0 = \{1,\dots,m\}$ and $F_0 = \{n+1,\dots,m\}$ for $1 \leq n \leq m$. Denote $\bGh = \bGh(Q,F,W)$. For $P \in \per(\bGh)$ and $M \in \pvd(\bGh)$, the complex $\RHom_{\bGh}(P,M)$ has finite-dimensional total cohomology. Therefore, if $ K_0(\per(\bGh))$ and $K_0(\pvd(\bGh))$ denote the Grothendieck groups of the corresponding triangulated categories, we have a pairing
\[
\langle-,-\rangle: K_0(\per(\bGh)) \times K_0(\pvd(\bGh)) \longrightarrow \Z
\]
defined by
\[
\langle[P],[M]\rangle \coloneqq \sum_{p \in \Z}(-1)^p\dim H^p(\RHom_{\bGh}(P,M)) = \sum_{p \in \Z}(-1)^p\dim\Ext^p_{\bGh}(P,M).
\]
It is called the \emph{Euler form} of $\bGh$. For $i \in Q_0$, let $S_i \in \pvd(\bGh) \subseteq \per(\bGh)$ denote the simple $H^0(\bGh)$-module associated with $i$. By \cite[Lemma 2.15]{KellerYang}, the dimension of $\Ext^p_{\bGh}(S_i,S_j)$ is the number $a_{ij}^p$ of arrows from $j$ to $i$ of degree $-p+1$ in the quiver $\widetilde{Q}$ in the definition of $\bGh$ (or $a_{ij}^p + 1$ if $p = 0$ and $i = j$). Consequently, if $\widehat{B} = (b_{ij})$ is the Euler matrix of $(Q,F)$, then one can easily check that $b_{ij} = \langle [S_i],[S_j]\rangle$. 

Suppose $\bGh$ is proper. In this case, the inclusion $\pvd(\bGh) \subseteq \per(\bGh)$ becomes an equality, and we can compute the Euler form between a pair of perfect dg $\bGh$-modules. This allows us to compute the inverse of the Euler matrix.

\begin{lemma}\label{lemma:inverse of Euler matrix}
    If $\bGh$ is proper, then the associated Euler matrix $\widehat{B}$ is invertible over $\Z$. In this case, the entry $(i,j)$ of $\widehat{B}^{-1}$ is given by $\langle[e_j\bGh],[e_i\bGh]\rangle$ for any $1 \leq i,j \leq m$.
\end{lemma}

\begin{proof}
    Since the objects in $\pvd(\bGh)$ have bounded cohomology, one can show that $K_0(\pvd(\bGh))$ is isomorphic to the Grothendieck group of the category of finite-dimensional $H^0(\bGh)$-modules. Since we are in the completed setting, this abelian group is free of rank $m$ and has as a basis the classes $[S_i]$ of the simple $H^0(\bGh)$-modules. On the other hand, one can use \cite[Lemma 2.14]{PlamondonCharacter} to show that $K_0(\per(\bGh))$ is also free of rank $m$ and has as a basis the classes $[e_i\bGh]$ of the indecomposable projective dg $\bGh$-modules. Under the hypothesis that $\bGh$ is proper, these two Grothendieck groups coincide. Using that $\langle e_i\bGh,S_j\rangle = \delta_{ij}$, we see that the transition matrix from the first basis to the second one is $(\langle [e_i\bGh],[e_j\bGh]\rangle)_{ij}$, while the transition matrix in the other direction is $\widehat{B}^T$. The lemma then follows.
\end{proof}

\Cref{lemma:inverse of Euler matrix} combined with \Cref{compatible pair in inv case} implies that the compatible pair $(\wt{B},\Lambda)$ of $(Q,F)$, constructed in \eqref{definition B tilde submatrix} -- \eqref{Lambda if det invertible}, exists when $\bGh$ is proper. Moreover, if $\Lambda = (\lambda_{ij})$, then
\[
\lambda_{ij} = \langle[e_i\bGh],[e_j\bGh]\rangle - \langle[e_j\bGh],[e_i\bGh]\rangle
\]
for all $1 \leq i,j \leq m$. We will now interpret this formula using the Higgs category $\mc{H} = \mc{H}(Q,F,W)$. Note that $(Q,F,W)$ is Jacobi-finite if we assume $\bGh$ proper. We define
\begin{equation}\label{def H-bracket}
[M,N]_{\mc{H}} \coloneqq \sum_{p \geq 0}(-1)^p(\dim\Ext^{-p}_{\mc{H}}(M,N) - \dim\Ext^{-p}_{\mc{H}}(N,M)).
\end{equation}
for $M,N \in \mc{H}$. The extension groups above are finite-dimensional since the relative cluster category is $\Hom$-finite.

\begin{lemma}\label{lemma:bracket is defined}
    If $\bGh$ is proper, then the sum $[M,N]_{\mc{H}}$ is finite for all $M,N \in \mc{H}$. Moreover, we have
    \[
    [e_i\bGh,e_j\bGh]_{\mc{H}} = \langle[e_i\bGh],[e_j\bGh]\rangle - \langle[e_j\bGh],[e_i\bGh]\rangle
    \]
    for $1 \leq i,j \leq m$.
\end{lemma}

\begin{proof}
The second statement is an immediate consequence of Proposition \ref{proposition:negative extensions can be computed in per Gamma} and the fact that $\bGh$ is connective, that is, concentrated in non-negative degrees. Let us prove the first statement. Assume first that $M \in \add\bGh$. By taking a conflation
\begin{equation}\label{eq:resolution for N in proof of finite sum}
    \begin{tikzcd}
	{T_1} & {T_0} & N & {}
	\arrow[tail, from=1-1, to=1-2]
	\arrow[two heads, from=1-2, to=1-3]
	\arrow[dashed, from=1-3, to=1-4]
    \end{tikzcd}
\end{equation}
with $T_0,T_1 \in \add\bGh$ and applying the cohomological functor $\Hom_{\mc{C}(Q,F,W)}(M,-)$, we get a long exact sequence
\[\begin{tikzcd}[row sep=tiny]
	\dotsb & \Ext^{-1}_{\mc{H}}(M,T_1) & \Ext^{-1}_{\mc{H}}(M,T_0) & \Ext^{-1}_{\mc{H}}(M,N) \\
	{\phantom{\dotsb}} & \Hom_{\mc{H}}(M,T_1) & \Hom_{\mc{H}}(M,T_0) & \Hom_{\mc{H}}(M,N).
	\arrow[from=1-1, to=1-2]
	\arrow[from=1-2, to=1-3]
	\arrow[from=1-3, to=1-4]
	\arrow[from=2-1, to=2-2]
	\arrow[from=2-2, to=2-3]
	\arrow[from=2-3, to=2-4]
\end{tikzcd}\]
By Proposition \ref{proposition:negative extensions can be computed in per Gamma} and because $M \in \add\bGh$, the negative extensions between $M$ and the $T_i$ can be computed in $\per(\bGh)$. Since $\Ext^{-p}_{\bGh}(\bGh,\bGh) \cong H^{-p}(\bGh)$ and $\bGh$ is proper, the terms $\Ext^{-p}_{\mc{H}}(M,T_i)$ in the long exact sequence vanish for $p$ large enough. We conclude that the same holds for $\Ext^{-p}_{\mc{H}}(M,N)$. A similar argument shows that $\Ext^{-p}_{\mc{H}}(N,M)$ also vanishes for $p$ large enough, so we deduce that $[M,N]_{\mc{H}}$ is a finite sum. Finally, the general case can be similarly reduced to the case where $M \in \add\bGh$ by taking a two-term resolution of $M$ as in \eqref{eq:resolution for N in proof of finite sum}.
\end{proof}

Let $(Q,F,W)$ be a non-degenerate ice quiver with potential. We attach an ice quiver with potential $(Q_t,F_t,W_t)$ to any vertex $t \in \mathbb{T}_n$ of the $n$-regular tree by considering the sequence of mutations from the initial root to $t$. By Proposition \ref{proposition:compatible pair of an ice quiver respects mutation}, if the compatible pair associated with $(Q,F)$ is defined, then so is the compatible pair associated with $(Q_t,F_t)$, and the latter can be computed from the former by applying the corresponding sequence of mutations of compatible pairs.

\begin{theorem}\label{thm:canonical quantum structure}
    Let $(Q,F,W)$ be a non-degenerate ice quiver with potential such that $\bGh(Q,F,W)$ is proper. Let $\mc{H} = \mc{H}(Q,F,W)$ be the associated Higgs category. Then the compatible pair associated with $(Q_t,F_t)$ is defined for all $t \in \mathbb{T}_n$ and its Poisson coefficient matrix $\Lambda_t = (\lambda_{ij}^t)$ is given by
    \[
    \lambda_{ij}^t = [M_i^t,M_j^t]_{\mc{H}},
    \]
    where $M_i^t \in \mc{H}$ is the reachable rigid indecomposable object such that $CC(M^t_i)$ is the $i$-th cluster variable $x_{i;t}$ in the seed of $\mc{A}_{Q,F}$ associated with $t$. In particular, this endows $\mc{A}_{Q,F}$ with a $\Uplambda$-cluster algebra structure.
\end{theorem}

\begin{proof}
    Fix $t \in \mathbb{T}_n$. By Proposition \ref{proposition:compatible pair of an ice quiver respects mutation} and Lemma \ref{lemma:inverse of Euler matrix}, the compatible pair associated with $(Q_t,F_t)$ is defined. Moreover, combined with Lemma \ref{lemma:bracket is defined}, we have $\lambda_{ij}^t = [e_i\bGh_t,e_j\bGh_t]_{\mc{H}_t}$, where $\bGh_t = \bGh(Q_t,F_t,W_t)$ and $\mc{H}_t$ is the associated Higgs category. If $\Phi_t: \mc{H}_t \to \mc{H}$ is the  equivalence given by the sequence of mutations from the root of $\mathbb{T}_n$ to $t$, then it sends $e_i\bGh_t$ to $M^t_i$. Since it comes from an equivalence between the corresponding relative cluster categories, $\Phi_t$ preserves negative extensions, so we conclude that $\lambda^t_{ij}$ can be computed as above.
\end{proof}

\begin{remark}\label{rem: Frobenis exact case}
    If $\bGh$ is proper and concentrated in degree zero, then $\mc{H}$ is a Frobenius \emph{exact} category by \cite[Theorem 6.2]{YilinWuJacobifinite}. In particular, the first connecting morphism in the long exact sequence in the proof of Lemma \ref{lemma:bracket is defined} is the zero map. By this observation together with the fact that $\bGh$ has no cohomology in strictly negative degrees, we deduce that
    \[
    [M,N]_{\mc{H}} = \dim\Hom_{\mc{H}}(M,N) - \dim\Hom_{\mc{H}}(N,M)
    \]
    for $M,N \in \mc{H}$. As a corollary, we conclude that the $\Uplambda$-cluster algebra structure on $\mc{A}_{Q,F}$ arising from Theorem \ref{thm:canonical quantum structure} coincides with the one given by \cite[Theorem 6.22]{GrabowskiPressland} (see also \cite{GLSQuantumstructure}). In fact, by including the terms with negative extensions in the long exact sequences in the proof of \cite[Theorem 6.22]{GrabowskiPressland}, their argument can be easily adapted to the case where $\bGh$ is not concentrated in degree zero. This gives an alternative proof that the matrices in Theorem \ref{thm:canonical quantum structure} yield compatible pairs and that they respect the mutation rule for Poisson coefficient matrices.
\end{remark}

As recalled in \Cref{section:review on cluster algebras}, attached to the $\Uplambda$-cluster algebra structure of $\mc{A}_{Q,F}$ yielded by \Cref{thm:canonical quantum structure}, there are Cao's tropical and $F$-invariants, see \eqref{def tropical invariant} and \eqref{def F-invariant}. We conclude this section with a homological interpretation of these invariants. 

\begin{proposition}\label{prop:additive interpretation of the tropical and F invariants}
Let $(Q,F,W)$ be a non-degenerate ice quiver with potential such that $\bGh(Q,F,W)$ is proper. Let $\mc{H} = \mc{H}(Q,F,W)$ be the associated Higgs category. For reachable rigid objects $M,N \in \mc{H}$, we have
\[
\langle CC(M),CC(N)\rangle_{\rm{trop}} = \dim\Ext^1_{\mc{H}}(M,N) + [M,N]_{\mc{H}}
\]
and
\[
(CC(M) \mid\mid CC(N))_F = 2 \cdot\dim\Ext^1_{\mc{H}}(M,N).
\]
\end{proposition}

\begin{proof}
Suppose $CC(M)$ is a cluster monomial for the seed attached to $t \in \mathbb{T}_n$. In particular, the $F$-polynomial $F^t_{CC(M)}$ is constant and equal to $1$. Thus, by the definition of the tropical invariant, we have
\[
\langle CC(M),CC(N)\rangle_{\rm{trop}} = ({\bf g}_{CC(M)}^t)^T\Lambda_t{\bf g}_{CC(N)}^t.
\]
We can compute the $g$-vectors above using Theorem \ref{thm:g-vectors as indices}. By the description of $\Lambda_t$ in Theorem \ref{thm:canonical quantum structure}, we deduce that
\[
\langle CC(M),CC(N)\rangle_{\rm{trop}} = [M,T_0]_{\mc{H}} - [M,T_1]_{\mc{H}},
\]
where $T_0$ and $T_1$ are such that there exists a conflation
\[\begin{tikzcd}
	{T_1} & {T_0} & N & {}
	\arrow[tail, from=1-1, to=1-2]
	\arrow[two heads, from=1-2, to=1-3]
	\arrow[dashed, from=1-3, to=1-4]
\end{tikzcd}\]
in $\mc{H}$ and $T_0,T_1$ belong to the cluster-tilting subcategory $\add T$ associated with $t$. Applying $\Hom_{\mc{C}(Q,F,W)}(M,-)$ and $\Hom_{\mc{C}(Q,F,W)}(-,M)$ to this conflation, we get long exact sequences
\[
\dotsb \rightarrow \Ext^{-1}_{\mc{H}}(M,N)  \rightarrow \Hom_{\mc{H}}(M,T_1)  \rightarrow \Hom_{\mc{H}}(M,T_0)  \rightarrow \Hom_{\mc{H}}(M,N)  \rightarrow 0
\]
and
\begin{align*}
\dotsb \rightarrow \Ext^{-1}_{\mc{H}}(N,M)  \rightarrow \Hom_{\mc{H}}(N,M)  \rightarrow \Hom_{\mc{H}}(T_0,M)  &\rightarrow \Hom_{\mc{H}}(T_1,M)\\
&\rightarrow \Ext^1_{\mc{H}}(N,M)  \rightarrow 0,
\end{align*}
where we used that $\Ext^1_{\mc{H}}(M,T_1) = \Ext^1_{\mc{H}}(T_0,M) = 0$ since $M,T_0,T_1 \in \add T$. Taking the difference of the Euler characteristics of these exact sequences, we deduce that
\[
[M,T_0]_{\mc{H}} - [M,T_1]_{\mc{H}} = \dim\Ext^1_{\mc{H}}(N,M) + [M,N]_{\mc{H}}.
\]
By the $2$-Calabi--Yau property, $\dim\Ext^1_{\mc{H}}(N,M) = \dim\Ext^1_{\mc{H}}(M,N)$, whence the formula for the tropical invariant. Finally, we have
\begin{align*}
    (CC(M) \mid\mid CC(N))_F &= \langle CC(M),CC(N)\rangle_{\rm{trop}} + \langle CC(N),CC(M)\rangle_{\rm{trop}}\\
    &= (\dim\Ext^1_{\mc{H}}(M,N) + [M,N]_{\mc{H}}) + (\dim\Ext^1_{\mc{H}}(N,M) + [N,M]_{\mc{H}})\\
    &= 2 \cdot \dim\Ext^1_{\mc{H}}(M,N)
\end{align*}
again by the $2$-Calabi--Yau property and by the skew-symmetry of $[-,-]_{\mc{H}}$.
\end{proof}

\begin{remark}\label{rem: recovering sym E invariant}
Keep the notation and hypotheses of the proposition above. By Proposition \ref{prop:Ext1 is the same in Higgs and absolute cluster categories}, we have
\[
\Ext^1_{\mc{H}}(M,N) \cong \Ext^1_{\overline{\mc{C}}}(p^*(M),p^*(N)),
\]
where $\ov{\mc{C}} = \mc{C}(\ov{Q},\ov{W})$ and $(\ov{Q},\ov{W})$ is the quiver with potential obtained from $(Q,F,W)$ by deleting the frozen vertices. Moreover, by \cite[Proposition 4.15]{PlamondonCluster} and since $p^*(M)$ and $p^*(N)$ are rigid, the dimension of the extension group above coincides with the $E$-invariant $E^{\rm{sym}}(\mc{M},\mc{N})$ of the corresponding pair of decorated representations $\mc{M}$ and $\mc{N}$ of $(\ov{Q},\ov{W})$ (in the sense of \cite{DerksenWeymanZelevinskyII}). Hence, by Proposition \ref{prop:additive interpretation of the tropical and F invariants}, we have
\[
(CC(M)\mid\mid CC(N))_F = 2 \cdot E^{\rm{sym}}(\mc{M},\mc{N}).
\]
This recovers \cite[Theorem 6.11]{CaoFinvariant} in our setting. We remark that the $F$-invariant above can be computed on the level of the cluster algebra with trivial coefficients $\mc{A}_{\ov{Q}}$ by \cite[Section 4.4]{CaoFinvariant}, which is the setting in the cited theorem. Furthermore, the term $2$ appears because our compatible pairs are of type $2I_n$.
\end{remark}

\section{Background on monoidal categorification}\label{section:background on monoidal categorification}

We provide the necessary background on the monoidal categorification of cluster algebras via quantum affine algebras, following the work of Kashiwara--Kim--Oh--Park \cite{KKOPmonI,KKOPmonII,KKOPmonIII}. After recalling some relevant properties of their $\Lambda$- and $\dd$-invariants, we state their result on $\Uplambda$-monoidal categorification using the category $\mathscr{C}_{\mf{g}}^{[a,b],\mc{D},\un{w}_0}$. We then describe the ice quiver of an initial monoidal seed in $\mathscr{C}_{\mf{g}}^{[a,b],\mc{D},\un{w}_0}$ and upgrade it to an ice quiver with potential that we will use in the subsequent sections to construct the corresponding additive categorification. For a particular case already considered by Hernandez--Leclerc in \cite{HLcluster}, we also present the modules that constitute this initial monoidal seed.

\subsection{Quantum affine algebras and the \texorpdfstring{$\Lambda$}{}-invariant} We follow the conventions of \cite{KKOPmonI,KKOPmonII}, where the reader can find more details and references.

We denote by ${\bf k}$ the algebraic closure of the subfield $\Cbb(q)$ in the algebraically closed field $\bigcup_{m > 0}\Cbb(\!(q^{1/m})\!)$, where $q$ is an indeterminate. Let $\mf{g}$ be an affine Kac--Moody Lie algebra. We denote by $U'_q(\mf{g})$ the associated quantum affine algebra, which is a Hopf algebra over ${\bf k}$. We work with the category $\mathscr{C}_{\mf{g}}$ of finite-dimensional integrable $U'_q(\mf{g})$-modules. It is a ${\bf k}$-linear non-semisimple abelian length category, and the coproduct of $U'_q(\mf{g})$ endows $\mathscr{C}_{\mf{g}}$ with the structure of a monoidal category, which is rigid in the sense that any $V \in \mathscr{C}_{\mf{g}}$ has a left dual $\DD^{-1}(V)$ and a right dual $\DD(V)$. We extend this notation and define $\DD^n(V)$ for any $n \in \Z$ in the natural way.

Although $\mathscr{C}_{\mf{g}}$ is not braided as a monoidal category, any two simple objects $V,W \in \mathscr{C}_{\mf{g}}$ determine a \emph{universal $R$-matrix} (introduced in \cite{DrinfeldICM}), which yields a canonical map $V \ten W \to W \ten V$ that is generically an isomorphism. Using the renormalizing coefficient of $V$ and $W$ coming from this universal $R$-matrix, Kashiwara--Kim--Oh--Park define in \cite{KKOPmonI} the \emph{$\Lambda$-invariant} $\Lambda(V,W) \in \Z$ and the \emph{$\dd$-invariant}
\[
\dd(V,W) = \frac{1}{2}(\Lambda(V,W) + \Lambda(W,V)).
\]
We remark that $\dd(V,W) \in \Z_{\geq 0}$ by Proposition 3.16 and Corollary 3.19 in \cite{KKOPmonI}.

The following result describes the $\Lambda$-invariant in terms of the $\dd$-invariant and the functors $\mathscr{D}^{\pm 1}$. Notice the similarity with the formula for the tropical invariant in Proposition \ref{prop:additive interpretation of the tropical and F invariants}.
\begin{proposition}\label{prop: lambda as sum of delta}
Let $V$ and $W$ be simple modules in $\mathscr{C}_{\mf{g}}$. Then 
\begin{align*}
\Lambda(V,W) & = \dd (V,W) +  \sum_{n \geq 1} (-1)^{n-1} \left[\dd(V, \DD^{-n}(W)) - \dd(V, \DD^n(W)) \right] \\
& = \dd (V,W) +  \sum_{n \geq 1} (-1)^{n-1} \left[\dd(V, \DD^{-n}(W))  - \dd(W, \DD^{-n}(V)) \right]
\end{align*}
\end{proposition}
\begin{proof}
    The first equality is \cite[Proposition 3.22]{KKOPmonI}. For the second we use the symmetry of the $\dd$-invariant and \cite[Lemma 3.7]{KKOPmonI}, which implies that $\dd(V',W') = \dd(\DD^{-1}(V'),\DD^{-1}(W'))$ for any simple objects $V',W' \in \mathscr{C}_{\mf{g}}$. We obtain that $\dd(V, \DD^n(W)) =\dd(W, \DD^{-n}(V))$ and hence the desired equality.
\end{proof}

There is an important class of simple $U_q'(\mf{g})$-modules in $\mathscr{C}_{\mf{g}}$ called the \emph{Kirillov-Reshetikhin} (KR) \emph{modules}. We recall their definition in the case where $\mf{g}$ is \emph{untwisted of simply-laced type}, that is, where it is of type $\mathsf{A}^{(1)}\mathsf{D}^{(1)}\mathsf{E}^{(1)}$ in Kac's classification \cite[Chapter 4]{KacInfDimLieAlgebras}. In this case, let $\Delta$ denote the corresponding Dynkin diagram of \emph{finite} type $\mathsf{ADE}$ obtained by removing an extending vertex from the Dynkin diagram of $\mf{g}$. Consider the infinite set of variables $\mathscr{Y} = \{Y_{i,a} \mid i \in \Delta_0, a \in {\bf k}^{\times}\}$. By the work of Chari--Pressley (see \cite[Theorem 12.2.6]{ChariPressleyGuideQuantumGroups}) and Frenkel--Reshetikhin \cite{FreRes99}, the simple objects in $\mathscr{C}_{\mf{g}}$ can be parametrized by the highest dominant monomial of their $q$-characters, which is a monomial with non-negative exponents in the variables of $\mathscr{Y}$. For such a monomial $m$, denote by $L(m)$ the corresponding simple module. For $i \in \Delta_0$, $a \in {\bf k}^{\times}$ and $r \geq 1$, the associated KR-module is defined as
\[
W^{(i)}_{r,a} = L(Y_{i,a}Y_{i,aq^2}\dotsb Y_{i,aq^{2(r-1)}}).
\]
We have an explicit description of the duality functors on such modules. Denote by $\mathbb{D}: \mathscr{Y} \to \mathscr{Y}$ the bijection sending $Y_{i,a}$ to $Y_{i^*,aq^h}$, where $i \mapsto i^*$ is the involution on $\Delta_0$ induced by the longest element of the Weyl group of $\Delta$ and $h$ is the Coxeter number of $\Delta$. We extend $\mathbb{D}$ to a bijection on the set of monomials in the variables of $\mathscr{Y}$. By \cite[Proposition 5.5]{FujitaOh} (and the references therein), we have $\DD^{n}(L(m)) \cong L(\mathbb{D}^n(m))$ for any dominant monomial $m$ in the variables of $\mathscr{Y}$ and $n \in \Z$. In particular, observe that $\DD^{-1}(W^{(i)}_{r,a}) \cong W^{(i^*)}_{r,aq^{-h}}$.

\subsection{\texorpdfstring{$\Uplambda$}{}-monoidal categorification}\label{Section: background quiver and cat}

Let $\mc{C}$ be a full subcategory of $\mathscr{C}_{\mf{g}}$ that contains the trivial module and is stable under tensor products, subquotients, and extensions. Denote by $K_0(\mc{C})$ the Grothendieck group of $\mc{C}$. It becomes a ring under the product induced by the monoidal structure. We say $\mc{C}$ is a \emph{monoidal categorification} of a cluster algebra $\mc{A}$ if it is endowed with a $\Z$-algebra isomorphism
\[
\varphi: K_0(\mc{C}) \longrightarrow \mc{A}
\]
such that any cluster monomial in $\mc{A}$ is the image of the class of a simple module. In this case, a simple object $V \in \mc{C}$ is \emph{reachable} if $\varphi([V])$ is a cluster monomial. We stress that $\varphi$ is part of the data of the monoidal categorification. For simplicity, we will denote $\varphi(V) \coloneqq \varphi([V])$ for $V \in \mathscr{C}_{\mf{g}}$. We allow the cluster algebra to have infinite rank and denote its set of indices by $K = K^{\rm{un}} \sqcup K^{\rm{fr}}$.

For a seed $\mc{S} = ({\bf x}_t, \widetilde{B}_t)$ in $\mc{A}$, write ${\bf x}_t = \{x_{i;t}\}_{i \in K}$. It determines (up to isomorphism) a set $\mc{V}_t = \{V_{i;t}\}_{i \in K}$ of simple modules in $\mc{C}$ defined by $\varphi(V_{i;t}) = x_{i;t}$. Note that $V_{i;t}$ and $V_{j;t}$ \emph{strongly commute} for any $i,j \in K$ in the sense that their tensor product is simple (since $\varphi(V_{i;t} \ten V_{j;t})$ is a cluster monomial). In particular, each $V_{i;t}$ is a \emph{real} simple module, that is, its tensor square is simple. We call the pair $\mathscr{S} = (\mc{V}_t,\widetilde{B}_t)$ a \emph{(reachable) monoidal seed} in $\mc{C}$. For $v \in K^{\rm{un}}$, we denote by $\mu_v(\mathscr{S})$ the monoidal seed that corresponds to the mutation of $\mc{S}$ in the direction of $v$. More generally, we define a monoidal seed in $\mc{C}$ as a pair $(\mc{V},\wt{B})$ where $\wt{B}$ is a $K \times K^{\rm{un}}$ extended exchange matrix and $\mc{V}$ is a strongly commuting family of real simple modules in $\mc{C}$ indexed by $K$.

Any monoidal seed $\mathscr{S} = (\mc{V}_t,\wt{B}_t)$ in $\mc{C}$ determines a square matrix $\Lambda_t = (\lambda^t_{ij})_{i,j \in K}$ given by the $\Lambda$-invariant
\[
\lambda_{ij}^t = \Lambda(V_{i;t},V_{j;t}).
\]
Since $\mc{V}_t$ is a strongly commuting family of real simple modules, we have $\dd(V_{i;t},V_{j;t}) = 0$ by \cite[Corollary 3.17]{KKOPmonI} and so $\Lambda_t$ is skew-symmetric. We say that the monoidal categorification is a \emph{$\Uplambda$-monoidal categorification} if the pair $(\widetilde{B}_t,\Lambda_t)$ is compatible for some reachable monoidal seed. In this case, by \cite[Proposition 6.4]{KKOPmonI} (see also \cite[Proposition 5.13]{CaoFinvariant}), this compatibility condition holds for \emph{any} reachable monoidal seed and, if we denote $\mu_v(\mathscr{S}) = (\mc{V}_{t'},\wt{B}_{t'})$ and the corresponding skew-symmetric matrix by $\Lambda_{t'}$, then $(\wt{B}_{t},\Lambda_t)$ and $(\wt{B}_{t'},\Lambda_{t'})$ are related by mutation of compatible pairs in the direction of $v \in K^{\rm{un}}$. Consequently, $\mc{A}$ is a $\Uplambda$-cluster algebra.

\begin{remark}
    The notion of monoidal categorification was first introduced in \cite{HLduke}, while that of $\Uplambda$-monoidal categorification in \cite{KKOPmonI} (building on \cite{KKKOjams}). Observe that the definition in \cite{KKOPmonI} is slightly more restrictive than the one given above, which is based on \cite[Section 5]{CaoFinvariant}.
\end{remark}

The following is one of the main results in \cite{CaoFinvariant} and a major motivation for the definition of the tropical and $F$-invariants.

\begin{theorem}[{\cite[Corollary 5.17]{CaoFinvariant}}]\label{thm:Lambda invariant as a tropical invariant}
    Let $\mc{C}$ be a $\Uplambda$-monoidal categorification of a $\Uplambda$-cluster algebra $\mc{A}$. Let $\varphi: K_0(\mc{C}) \to \mc{A}$ be the associated isomorphism. For reachable simple objects $V,W \in \mc{C}$, we have
    \[
    \langle\varphi(V),\varphi(W)\rangle_{\rm{trop}} = \Lambda(V,W)
    \]
    and
    \[
    (\varphi(V)\mid\mid\varphi(W))_F = 2 \cdot \dd(V,W).
    \]
\end{theorem}

Let us now briefly recall the $\Uplambda$-monoidal categorification obtained in \cite{KKOPmonII,KKOPmonIII}. By \cite[Theorem 3.6]{KKOPSimplyLacedRootSystems}, one can associate a simply-laced finite type Dynkin diagram $\Delta$ to $U'_q(\mf{g})$. If $\mf{g}$ is untwisted of simply-laced type, then $\Delta$ is the corresponding Dynkin diagram of finite type. In \cite{KKOP_PBWQuantumAffine}, the notion of \emph{complete duality datum} for $U'_q(\mf{g})$ is introduced, which depends on $\Delta$ and is intimately connected to the quantum affine Schur--Weyl duality of \cite{KKKinv}. A \emph{complete PBW-pair} is a pair $(\mc{D},\un{w}_0)$ where $\mc{D}$ is a complete duality datum for $U'_q(\mf{g})$ and $\un{w}_0 = (i_1,\dots,i_{l(w_0)}) \in \Delta^{l(w_0)}$ is a reduced expression of the longest element $w_0$ of the Weyl group of type $\Delta$. By \cite{KKOP_PBWQuantumAffine}, any such pair determines a family $\{\mathsf{S}_s\}_{s \in \Z}$ of simple modules in $\mathscr{C}_{\mf{g}}$ called the \emph{affine cuspidal modules}.

For $a,b \in \Z \cup \{\pm\infty\}$ such that $a \leq b$, we denote
\[
[a,b] = \{s \in \Z \mid a \leq s \leq b\}.
\]
Given such an interval and a complete PBW-pair $(\mc{D},\un{w}_0)$, Kashiwara--Kim--Oh--Park \cite{KKOPmonII} define $\mathscr{C}_{\mf{g}}^{[a,b],\mc{D},\un{w}_0}$ as the smallest full subcategory of $\mathscr{C}_{\mf{g}}$ that is closed under taking tensor products, subquotients and extensions, and contains the trivial module and the affine cuspidal modules $\mathsf{S}_s$ associated with $(\mc{D},\un{w}_0)$ for $s \in [a,b]$.

\begin{theorem}[{\cite{KKOPmonII,KKOPmonIII}}]\label{thm:KKOP monoidal categorification}
    The Grothendieck ring of $\mathscr{C}_{\mf{g}}^{[a,b],\mc{D},\un{w}_0}$ has a $\Uplambda$-cluster algebra structure and $\mathscr{C}_{\mf{g}}^{[a,b],\mc{D},\un{w}_0}$ is a $\Uplambda$-monoidal categorification of this $\Uplambda$-cluster algebra.
\end{theorem}

We describe an initial monoidal seed for the $\Uplambda$-monoidal categorification above in the next section.

\subsection{Descriptions for the initial monoidal seed}\label{section:initial monoidal seed} From now on, $\Delta$ denotes the simply-laced finite type Dynkin diagram associated to $U_q'(\mf{g})$. For two vertices $i,j \in \Delta_0$, we write $i \sim j$ to denote that $i$ and $j$ are adjacent in $\Delta$. Let $w_0$ the longest element of the Weyl group of type $\Delta$ and denote by $i \mapsto i^*$ the involution on $\Delta_0$ defined by $w_0(\alpha_i) = -\alpha_{i^*}$, where $\alpha_i$ is the simple root associated with $i \in \Delta_0$. For a reduced expression $\un{w}_0 = (i_1,\dots,i_{l(w_0)})$ for $w_0$, we define an infinite sequence $\widehat{\un{w}}_0 = (i_s)_{s \in \Z}$ extending $\un{w}_0$ and such that $i_{s+l(w_0)} = i_s^*$ for all $s \in \Z$. We define
\[
s^- = \max\{t \in \Z \mid t<s, i_t=i_s\} \quad \textrm{and} \quad s^+ = \min\{t \in \Z \mid t>s, i_t=i_s\}
\]
for $s \in \Z$. We now use the sequence $\widehat{\un{w}}_0$ to attach an ice quiver (with potential) to any integer interval bounded on the right.

Fix $a \in \Z \cup \{-\infty\}$ and $b \in \Z$ with $a \leq b$. We define a quiver $Q^{[a,b]}(\underline{w}_0)$ as follows. Its vertex set is the interval $[a,b]$ and there is an arrow $s \to t$ if and only if one of the conditions below holds:
\begin{enumerate}[(1)]
    \item $t = s^-$;
    \item $i_s \sim i_t$ and $s^- < t^- < s < t$;
    \item $i_s \sim i_t$, $s < t$ and $s^-,t^- < a$.
\end{enumerate}
We declare a vertex $s$ to be frozen if $s^- < a$, and we denote by $F^{[a,b]}(\underline{w}_0)$ the full subquiver on frozen vertices. This yields an ice quiver $(Q^{[a,b]}(\un{w}_0),F^{[a,b]}(\un{w}_0))$. Note that $F^{[a,b]}(\underline{w}_0)$ is not empty precisely when $a \in \Z$, and an arrow $s \to t$ is frozen if and only if condition (3) holds. In this case, we also define a potential $W^{[a,b]}(\underline{w}_0)$ on $Q^{[a,b]}(\underline{w}_0)$ as the sum of all simple cordless cycles. More precisely, for any arrow $s \to t$ with $s < t$ and $t$ unfrozen, there is a unique cycle in $Q^{[a,b]}(\underline{w}_0)$ of the form
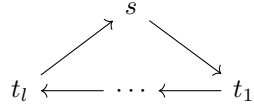
\[\begin{tikzcd}
	& s & \\
	{t_l} & \cdots & t_1
	\arrow[from=1-2, to=2-3]
	\arrow[from=2-1, to=1-2]
	\arrow[from=2-2, to=2-1]
	\arrow[from=2-3, to=2-2]
\end{tikzcd}\]
where $t_1 = t$ and $t_{i+1} = t_i^-$ for $1 \leq i < l$.  Then $W^{[a,b]}(\underline{w}_0)$ is the sum of all such cycles.

\begin{remark}
When drawing the ice quiver $(Q^{[a,b]}(\underline{w}_0), F^{[a,b]}(\underline{w}_0))$, we use the following conventions. We draw the vertices on horizontal lines indexed by $\Delta_0$. We place a vertex $s$ on the line $i_s$ and, if $s < t$, then $s$ is to the left of $t$. Thus, the only horizontal arrows in the picture are those of the form $s \to s^-$, which point to the left. We depict a frozen vertex inside a blue box, and frozen arrows are also blue.
\end{remark}

\begin{example}
    For $\Delta = \mathsf{A}_3$ and $\underline{w}_0 = (1,2,3,2,1,2)$, we have
    \[
    \widehat{\underline{w}}_0 = (\dots,3,2,1,\underbrace{2,3,2,1,2,3,2,1,2}_{[-2,6]},3,2,1,2,3,2,\dots).
    \]
    The ice quiver $(Q^{[-2,6]}(\underline{w}_0),F^{[-2,6]}(\underline{w}_0))$ is given by
    
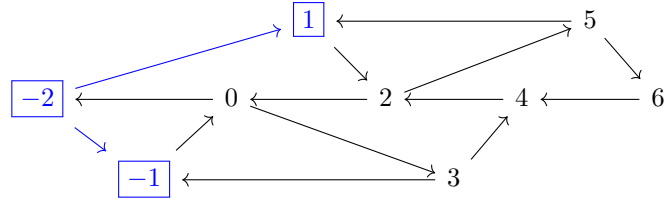
\[\begin{tikzcd}[sep=small]
    &&& \color{blue}\boxed{1} &&&& 5\\
	\color{blue}\boxed{-2} && 0 && 2 && 4 && 6 \\
	& \color{blue}\boxed{-1} &&&& 3 &&&
	\arrow[from=3-2, to=2-3]
	\arrow[from=3-6, to=3-2]
	\arrow[from=3-6, to=2-7]
	\arrow[color=blue, from=2-1, to=3-2]
	\arrow[color=blue, from=2-1, to=1-4]
	\arrow[from=2-3, to=3-6]
	\arrow[from=2-3, to=2-1]
	\arrow[from=2-5, to=2-3]
	\arrow[from=2-5, to=1-8]
	\arrow[from=2-7, to=2-5]
	\arrow[from=2-9, to=2-7]
	\arrow[from=1-4, to=2-5]
	\arrow[from=1-8, to=2-9]
	\arrow[from=1-8, to=1-4]
    \end{tikzcd}\]
    The potential $W^{[-2,6]}(\underline{w}_0)$ is the sum of the six simple cordless cycles.
\end{example}

By the next result, the ice quiver above describes the $\Uplambda$-cluster algebra structure from Theorem \ref{thm:KKOP monoidal categorification} for right-bounded intervals.

\begin{proposition}[{\cite{KKOPmonII,KKOPmonIII}}]\label{prop:quiver of initial monoidal seed of KKOP}
    Let $(\mc{D},\un{w}_0)$ be a complete PBW-pair. Let $a \in \Z \cup \{-\infty\}$ and $b \in \Z$ be such that $a \leq b$. Then $K_0(\mathscr{C}_{\mf{g}}^{[a,b],\mc{D},\un{w}_0})$ has a $\Uplambda$-seed whose extended exchange matrix is the one associated to the ice quiver $(Q^{[a,b]}(\un{w}_0),F^{[a,b]}(\un{w}_0))$.
\end{proposition}

\begin{remark}
    Our quivers are not exactly the ones appearing in \cite{KKOPmonII,KKOPmonIII}. Indeed, they do not impose the condition (3) in our definition of $Q^{[a,b]}(\underline{w}_0)$. Therefore, our quiver may have additional arrows, but they are all between frozen vertices and do not influence the extended exchange matrix in the proposition above. We remark that ice quivers with potential similar to ours have appeared in \cite{BuanIyamaReitenScott} and \cite{BuanIyamaReitenSmith}. On the other hand, they consider a signed version of the potential we consider, and so does \cite{ContuMonoidalAdditive}. We adopt the same signs as in \cite{HLcluster}.
\end{remark}

The simple modules in the monoidal seed that corresponds to the $\Uplambda$-seed in Proposition \ref{prop:quiver of initial monoidal seed of KKOP} are examples of \emph{affine determinantial modules}. We give a more explicit description of this monoidal seed in the case where $\mf{g}$ is untwisted of simply-laced type and $(\mc{D},\un{w}_0)$ is an adapted PBW-pair. In this case, these simple modules will be KR-modules. The general setting is treated in \cite{KKOPmonII}.

We first need to recall some definitions. Let $Q$ be a quiver. We say a sequence ${\bf v} = (v_1,\dots,v_r)$ of vertices of $Q$ is a \emph{source} (resp. \emph{sink}) sequence if $v_i$ is a source (resp. sink) in $\mu_{v_{i-1}}\dotsb\mu_{v_2}\mu_{v_1}(Q)$ for all $1 \leq i \leq r$. Alternatively, we say that ${\bf v}$ is \emph{adapted} to $Q$ if it is a source sequence.

Fix $\epsilon: \Delta_0 \to \{0,1\}$ such that $\epsilon_i \neq \epsilon_j$ if $i \sim j$, where $\epsilon_i \coloneqq \epsilon(i)$. A \emph{height function} is a function $\xi: \Delta_0 \to \Z$ such that $\xi_i \coloneqq \xi(i)$ has the same parity as $\epsilon_i$ and $|\xi_i - \xi_j| = 1$ for $i \sim j$. It determines an orientation $Q_{\xi}$ of $\Delta$ by defining an arrow $i \to j$ if $\xi_j = \xi_i + 1$. We say $i \in \Delta_0$ is a \emph{source} (resp. \emph{sink}) of $\xi$ if and only if $i$ is a source (resp. sink) of $Q_{\xi}$. If $i$ is a source of $\xi$, we define the \emph{reflection} $s_i\xi: \Delta_0 \to \Z$ of $\xi$ at $i$ to be the function such that $(s_i\xi)_j = \xi_j$ for $j \neq i$ and $(s_i\xi)_i = \xi_i + 2$. Note that $s_i\xi$ is indeed a height function and $Q_{s_i\xi}$ is the reflection $\mu_i(Q_{\xi})$ of $Q_{\xi}$ at the source $i$. Similarly, if $i$ is a sink of $\xi$, we define the \emph{reflection} of $\xi$ at $i$ as a height function $s_i^{-1}\xi: \Delta_0 \to \Z$ given by $(s_i^{-1}\xi)_j = \xi_j$ for $j \neq i$ and $(s_i^{-1}\xi)_i = \xi_i - 2$.

\begin{remark}
    We follow a convention opposite to that of \cite{FujitaOh} and \cite{KKOPmonII} to pass between height functions and quivers. In particular, a source (resp. sink) of a height function in our sense is the same as a sink (resp. source) in their sense.
\end{remark}

Let $Q_{\rm{HL}}$ be the quiver whose vertex set is the subset of $\Delta_0 \times \Z$ of pairs $(i,p)$ such that $p$ has the same parity as $\epsilon_i$. We define an arrow from $(i,p)$ to $(j,r)$ if
\begin{enumerate}[(1)]
    \item $i \sim j$ and $r = p+1$;
    \item $i = j$ and $r = p-2$.
\end{enumerate}
For a height function $\xi$, we let $Q_{\rm{HL}}^{<\xi}$ be the full subquiver of $Q_{\rm{HL}}$ on the vertices $(i,p)$ such that $p < \xi_i$.

\begin{lemma}\label{lemma:identification of HL with GLS quivers}
    Let $\xi: \Delta_0 \to \Z$ be a height function. Let $\un{w}_0 = (i_1,\dots,i_{l(w_0)})$ be a reduced expression of the longest element that is adapted to $Q_{\xi}$. Then there is a unique isomorphism $Q_{\rm{HL}}^{<\xi} \to Q^{[-\infty,0]}(\un{w}_0)$ sending a vertex of the form $(i,p)$ to a vertex in the $i$-th row of $Q^{[-\infty,0]}(\un{w}_0)$.
\end{lemma}

\begin{proof}
    It follows from \cite[Proposition 7.27]{KKOPmonII} and the description of admissible sequences in \cite[Proposition 6.11]{KKOPmonII}.
\end{proof}

\begin{remark}\label{rmk:reflecting height function gives all cases}
    Note that any quiver $Q^{[-\infty,b]}(\un{w}_0)$ is isomorphic to $Q^{[-\infty,0]}(\un{w}_0')$ for some reduced expression $\un{w}_0'$. To obtain $\un{w}_0'$, one can shift the sequence $\widehat{\un{w}}_0$ by $b$ units and take the corresponding reduced expression. This shift operation corresponds to applying reflections to $\xi$ in the lemma above.
\end{remark}

For the rest of this section, we assume $\mf{g}$ is untwisted of simply-laced type. By \cite[Theorem 6.12]{KKOPmonII}, a height function $\xi$ determines a complete duality datum $\mc{D}_{\xi}$ for $U_q'(\mf{g})$. We will say that a complete PBW-pair $(\mc{D},\un{w}_0)$ is \emph{adapted to an orientation of $\Delta$} if there exists a height function $\xi: \Delta_0 \to \Z$ such that $\mc{D} = \mc{D}_{\xi}$ and $\un{w}_0$ is adapted to $Q_{\xi}$. In this case, we describe a monoidal seed of $\mathscr{C}_{\mf{g}}^{[-\infty,0],\mc{D},\un{w}_0}$.

For a vertex $(i,p)$ in $Q_{\rm{HL}}^{<\xi}$, we attach the KR-module $\mc{M}^{<\xi}_{(i,p)} = W^{(i)}_{r,q^p}$, where $r \geq 1$ is the largest integer such that $p + 2(r-1) < \xi_i$. This gives a triple $\mathscr{S}^{<\xi} = (\{\mc{M}^{<\xi}_{(i,p)}\},Q_{\rm{HL}}^{<\xi},\varnothing)$, where the set $\{\mc{M}^{<\xi}_{(i,p)}\}$ is indexed by the vertices of $Q_{\rm{HL}}^{<\xi}$. The empty set in the triple represents the set of frozen vertices of $Q_{\rm{HL}}^{<\xi}$. In the next result, we will view $\mathscr{S}^{<\xi}$ as a monoidal seed whose extended exchange matrix comes from the ice quiver $(Q_{\rm{HL}}^{<\xi},\varnothing)$.

\begin{theorem}[{\cite[Section 8]{KKOPmonII}}]\label{thm:initial monoidal seed for HL category}
    Suppose $\mf{g}$ is untwisted of simply-laced type. Let $\xi: \Delta_0 \to \Z$ be a height function. Let $\mc{D}_{\xi}$ be the corresponding complete duality datum and choose a reduced expression $\un{w}_0$ for the longest element adapted to $Q_{\xi}$. Then $\mathscr{S}^{<\xi}$ is a reachable monoidal seed in $\mathscr{C}_{\mf{g}}^{[-\infty,0],\mc{D}_{\xi},\un{w}_0}$ for the $\Uplambda$-monoidal categorification in Theorem \ref{thm:KKOP monoidal categorification}.
\end{theorem}

\begin{remark}
    Let $\xi: \Delta_0 \to \Z$ be the unique height function that only assumes the values $1$ and $2$. As explained in \cite[p. 894]{KKOPmonII}, the category $\mathscr{C}_{\mf{g}}^{[-\infty,0],\mc{D}_{\xi},\un{w}_0}$ above coincides with the category $\mathscr{C}^-$ in \cite{HLcluster}. Moreover, the assignment of KR-modules in the monoidal seed $\mathscr{S}^{<\xi}$ is the one given previously by Hernandez--Leclerc.
\end{remark}

\begin{remark}\label{rmk:initial monoidal seed of finite interval}
    For a height function $\xi: \Delta_0 \to \Z$, let $\mc{D}_{\xi}$ and $\un{w}_0$ be as in Theorem \ref{thm:initial monoidal seed for HL category}. For an integer $a \leq 0$, a reachable monoidal seed for $\mathscr{C}_{\mf{g}}^{[a,0],\mc{D}_{\xi},\un{w}_0}$ can be obtained by restricting $\mathscr{S}^{<\xi}$ to the subquiver that corresponds to $Q^{[a,0]}(\un{w}_0)$ under the isomorphism of Lemma \ref{lemma:identification of HL with GLS quivers} and freezing the vertices that correspond to $F^{[a,0]}(\un{w}_0)$. This follows from \cite[Lemma 7.15]{KKOPmonII} (see also the proof of their Proposition 8.11). Moreover, by applying reflections to $\xi$ as in Remark \ref{rmk:reflecting height function gives all cases}, we can replace the right boundary of the interval in the theorem and in this remark by any $b \in \Z$.
\end{remark}

\section{Maximal green sequences}\label{section:maximal green sequences}

When $\mf{g}$ is untwisted of simply-laced type, we will explain how to realize the left duality functor $\DD^{-1}$ on KR-modules over $U'_q(\mf{g})$ using an infinite version of a maximal green sequence. After recalling the definition and some properties of maximal green sequences, we show how to construct them for triangle products of Dynkin quivers following \cite{GenzKoshevoy}. We then exploit an idea of \cite{HLcluster} to obtain the desired realization of $\DD^{-1}$.

\subsection{Definition and effect on additive categorification} We recall the definition of (maximal) green sequences, which was first given in \cite{KellerQuantumDilog}. Let $Q$ be a finite quiver without loops or $2$-cycles. Its \emph{framed quiver} $\widehat{Q}$ (resp. \emph{coframed quiver} $\widecheck{Q}$) is obtained from $Q$ by adding a vertex $v'$ for any $v \in Q_0$ and an arrow $v \to v'$ (resp. $v' \to v$). The new vertices are declared to be frozen. If $R$ is a quiver obtained from $\widehat{Q}$ by a sequence of mutations (at unfrozen vertices), then we say that an unfrozen vertex $v$ is \emph{green} in $R$ if there is no arrow from frozen vertices to $v$, and \emph{red} otherwise. A sequence ${\bf v} = (v_1,\dots,v_r)$ of vertices in $Q$ is called a \emph{green sequence} if $v_i$ is green in $\mu_{v_{i-1}}\dotsb\mu_{v_2}\mu_{v_1}(\widehat{Q})$ for all $1 \leq i \leq r$. Such a sequence is \emph{maximal} if all unfrozen vertices of $\mu_{\bf v}(\widehat{Q}) \coloneqq \mu_{v_r} \dotsb\mu_{v_2}\mu_{v_1}(\widehat{Q})$ are red.

\begin{proposition}[{\cite[Proposition 2.10]{BrustleDupontPerotin}}]
    Suppose $Q$ admits a maximal green sequence ${\bf v}$. Then there is a unique isomorphism from $\mu_{\bf v}(\widehat{Q})$ to $\widecheck{Q}$ that fixes the frozen vertices. It sends $i \in Q_0$ to $\sigma_{\bf v}(i) \in Q_0$ for a unique permutation $\sigma_{\bf v}$ of $Q_0$.
\end{proposition}

The next results provide a way to construct maximal green sequences that are either source or sink sequences.

\begin{lemma}[{\cite[Lemma 2.20]{BrustleDupontPerotin}}]\label{lemma:source maximal green sequence}
    Let $Q$ be an acyclic quiver with $n$ vertices. If ${\bf v} = (v_1,\dots,v_n)$ is a source sequence for $Q$ such that $Q_0 = \{v_1,\dots,v_n\}$, then ${\bf v}$ is a maximal green sequence and $\sigma_{\bf v} = \id$.
\end{lemma}

\begin{proposition}\label{proposition:sink maximal green sequence}
    Suppose $Q$ is an orientation of an $\mathsf{ADE}$ Dynkin diagram $\Delta$. Let ${\bf v} = (v_1,\dots,v_{l(w_0)})$ be a sink sequence for $Q$ that is also a reduced expression of the longest element $w_0$ of $\Delta$. Then ${\bf v}$ is a maximal green sequence and $\sigma_{\bf v}$ coincides with the involution induced by $w_0$. 
\end{proposition}

\begin{proof}
    A proof that ${\bf v}$ is a maximal green sequence is given in the proof of \cite[Theorem 4.4]{BrustleDupontPerotin}. Their argument also shows that $\sigma_{\bf v}$ agrees with the permutation on the orbits of the Auslander--Reiten translation induced by the shift functor in the bounded derived category of $Q$, which is well-known to coincide with the involution induced by $w_0$.
\end{proof}

Now let $(Q,W)$ be a Jacobi-finite and non-degenerate quiver with potential. Denote by $\bGh = \bGh(Q,W)$ the corresponding completed (absolute) Ginzburg dg algebra. For a sequence of vertices ${\bf v} = (v_1,\dots,v_r)$ of $Q$, we denote by $(\mu_{\bf v}(Q),\mu_{\bf v}(W))$ the quiver with potential obtained by mutation following ${\bf v}$. We write $\mu_{\bf v}(\bGh)$ for $\bGh(\mu_{\bf v}(Q),\mu_{\bf v}(W))$ and we let
\[
\Phi_{\bf v}: \mc{C}(\mu_{\bf v}(Q),\mu_{\bf v}(W)) \to \mc{C}(Q,W)
\]
be the composition of equivalences $\Phi_{v_1}\dotsb\Phi_{v_r}$ as in Section \ref{section:relative cluster and Higgs categories}. The following result describes the action of $\Phi_{\bf v}$ when ${\bf v}$ is a maximal green sequence.

\begin{theorem}[{\cite{KellerQuantumDilog},\cite[Section 7]{KellerClusterDerivedSurvey}}]\label{thm:maximal green sequence gives inverse shift}
    Suppose $Q$ admits a maximal green sequence ${\bf v}$. Then the equivalence $\Phi_{\bf v}$ sends $e_i\mu_{\bf v}(\bGh)$ to $\Sigma^{-1}e_{\sigma_{\bf v}(i)}\bGh$. In particular, $\Phi_{\bf v}(\mu_{\bf v}(\bGh)) \cong \Sigma^{-1}\bGh$.
\end{theorem}

\subsection{Triangle products of quivers}\label{section:triangle products and green sequences} Let $Q$ and $Q'$ be two quivers. We define their \emph{tensor product} $Q \otimes Q'$ to be the quiver with vertex set $Q_0 \times Q'_0$ and with the following arrows. For any arrow $\alpha: i \to j$ in $Q$ and any $i' \in Q'_0$, there is an arrow $(\alpha,i'):(i,i') \to (j,i')$ in $Q \otimes Q'$, and similarly, for any arrow $\alpha':i' \to j'$ in $Q'$ and any $i \in Q_0$, there is an arrow $(i,\alpha'):(i,i') \to (i,j')$ in $Q \otimes Q'$. We also define the \emph{triangle product} $Q \boxtimes Q'$ as the quiver obtained from $Q \otimes Q'$ by adding an arrow $(\alpha,\alpha')^{\mathrm{op}}:(j,j') \to (i,i')$ for all pairs of arrows $\alpha: i \to j$ in $Q$ and $\alpha': i' \to j'$ in $Q'$.

Following \cite{GenzKoshevoy}, we provide a method for constructing maximal green sequences for $Q \boxtimes Q'$. Let ${\bf v} = (v_1,\dots,v_r)$ and ${\bf w} = (w_1,\dots,w_s)$ be sequences of vertices in $Q$ and $Q'$, respectively. We define ${\bf v} \boxtimes {\bf w}$ as the sequence
\[
((v_1,w_1),(v_1,w_2),\dots,(v_1,w_s),(v_2,w_1),\dots,(v_2,w_s),\dots\dots,(v_r,w_1),\dots,(v_r,w_s)),
\]
which is a sequence of vertices in $Q \boxtimes Q'$.

\begin{theorem}[{\cite[Theorem 3.8]{GenzKoshevoy}}]\label{thm:maximal green sequence for triangle product}
    Let $Q$ and $Q'$ be finite quivers. Suppose that $Q$ admits a sink maximal green sequence ${\bf v}$ and $Q'$ admits a source maximal green sequence ${\bf w}$. Then ${\bf v} \boxtimes {\bf w}$ is a maximal green sequence for $Q \boxtimes Q'$ and $\sigma_{{\bf v} \boxtimes {\bf w}} = \sigma_{\bf v} \times \sigma_{\bf w} = \sigma_{\bf v} \times \id$.
\end{theorem}

In particular, by Lemma \ref{lemma:source maximal green sequence} and Proposition \ref{proposition:sink maximal green sequence}, we can apply the theorem above when $Q$ is a Dynkin quiver and $Q'$ is acyclic. It is conjectured in \cite{GenzKoshevoy} that this is precisely the case for which the result above applies.

\begin{example}
    Consider the quivers
    \[
    Q = \begin{tikzcd}
	1 & 2 & 3
	\arrow[from=1-1, to=1-2]
	\arrow[from=1-2, to=1-3]
\end{tikzcd} \quad \textrm{and} \quad Q' = \begin{tikzcd}
	1 & 2 & 3 & 4.
	\arrow[from=1-2, to=1-1]
	\arrow[from=1-3, to=1-2]
    \arrow[from=1-4, to=1-3]
\end{tikzcd}
    \]
The triangle product $Q \boxtimes Q'$ is given as follows:
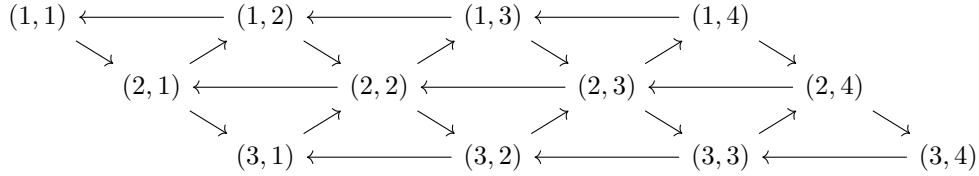
\[\begin{tikzcd}[sep=small]
	{(1,1)} && {(1,2)} && {(1,3)} && {(1,4)} && \\
	& {(2,1)} && {(2,2)} && {(2,3)} && {(2,4)} \\
	&& {(3,1)} && {(3,2)} && {(3,3)} && {(3,4)}
	\arrow[from=1-1, to=2-2]
	\arrow[from=1-3, to=1-1]
	\arrow[from=1-3, to=2-4]
	\arrow[from=1-5, to=1-3]
	\arrow[from=1-5, to=2-6]
	\arrow[from=1-7, to=1-5]
	\arrow[from=1-7, to=2-8]
	\arrow[from=2-2, to=1-3]
	\arrow[from=2-2, to=3-3]
	\arrow[from=2-4, to=1-5]
	\arrow[from=2-4, to=2-2]
	\arrow[from=2-4, to=3-5]
	\arrow[from=2-6, to=1-7]
	\arrow[from=2-6, to=2-4]
	\arrow[from=2-6, to=3-7]
	\arrow[from=2-8, to=2-6]
	\arrow[from=2-8, to=3-9]
	\arrow[from=3-3, to=2-4]
	\arrow[from=3-5, to=2-6]
	\arrow[from=3-5, to=3-3]
	\arrow[from=3-7, to=2-8]
	\arrow[from=3-7, to=3-5]
	\arrow[from=3-9, to=3-7]
\end{tikzcd}\]
\end{example}
By Proposition \ref{proposition:sink maximal green sequence}, ${\bf v} = (3,2,1,3,2,3)$ is a sink maximal green sequence for $Q$. By Lemma \ref{lemma:source maximal green sequence}, ${\bf w} = (4,3,2,1)$ is a source maximal green sequence for $Q'$. By Theorem \ref{thm:maximal green sequence for triangle product}, the sequence
\begin{align*}
    {\bf v} \boxtimes {\bf w} = (&(3,4),(3,3),(3,2),(3,1),(2,4),(2,3),(2,2),(2,1),\\
    &(1,4),(1,3),(1,2),(1,1),(3,4),(3,3),(3,2),(3,1),\\
    &(2,4),(2,3),(2,2),(2,1),(3,4),(3,3),(3,2),(3,1))
\end{align*}
is a maximal green sequence for $Q \boxtimes Q'$. Moreover, $\sigma_{{\bf v} \boxtimes {\bf w}}$ is an involution fixing the vertices $(2,i)$ and exchanging $(1,i)$ with $(3,i)$ for $1 \leq i \leq 4$.

\subsection{Monoidal duality from a maximal green sequence} In this section, we suppose $\mf{g}$ is untwisted of simply-laced type. Fix a height function $\xi: \Delta_0 \to \Z$. Let $\mc{D}_{\xi}$ be the complete duality datum associated with $\xi$ and fix a reduced word $\underline{w}_0$ for the longest element adapted to $Q_{\xi}$. By Theorem \ref{thm:initial monoidal seed for HL category}, the category $\mathscr{C}_{\mf{g}}^{[-\infty,0],\mc{D}_{\xi},\underline{w}_0}$ has an initial monoidal seed $\mathscr{S}^{<\xi}$ consisting of KR-modules $\mc{M}^{<\xi}_{(i,p)}$, where $(i,p)$ is a vertex of $Q^{<\xi}_{\rm{HL}}$. Our goal in this section is to define a sequence of mutations of $\mathscr{S}^{<\xi}$ such that the modules in the resulting monoidal seed are precisely the left duals $\DD^{-1}(\mc{M}^{<\xi}_{(i,p)})$.

By the definition of the quiver $Q^{<\xi}_{\rm{HL}}$, it is clear that it is isomorphic to the triangle product of $Q_{\xi}$ with the infinite quiver
\[
\dotsb \longleftarrow 3 \longleftarrow 2 \longleftarrow 1.
\]
By combining a sink maximal green sequence of $Q_{\xi}$ with the sequence $(1,2,3,\dots)$ for the quiver above, we will define an infinite version of the sequence of Theorem \ref{thm:maximal green sequence for triangle product}. We first consider a simpler case. For $j \in \Delta_0$, define the sequence
\begin{equation}\label{eq:mutation along a row}
{\bf v}_j = ((j,\xi_j-2), (j,\xi_j-4), (j,\xi_j - 6),\dots)
\end{equation}
of vertices of $Q^{<\xi}_{\rm{HL}}$. Note that this sequence contains all the vertices in the row corresponding to $j$.

\begin{remark}
    In the next results, we will apply some infinite mutation sequences from \cite{HLcluster} on the level of monoidal seeds. The final result should be interpreted as a limit, as in loc.\ cit., which can be computed on each vertex in finitely many steps. We remark that their mutation sequences do not have arrows ``going to infinity'', so that the limit is indeed well-defined and can be lifted to the categorical level. However, the resulting monoidal seed is not reachable.
\end{remark}

\begin{lemma}\label{lemma:mutation along row shifts the spectral parameter}
    If $j$ is a sink of $\xi$, then the monoidal seed obtained by mutating $\mathscr{S}^{<\xi}$ along the sequence ${\bf v}_j$ is isomorphic to $\mathscr{S}^{<s_j^{-1}\xi}$.
\end{lemma}

\begin{proof}
Let us explain how to prove the lemma using the arguments in  \cite[Section 3.2.3]{HLcluster}. Suppose first that $\mathfrak{g}$ is of type $\mathsf{A}_2$, and label the vertices of $\Delta$ by $1$ and $2$. Without loss of generality, assume that $j = 2$. We can thus depict $\mathscr{S}^{<\xi}$ as
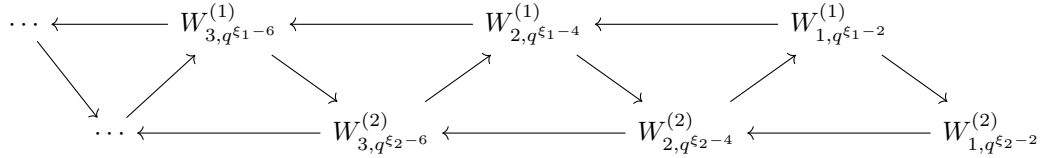
\[\begin{tikzcd}[column sep=small]
	\dotsb && {W^{(1)}_{3,q^{\xi_1-6}}} && {W^{(1)}_{2,q^{\xi_1-4}}} && {W^{(1)}_{1,q^{\xi_1-2}}} & \\
	& \dotsb && {W^{(2)}_{3,q^{\xi_2-6}}} && {W^{(2)}_{2,q^{\xi_2-4}}} && {W^{(2)}_{1,q^{\xi_2-2}}}
	\arrow[from=1-1, to=2-2]
	\arrow[from=1-3, to=1-1]
	\arrow[from=1-3, to=2-4]
	\arrow[from=1-5, to=1-3]
	\arrow[from=1-5, to=2-6]
	\arrow[from=1-7, to=1-5]
	\arrow[from=1-7, to=2-8]
	\arrow[from=2-2, to=1-3]
	\arrow[from=2-4, to=1-5]
	\arrow[from=2-4, to=2-2]
	\arrow[from=2-6, to=1-7]
	\arrow[from=2-6, to=2-4]
	\arrow[from=2-8, to=2-6]
\end{tikzcd}\]
where we write over each vertex $(i,p)$ of $Q_{\rm{HL}}^{<\xi}$ the KR-module $\mc{M}^{<\xi}_{(i,p)}$. The sequence ${\bf v}_2$ reads the bottom row from right to left. By the mutation sequence in \cite[Section 6.1]{HLcluster} and by the argument in \cite[Section 3.2.3]{HLcluster} on $T$-systems (which admit an interpretation as short exact sequences \cite{NakajimatqKRmodules,HernandezKRconjecture}), the monoidal seed $\mu_{{\bf v}_2}(\mathscr{S}^{<\xi})$ obtained after mutating along ${\bf v}_2$ is
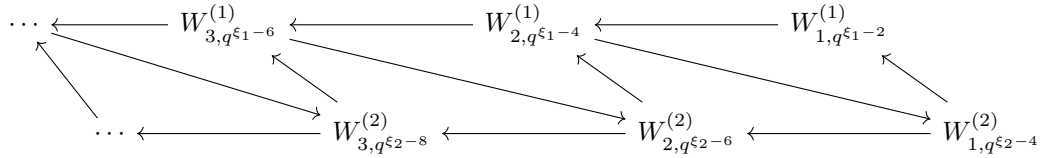
\[\begin{tikzcd}[column sep=small]
	\dotsb && {W^{(1)}_{3,q^{\xi_1-6}}} && {W^{(1)}_{2,q^{\xi_1-4}}} && {W^{(1)}_{1,q^{\xi_1-2}}} & \\
	& \dotsb && {W^{(2)}_{3,q^{\xi_2-8}}} && {W^{(2)}_{2,q^{\xi_2-6}}} && {W^{(2)}_{1,q^{\xi_2-4}}}
	\arrow[from=1-1, to=2-4]
	\arrow[from=1-3, to=1-1]
	\arrow[from=1-3, to=2-6]
	\arrow[from=1-5, to=1-3]
	\arrow[from=1-5, to=2-8]
	\arrow[from=1-7, to=1-5]
	\arrow[from=2-2, to=1-1]
	\arrow[from=2-4, to=1-3]
	\arrow[from=2-4, to=2-2]
	\arrow[from=2-6, to=1-5]
	\arrow[from=2-6, to=2-4]
	\arrow[from=2-8, to=1-7]
	\arrow[from=2-8, to=2-6]
\end{tikzcd}\]
Note that the spectral parameters of the KR-modules on the bottom row are shifted by $q^{-2}$. Sliding the bottom row to the left, we see that $\mu_{{\bf v}_2}(\mathscr{S}^{<\xi})$ coincides with
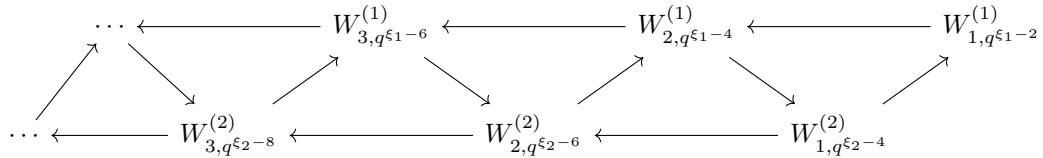
\[\begin{tikzcd}[column sep=small]
	& \dotsb && {W^{(1)}_{3,q^{\xi_1-6}}} && {W^{(1)}_{2,q^{\xi_1-4}}} && {W^{(1)}_{1,q^{\xi_1-2}}} \\
	\dotsb && {W^{(2)}_{3,q^{\xi_2-8}}} && {W^{(2)}_{2,q^{\xi_2-6}}} && {W^{(2)}_{1,q^{\xi_2-4}}}
	\arrow[from=1-2, to=2-3]
	\arrow[from=1-4, to=1-2]
	\arrow[from=1-4, to=2-5]
	\arrow[from=1-6, to=1-4]
	\arrow[from=1-6, to=2-7]
	\arrow[from=1-8, to=1-6]
	\arrow[from=2-1, to=1-2]
	\arrow[from=2-3, to=1-4]
	\arrow[from=2-3, to=2-1]
	\arrow[from=2-5, to=1-6]
	\arrow[from=2-5, to=2-3]
	\arrow[from=2-7, to=1-8]
	\arrow[from=2-7, to=2-5]
\end{tikzcd}\]
which is isomorphic to $\mathscr{S}^{<s_2^{-1}\xi}$, as desired. The proof when $\mf{g}$ is of type $\mathsf{A}_1$ is similar. Finally, as explained in \cite[Section 3.2.3]{HLcluster}, the general case can be reduced to type $\mathsf{A}_2$ by looking at the neighbors of $j$ in $\Delta$ one at a time.
\end{proof}

Denote $\widehat{\un{w}}_0 = (i_s)_{s \in \Z}$. Note that $(i_0, i_{-1},\dots,i_{_{-l(w_0)+1}})$ is a sink sequence for $Q_{\xi}$ that is also a reduced expression of $w_0$. We set
\[
{\bf v}_{\mathscr{D}^{-1}} = ({\bf v}_{i_0},{\bf v}_{i_{-1}},\dots,{\bf v}_{i_{-l(w_0)+1}})
\]
to be the concatenation of the infinite sequences as defined in \eqref{eq:mutation along a row}. Denote $\mathscr{D}^{-1}(\mathscr{S}^{<\xi}) = (\{\mathscr{D}^{-1}(\mc{M}^{<\xi}_{(i,p)})\},Q^{<\xi}_{\rm{HL}},\varnothing)$ the monoidal seed obtained by applying $\mathscr{D}^{-1}$ to $\mathscr{S}^{<\xi}$.

\begin{proposition}\label{proposition:mutation sequence the gives left dual}
    Let $\xi: \Delta_0 \to \Z$ be a height function and let ${\bf v}_{\mathscr{D}^{-1}}$ be the sequence of vertices of $Q^{<\xi}_{\rm{HL}}$ defined above. Then there is an isomorphism of monoidal seeds $\mu_{{\bf v}_{\mathscr{D}^{-1}}}(\mathscr{S}^{<\xi}) \to \mathscr{D}^{-1}(\mathscr{S}^{<\xi})$ sending a vertex of $Q^{<\xi}_{\rm{HL}}$ of the form $(i,\xi_i -2r)$ to $(i^*,\xi_{i^*}-2r)$.
\end{proposition}

\begin{proof}
    Let $\xi' = s^{-1}_{i_{-l(w_0)+1}}\dotsb s_{i_{-1}}^{-1}s_{i_0}^{-1}\xi$. By repeated applications of Lemma \ref{lemma:mutation along row shifts the spectral parameter}, $\mu_{{\bf v}_{\mathscr{D}^{-1}}}(\mathscr{S}^{<\xi})$ is isomorphic to $\mathscr{S}^{<\xi'}$. By Lemma \ref{lemma:identification of HL with GLS quivers} (and Remark \ref{rmk:reflecting height function gives all cases}), $Q^{<\xi'}_{\rm{HL}}$ is isomorphic to the quiver $Q^{[-\infty,-l(w_0)]}(\un{w}_0)$. By the construction of the infinite sequence $\widehat{\un{w}}_0$, there is an isomorphism from this last quiver to $Q^{[-\infty,0]}(\un{w}_0)$ sending the $r$-th vertex on the row corresponding to $i \in \Delta_0$ (counted from right to left) to the $r$-th vertex on the row corresponding to $i^*$. By a second application of Lemma \ref{lemma:identification of HL with GLS quivers}, we obtain the desired isomorphism on the level of the quivers of the monoidal seeds. Now, again by Lemma \ref{lemma:mutation along row shifts the spectral parameter}, notice that the module in $\mu_{{\bf v}_{\mathscr{D}^{-1}}}(\mathscr{S}^{<\xi})$ associated with a vertex of the form $(i,\xi_i-2r)$ is $W^{(i)}_{r,q^{\xi_i-2r-2t_i}}$, where $t_i$ is the number of times the vertex $i$ appears in $(i_0, i_{-1},\dots,i_{_{-l(w_0)+1}})$. Since $\un{w}_0$ is adapted to $Q_{\xi}$, we have $
    t_i = (h + \xi_i - \xi_{i^*})/2$ by \cite[Corollary 2.20(a)]{Bedard}, where $h$ is the Coxeter number of $\Delta$, so this KR-module is
    \[
    W^{(i)}_{r,q^{\xi_{i^*}-2r-h}} \cong \mathscr{D}^{-1}\left(W^{(i^*)}_{r,q^{\xi_{i^*}-2r}}\right) = \mathscr{D}^{-1}\left(\mc{M}^{<\xi}_{(i^*,\xi_{i^*}-2r)}\right),
    \]
    finishing the proof.
\end{proof}

\begin{remark}
    We can iterate the result above as follows. Denote by ${\bf v}_{\mathscr{D}^{-1}}^*$ the sequence obtained from ${\bf v}_{\mathscr{D}^{-1}}$ by replacing each vertex of the form $(i,\xi_i-2r)$ by $(i^*,\xi_{i^*}-2r)$. If we define ${\bf v}_{\mathscr{D}^{-m}}$ as the concatenation of the first $m$ terms of the sequence
    \[
    {\bf v}_{\mathscr{D}^{-1}},{\bf v}_{\mathscr{D}^{-1}}^*,{\bf v}_{\mathscr{D}^{-1}},{\bf v}_{\mathscr{D}^{-1}}^*,\dots,
    \]
    then $\mu_{{\bf v}_{\mathscr{D}^{-m}}}(\mathscr{S}^{<\xi})$ is isomorphic to $\mathscr{D}^{-m}(\mathscr{S}^{<\xi})$.
\end{remark}

\begin{remark}\label{rmk:truncation of mutation sequence that gives dual}
    For $m \geq 1$, the computation of a module in $\mu_{{\bf v}_{\mathscr{D}^{-m}}}(\mathscr{S}^{<\xi})$ associated with a given vertex depends only on finitely many entries of the sequence ${\bf v}_{\mathscr{D}^{-m}}$. This leads to the following observation. For $r\geq 1$, define ${\bf v}_{\mathscr{D}^{-m}}^{\leq r}$ to be the finite subsequence of ${\bf v}_{\mathscr{D}^{-m}}$ consisting of the vertices of the form $(i,\xi_i -2s)$ with $1 \leq s \leq r$. For a fixed vertex $(i,\xi_i-2s)$, we can find $r \geq s$ large enough so that $\mathscr{D}^{-m}(\mc{M}^{<\xi}_{(i,\xi_i-2s)})$ is the module in $\mu_{{\bf v}_{\mathscr{D}^{-m}}^{\leq r}}(\mathscr{S}^{<\xi})$ corresponding to the vertex $(i,\xi_i-2s)$ (if $m$ is even) or $(i^*,\xi_{i^*}-2s)$ (if $m$ is odd).
\end{remark}

\begin{remark}\label{rmk:sequence that gives dual is indeed a maximal green sequence}
    Let $a \leq 0$ be such that every row of $Q^{[a,0]}(\un{w}_0)$ has exactly $r+1$ vertices for some $r \geq 0$. Thus, $Q^{[a,0]}(\un{w}_0)$ is isomorphic to the triangle product of $Q_{\xi}$ and a linear orientation of $\mathsf{A}_{r+1}$. Identifying  $Q^{<\xi}_{\rm{HL}}$ with $Q^{[-\infty,0]}(\un{w}_0)$ by Lemma \ref{lemma:identification of HL with GLS quivers}, we can view the sequence ${\bf v}_{\mathscr{D}^{-1}}^{\leq r}$ from Remark \ref{rmk:truncation of mutation sequence that gives dual} as a sequence of unfrozen vertices in $Q^{[a,0]}(\un{w}_0)$. By Lemma \ref{lemma:source maximal green sequence}, Proposition \ref{proposition:sink maximal green sequence}, and Theorem \ref{thm:maximal green sequence for triangle product}, ${\bf v}_{\mathscr{D}^{-1}}^{\leq r}$ is a maximal green sequence for the quiver obtained from $Q^{[a,0]}(\un{w}_0)$ by removing the frozen vertices, and the associated permutation coincides with the one described in Proposition \ref{proposition:mutation sequence the gives left dual}.
\end{remark}

\section{The inverse dualizing bimodules and their properties}\label{section:inverse dualizing bimodules}

In this section, we recall the definition and main properties of the relative inverse dualizing bimodule and the relative derived preprojective algebra. More details can be found in \cite{KellerCY}, \cite{YeungRelativeCY}, \cite{YilinWuJacobifinite}, and \cite{KellerWang}. In these references, the definitions are given for general dg categories, but we will present them in the generality that we shall later need. 

We aim to provide the main tools to prove Theorem \ref{thm:coincidence of extensions} and \Cref{cor:adapted words give proper Ginzburgs}. The latter states that the relative Ginzburg dg algebras arising in our setting are proper and hence we can apply \Cref{thm:canonical quantum structure}, yielding a $\Uplambda$-cluster algebra structure.

Properness will be shown for a quasi-isomorphic algebra, namely a relative derived preprojective algebra, see \Cref{theorem:Ginzburg algebras as completion}. For such algebras, we observe in \Cref{lemma:Omega nilpotent implies proper} that properness follows if the relative inverse dualizing bimodule is nilpotent. The latter property will be obtained through a reduction to a quiver of type $\mathsf{A}_1$ and the gluing result of \Cref{section:gluing bimodules}, which is of independent interest. \smallskip

{\it Notation:} For a small $k$-category $\mc{A}$, we denote by $\Modcat\mc{A}$ its category of \emph{right} modules and by $\modcat\mc{A}$ the full subcategory of $\mc{A}$-modules $M: \mc{A}^{\rm{op}} \to \Modcat k$ such that $M(x)$ is finite-dimensional for all $x \in \mc{A}$. Let $\mc{D}(\mc{A})$, $\per(\mc{A})$, and $\pvd(\mc{A})$ denote the derived, the perfect derived, and the perfectly valued derived categories of $\mc{A}$. The suspension functors of these triangulated categories will be denoted by $\Sigma$. We call $\mc{A}$ \emph{proper} if it is $\Hom$-finite and equivalent to a category with finitely many objects. We denote by $\mc{A}^e = \mc{A}^{\rm{op}} \ten \mc{A}$ the enveloping category of $\mc{A}$. We say $\mc{A}$ is \emph{smooth} if the regular $\mc{A}$-bimodule $\mc{A}$ belongs to $\per(\mc{A}^e)$. When $\mc{A}$ is proper, smoothness is equivalent to finiteness of the global dimension since we assume $k$ is algebraically closed (hence perfect). If $\per(\mc{A}) = \pvd(\mc{A})$, we denote this category by $\mc{D}^b(\mc{A})$. This is the case, for example, if $\mc{A}$ is smooth and proper. We view $k$-algebras as $k$-categories with a single object. Morphisms of algebras are not necessarily unital.

\subsection{The definition} Let $\mc{A}$ be a small $k$-category. For $M \in \mc{D}(\mc{A}^e)$, we define its bimodule dual $M^{\vee} = \RHom_{\mc{A}^e}(M,\mc{A}^e)$. We view it as an object of $\mc{D}(\mc{A}^e)$ (that is, as a \emph{right} dg $\mc{A}^e$-module) by composing with the duality $(\mc{A}^e)^{\rm{op}} \to \mc{A}^e$ sending $f \ten g$ to $g \ten f$ for any morphisms $f,g$ in $\mc{A}$. We define the \emph{absolute inverse dualizing bimodule} of $\mc{A}$ as the bimodule dual of the regular $\mc{A}$-bimodule:
\[
\Omega_{\mc{A}} = \mc{A}^{\vee} = \RHom_{\mc{A}}(\mc{A},\mc{A}^e) \in \mc{D}(\mc{A}^e).
\]

\begin{proposition}[{\cite[Lemma 4.1]{KellerCYtriangulated}}]
    Suppose $\mc{A}$ is smooth. For $L \in \pvd(\mc{A})$ and $M \in \mc{D}(\mc{A})$, there is a canonical isomorphism
    \[
    D\Hom_{\mc{D}(\mc{A})}(L,M) \xrightarrow{\sim} \Hom_{\mc{D}(\mc{A})}(M \lten_{\mc{A}} \Omega_{\mc{A}},L).
    \]
\end{proposition}

\begin{remark}\label{rmk:absolute Omega gives Serre functor}
If $\mc{A}$ is proper and smooth, then $\mathbb{S}_{\mc{A}} = - \lten_{\mc{A}} D\mc{A}^{\rm{op}}$ defines a Serre functor on the bounded derived category $\mc{D}^b(\mc{A})$, that is,  $\mathbb{S}_{\mc{A}}$ is an autoequivalence such that there is a natural isomorphism
\[
\Hom_{\mc{D}^b(\mc{A})}(L,M) \cong D\Hom_{\mc{D}^b(\mc{A})}(M,\mathbb{S}_{\mc{A}}(L))
\]
for $L,M \in \mc{D}^b(\mc{A})$. By the proposition above, there is an isomorphism of functors $\mathbb{S}_{\mc{A}}^{-1} \cong - \lten_{\mc{A}}\Omega_{\mc{A}}$ and, consequently, $\mathbb{S}_{\mc{A}} \cong \RHom_{\mc{A}}(\Omega_{\mc{A}},-)$.
\end{remark}

Let $\mc{B} \to \mc{A}$ be a linear functor between small $k$-categories. Since we work over a field, we have canonical isomorphisms in the derived category of bimodules:
\[
\mc{A} \lten_{\mc{B}} \mc{A} = \mc{A} \lten_{\mc{B}} \mc{B} \lten_{\mc{B}} \mc{A} = \mc{B} \lten_{\mc{B}^e} \mc{A}^e.
\]
We will use these identifications freely. There is a canonical map $\mc{A} \lten_{\mc{B}} \mc{A} \longrightarrow \mc{A}$ which corresponds to the functor $\mc{B} \to \RHom_{\mc{A}^e}(\mc{A}^e,\mc{A}) = \mc{A}$, viewed as a morphism of $\mc{B}$-bimodules, under the derived tensor-$\Hom$ adjunction. Alternatively, it coincides with the composition
\[
\mc{A} \lten_{\mc{B}} \mc{A} \longrightarrow H^0(\mc{A} \lten_{\mc{B}} \mc{A}) = \mc{A} \ten_{\mc{B}} \mc{A} \longrightarrow \mc{A},
\]
where the second map is induced by the composition law in the category $\mc{A}$. We define the \emph{relative inverse dualizing bimodule} $\Omega_{\mc{A},\mc{B}} \in \mc{D}(\mc{A}^e)$ as the bimodule dual of the cone of the canonical map $\mc{A} \lten_{\mc{B}} \mc{A} \longrightarrow \mc{A}$.

\begin{lemma}[{\cite[Lemma 2.9]{KellerCY}}]\label{lemma:bimodule dual compatible with induction}
    If $M \in \per(\mc{B}^e)$, then there is a canonical isomorphism
    \[
    (M \lten_{\mc{B}^e}\mc{A}^e)^{\vee} \cong M^{\vee} \lten_{\mc{B}^e} \mc{A}^e.
    \]
    In particular, if $\mc{B}$ is smooth, then $(\mc{A}\lten_{\mc{B}}\mc{A})^{\vee} \cong \Omega_{\mc{B}} \lten_{\mc{B}^e}\mc{A}^e$.
\end{lemma}

\begin{remark}\label{rmk:computing Omega with Serre functor}
    By the lemma above, if $\mc{B}$ is smooth, then $\Omega_{\mc{A},\mc{B}}$ is the cocone of a canonical morphism $\Omega_{\mc{A}} \to \Omega_{\mc{B}} \lten_{\mc{B}^e}\mc{A}^e$. Consequently, if $\mc{A}$ and $\mc{B}$ are both smooth and proper, then we can compute $M \lten_{\mc{A}}\Omega_{\mc{A},\mc{B}}$ as the cocone of a canonical map
    \[
    \mathbb{S}^{-1}_{\mc{A}}(M) \longrightarrow \mathbb{S}^{-1}_{\mc{B}}(\restr{M}{\mc{B}})\lten_{\mc{B}}\mc{A}
    \]
    and $\RHom_{\mc{A}}(\Omega_{\mc{A},\mc{B}},M)$ as the cone of a canonical map
    \[
    \RHom_{\mc{B}}(\mc{A},\mathbb{S}_{\mc{B}}(\restr{M}{\mc{B}}))\longrightarrow\mathbb{S}_{\mc{A}}(M)
    \]
    for any $M \in \mc{D}^b(\mc{A})$. The resulting complexes are again in $\mc{D}^b(\mc{A})$.
\end{remark}

\subsection{Relative derived preprojective algebras} We keep the notation from the previous section. Suppose that $\mc{A}$ and $\mc{B}$ are both smooth. Let $\omega \in \mc{D}(\mc{A}^e)$ be a cofibrant replacement of $\Sigma^2\Omega_{\mc{A},\mc{B}}$. The \emph{relative derived preprojective category} $\bP_3(\mc{A},\mc{B})$ is the tensor dg category
\[
T_{\mc{A}}(\omega) = \bigoplus_{n \geq 0} \mc{\omega}^{\ten_{\mc{A}}n} = \mc{A} \oplus \omega \oplus (\omega \ten_{\mc{A}} \omega) \oplus \dotsb.
\]
Its objects are the same as those of $\mc{A}$, and the morphism complexes are the evaluations of the dg $\mc{A}$-bimodule above. Composition is given by concatenation and the $\mc{A}$-bimodule structure. We remark that $\bP_3(\mc{A},\mc{B})$ is defined up to quasi-equivalence. When $\mc{A}$ and $\mc{B}$ are algebras, we refer to $\bP_3(\mc{A},\mc{B})$ as the associated \emph{relative derived preprojective algebra}.

\begin{remark}
    If $\mc{A}$ has finitely many objects, we can turn $\bP_3(\mc{A},\mc{B})$ into a dg algebra by taking the direct sum of its morphism complexes. The resulting dg algebra will be called the \emph{total dg algebra} of $\bP_3(\mc{A},\mc{B})$. It is derived Morita equivalent to $\bP_3(\mc{A},\mc{B})$.
\end{remark}

We now show how to realize a class of relative derived preprojective algebras as relative Ginzburg dg algebras. 

Let $(Q',F)$ be a finite ice quiver. Let $A$ be the finite-dimensional quotient of $kQ'$ by an admissible ideal $I$, and suppose that $A$ has global dimension at most $2$. Let $B = kF$ and assume $I \cap B = 0$, so that we can view $B$ as a subalgebra of $A$. Let $R \subseteq I$ be the union over all $i,j \in Q'_0$ of a set of representatives of a basis of $e_j(I/(IJ+JI))e_i$, where $J \subseteq kQ'$ denotes the ideal generated by the arrows. Suppose further that $R$ generates $I$ as an ideal. We define a new quiver $Q$ by adding additional arrows $\rho_r: j \to i$ for each relation $r \in R$ from $i$ to $j$. We equip it with the potential
\[
W = \sum_{r \in R}r\rho_r.
\]
In this way, we have an ice quiver with potential $(Q,F,W)$. Its relative Ginzburg dg algebra relates to the relative derived preprojective algebra of $A$ and $B$ as follows. 

\begin{theorem}\label{theorem:Ginzburg algebras as completion}
    With the notation and hypotheses above, the relative derived preprojective algebra $\bP_3(A,B)$ is quasi-isomorphic to the non-completed relative Ginzburg dg algebra $\bG(Q,F,W)$.
\end{theorem}

\begin{proof}
    The absolute case, that is, when $F$ is empty, is \cite[Theorem 6.10]{KellerCY} (see also \cite[Section 3]{KellerErratumCY}). The key step in the proof is to replace $A$ by a quasi-isomorphic cofibrant dg algebra $\tilde{A}$ for which the derived preprojective algebra can be computed using an explicit description of $\Omega_{\tilde{A}}$. This replacement is of the form $\tilde{A} = (k\tilde{Q},d)$, where $\tilde{Q}$ is a finite non-positively graded quiver containing $Q'$. In the relative case, we can view $B$ as a dg subalgebra of $\tilde{A}$ via the inclusion $F \subseteq \tilde{Q}$, then one can similarly compute $\bP_3(\tilde{A},B)$ using an explicit description of $\Omega_{\tilde{A},B}$ (see, for instance, \cite[Section 3.6]{YilinWuJacobifinite}).
\end{proof}

We conclude this section with a useful observation. We say that $M \in \mc{D}(\mc{A}^e)$ is \emph{nilpotent} if there exists $n \geq 1$ such that $M^{\lten_{\mc{A}}n} = 0$ in $\mc{D}(\mc{A}^e)$.

\begin{lemma}\label{lemma:Omega nilpotent implies proper}
    Let $\mc{B} \to \mc{A}$ be a linear functor between small $k$-categories. Suppose that $\mc{A}$ and $\mc{B}$ are both smooth and proper. If $\Omega_{\mc{A},\mc{B}}$ is nilpotent, then $\bP_3(\mc{A},\mc{B})$ is proper.
\end{lemma}

\begin{proof}
Observe that $\Omega_{\mc{A},\mc{B}}$ and its derived tensor powers are in $\pvd(\mc{A}^e)$. The lemma then follows from the construction of $\bP_3(\mc{A},\mc{B})$.
\end{proof}

\subsection{A gluing result for relative inverse dualizing bimodules}\label{section:gluing bimodules} Let $\mc{A}$ be a small $k$-category. Let $\mc{A}_1, \mc{A}_2 \subseteq \mc{A}$ be strictly full subcategories which together contain all objects of $\mc{A}$. Denote $\mc{C} = \mc{A}_1 \cap \mc{A}_2$. We will assume that one of the following conditions holds:
\begin{equation}\label{eq:source condition for gluing theorem}
\Hom(\mc{A}_1\setminus \mc{C}, \mc{A}_2) = \Hom(\mc{A}_2\setminus \mc{C}, \mc{A}_1) = 0
\end{equation}
or
\begin{equation}\label{eq:sink condition for gluing theorem}
\Hom(\mc{A}_2, \mc{A}_1\setminus \mc{C}) = \Hom(\mc{A}_1, \mc{A}_2\setminus \mc{C}) = 0.
\end{equation}
We have the following ``gluing'' result.

\begin{proposition}\label{proposition:gluing of bimodules} Consider $\mc{A}_i, \mc{A}$ and $\mc{C}$ as above such that \eqref{eq:source condition for gluing theorem} or \eqref{eq:sink condition for gluing theorem} holds. Let $\mc{B} \subseteq \mc{A}_1$ be a not necessarily full subcategory. If $\mc{A}_1$, $\mc{A}_2$, $\mc{B}$ and $\mc{C}$ are smooth, then there is a distinguished triangle
    \[\begin{tikzcd}
	{\Omega_{\mc{A}_2,\mc{C}} \lten_{\mc{A}_2^e} \mc{A}^e} & {\Omega_{\mc{A},\mc{B}}} & {\Omega_{\mc{A}_1,\mc{B}} \lten_{\mc{A}_1^e} \mc{A}^e} & {\Sigma\Omega_{\mc{A}_2,\mc{C}} \lten_{\mc{A}_2^e} \mc{A}^e}
	\arrow[from=1-1, to=1-2]
	\arrow[from=1-2, to=1-3]
	\arrow[from=1-3, to=1-4]
    \end{tikzcd}\]
    in $\mc{D}(\mc{A}^e)$.
\end{proposition}

\begin{proof}
    The inclusions $\mc{C} \subseteq \mc{A}_1,\mc{A}_2 \subseteq \mc{A}$ induce the commutative square in $\mc{D}(\mc{A}^e)$ on the top left corner in the diagram below. It can be completed to the full diagram by the $3 \times 3$ lemma for triangulated categories (see \cite[Proposition 1.1.11]{BeilinsonBernsteinDeligne81}):
    \[\begin{tikzcd}
	{\mc{C} \lten_{\mc{C}^e} \mc{A}^e} & {\mc{A}_1 \lten_{\mc{A}_1^e} \mc{A}^e} & X & {\Sigma \mc{C} \lten_{\mc{C}^e} \mc{A}^e} \\
	{\mc{A}_2 \lten_{\mc{A}_2^e} \mc{A}^e} & {\mc{A}} & Y & {\Sigma \mc{A}_2 \lten_{\mc{A}_2^e} \mc{A}^e} \\
	V & W & Z & {\Sigma V} \\
	{\Sigma\mc{C} \lten_{\mc{C}^e} \mc{A}^e} & {\Sigma\mc{A}_1 \lten_{\mc{A}_1^e} \mc{A}^e} & {\Sigma X}
	\arrow["{g}", from=1-1, to=1-2]
	\arrow["{f}"', from=1-1, to=2-1]
	\arrow[from=1-2, to=1-3]
	\arrow["{f'}", from=1-2, to=2-2]
	\arrow[from=1-3, to=1-4]
	\arrow[from=1-3, to=2-3]
	\arrow[from=1-4, to=2-4]
	\arrow["{g'}", from=2-1, to=2-2]
	\arrow[from=2-1, to=3-1]
	\arrow[from=2-2, to=2-3]
	\arrow[from=2-2, to=3-2]
	\arrow[from=2-3, to=2-4]
	\arrow[from=2-3, to=3-3]
	\arrow[from=2-4, to=3-4]
	\arrow[from=3-1, to=3-2]
	\arrow[from=3-1, to=4-1]
	\arrow[from=3-2, to=3-3]
	\arrow[from=3-2, to=4-2]
	\arrow[from=3-3, to=3-4]
	\arrow[from=3-3, to=4-3]
	\arrow[from=4-1, to=4-2]
	\arrow[from=4-2, to=4-3]
\end{tikzcd}\]
Here, all squares are commutative, and the first three rows and columns are distinguished triangles in $\mc{D}(\mc{A}^e)$. We now prove that $Z \cong 0$ using the conditions \eqref{eq:source condition for gluing theorem} or \eqref{eq:sink condition for gluing theorem}.

First, let us assume that \eqref{eq:source condition for gluing theorem} holds.  It suffices to show that $x^{\wedge} \lten_{\mc{A}} Z \cong 0$ for all $x \in \mc{A}$, where we write $x^{\wedge}$ for the representable $\mc{A}$-module $\mc{A}(-,x)$. We argue separately depending on whether $x$ is in $\mc{A}_1$ or $\mc{A}_2$. Suppose first that $x \in \mc{A}_1$. Let us start by showing that $x^{\wedge} \lten_{\mc{A}} f'$ is an isomorphism. Using that $\mc{A}_1 \lten_{\mc{A}_1^e} \mc{A}^e \cong \mc{A} \lten_{\mc{A}_1} \mc{A}$, note that $x^{\wedge} \lten_{\mc{A}} f'$ identifies with the canonical map
\[
x^{\wedge} \lten_{\mc{A}_1} \mc{A} \longrightarrow x^{\wedge}
\]
induced by the $\mc{A}$-module action on $x^{\wedge}$. Since $\mc{A}_1$ is a full subcategory of $\mc{A}$ and $x \in \mc{A}_1$, the restriction of $x^{\wedge}$ to $\mc{D}(\mc{A}_1)$ is the representable $\mc{A}_1$-module $\mc{A}_1(-,x)$. Thus, $x^{\wedge} \lten_{\mc{A}_1} \mc{A} \cong x^{\wedge}$ and the above map becomes the identity on $x^{\wedge}$, which is indeed an isomorphism. Let us now prove that $x \lten_{\mc{A}} f$ is an isomorphism. Consider the canonical map
\[
\varphi: x^{\wedge} \lten_{\mc{C}} \mc{A}_2 \longrightarrow \restr{x^{\wedge}}{\mc{A}_2}
\]
in $\mc{D}(\mc{A}_2)$. Then $x^{\wedge} \lten_{\mc{A}} f$ identifies with $\varphi \lten_{\mc{A}_2} \mc{A}$. Our problem is reduced to showing that $\varphi$ is an isomorphism. Since there are no morphisms from an object in $\mc{A}_2 \setminus \mc{C}$ to an object in $\mc{A}_1$ by the condition (\ref{eq:source condition for gluing theorem}), we can find a projective resolution of $\restr{x^{\wedge}}{\mc{A}_2}$ whose terms are direct sums of direct summands of representable modules of the form $\mc{A}_2(-,y)$ with $y \in \mc{C}$. This resolution can be taken to be finite because $\mc{A}_2$ is smooth, so it suffices to verify that the canonical map
\[
\mc{A}_2(-,y) \lten_{\mc{C}} \mc{A}_2 \longrightarrow \mc{A}_2(-,y)
\]
is an isomorphism for any $y \in \mc{C}$. This can be done as before. Now that we know that $x^{\wedge} \lten_{\mc{A}}f$ and $x^{\wedge} \lten_{\mc{A}} f'$ are isomorphisms, we have $x^{\wedge} \lten_{\mc{A}} V \cong x^{\wedge} \lten_{\mc{A}} W \cong 0$, hence $x^{\wedge} \lten_{\mc{A}} Z \cong 0$, as desired. If $x \in \mc{A}_2$, a symmetric argument shows that $x^{\wedge} \lten_{\mc{A}} g$ and $x^{\wedge} \lten_{\mc{A}} g'$ are isomorphisms, so $x^{\wedge} \lten_{\mc{A}} X \cong x^{\wedge} \lten_{\mc{A}} Y \cong 0$ and again we have $x^{\wedge} \lten_{\mc{A}} Z \cong 0$. This concludes the proof that $Z \cong 0$ when \eqref{eq:source condition for gluing theorem} holds.

On the other hand, if \eqref{eq:sink condition for gluing theorem} holds, we can show instead that $Z \lten_{\mc{A}}\mc{A}(x,-) \cong 0$ for all $x \in \mc{A}$. This can be proven as in the previous paragraph using a symmetric argument. Consequently, we have $Z \cong 0$ independently of which condition holds. We deduce that $V \cong W$ and $X \cong Y$.

As we did in the beginning, the inclusions $\mc{B} \subseteq \mc{A}_1 \subseteq \mc{A}$ induce the commutative square in $\mc{D}(\mc{A}^e)$ on the top left corner in the diagram below. The full diagram is obtained by applying the octahedral axiom:
\[\begin{tikzcd}
	{\mc{B} \lten_{\mc{B}^e} \mc{A}^e} & {\mc{A}_1 \lten_{\mc{A}_1^e} \mc{A}^e} & \color{blue}{M} & {\Sigma \mc{B} \lten_{\mc{B}^e} \mc{A}^e} \\
	{\mc{B} \lten_{\mc{B}^e} \mc{A}^e} & {\mc{A}} & \color{blue}{N} & {\Sigma \mc{B} \lten_{\mc{B}^e} \mc{A}^e} \\
	& W & \color{blue}{W} \\
	& {\Sigma\mc{A}_1 \lten_{\mc{A}_1^e} \mc{A}^e} & \color{blue}{\Sigma M}
	\arrow[from=1-1, to=1-2]
	\arrow[from=1-1, to=2-1, equal]
	\arrow[from=1-2, to=1-3]
	\arrow["f'", from=1-2, to=2-2]
	\arrow[from=1-3, to=1-4]
	\arrow[from=1-3, to=2-3, color=blue]
	\arrow[from=1-4, to=2-4, equal]
	\arrow[from=2-1, to=2-2]
	\arrow[from=2-2, to=2-3]
	\arrow[from=2-2, to=3-2]
	\arrow[from=2-3, to=2-4]
	\arrow[from=2-3, to=3-3, color=blue]
	\arrow[from=3-2, to=3-3, equal]
	\arrow[from=3-2, to=4-2]
	\arrow[from=3-3, to=4-3, color=blue]
	\arrow[from=4-2, to=4-3]
\end{tikzcd}\]
Here, the squares are commutative, and the rows and columns are distinguished triangles in $\mc{D}(\mc{A}^e)$. Observe that the map $f'$ and its cone $W$ are the same as in the previous diagram. Applying the bimodule dual functor $(-)^{\vee} = \RHom_{\mc{A}^e}(-,\mc{A}^e)$ to the triangle in blue, we get the distinguished triangle
\[\begin{tikzcd}
	{W^{\vee}} & {N^{\vee}} & {M^{\vee}} & {\Sigma W^{\vee},}
	\arrow[from=1-1, to=1-2]
	\arrow[from=1-2, to=1-3]
	\arrow[from=1-3, to=1-4]
\end{tikzcd}\]
which we now show is the triangle from the statement. By definition, $N^{\vee} = \Omega_{\mc{A},\mc{B}}$. If we write $C_1$ for the cone of the canonical map $\mc{A}_1 \lten_{\mc{B}}\mc{A}_1 \to \mc{A}_1$, then
\[
M^{\vee} \cong (C_1 \lten_{\mc{A}_1}\mc{A}^e)^{\vee} \cong \RHom_{\mc{A}_1^e}(C_1,\mc{A}_1^e) \lten_{\mc{A}_1}\mc{A}^e = \Omega_{\mc{A}_1,\mc{B}} \lten_{\mc{A}_1}\mc{A}^e,
\]
where the second isomorphism follows by Lemma \ref{lemma:bimodule dual compatible with induction}. Note that $C_1$ is perfect because $\mc{A}_1$ and $\mc{B}$ are smooth, so the hypothesis of this lemma is satisfied. Finally, if $C_2$ is the cone of the canonical map $\mc{A}_2\lten_{\mc{C}}\mc{A}_2 \to \mc{A}_2$, then
\[
W^{\vee} \cong V^{\vee} \cong (C_2 \lten_{\mc{A}_2}\mc{A}^e)^{\vee} \cong \RHom_{\mc{A}_2^e}(C_2,\mc{A}_2^e) \lten_{\mc{A}_2}\mc{A}^e = \Omega_{\mc{A}_2,\mc{C}} \lten_{\mc{A}_2}\mc{A}^e,
\]
where the first isomorphism is a consequence of $V \cong W$ and the third one follows again by Lemma \ref{lemma:bimodule dual compatible with induction} since $\mc{A}_2$ and $\mc{C}$ are smooth. This finishes the proof.
\end{proof}

It will be important to understand the (co)induction functor from $\mc{A}_i$ to $\mc{A}$ when \eqref{eq:source condition for gluing theorem} or \eqref{eq:sink condition for gluing theorem} holds. For $i=1,2$ and an $\mc{A}_i$-module $M: \mc{A}_i^{\rm{op}} \to \Modcat k$, we define an $\mc{A}$-module $\mu_i(M): \mc{A}^{\rm{op}} \to \Modcat k$ on objects by
\[
\mu_i(M)(x) = \begin{cases}
        M(x) &\textrm{if } x \in \mc{A}_i,\\
        0   &\textrm{if } x \not\in \mc{A}_i,
    \end{cases}
\]
and on morphisms by $\mu_i(M)(f) = M(f)$ for any morphism $f$ in $\mc{A}_i$. This assignment is indeed functorial because either $\Hom(\mc{A}\setminus \mc{A}_i, \mc{A}_i) = 0$ or $\Hom(\mc{A}_i,\mc{A} \setminus \mc{A}_i) = 0$. It defines an exact and fully faithful functor $\mu_i: \Modcat\mc{A}_i \to \Modcat\mc{A}$ whose image are the $\mc{A}$-modules that vanish on $\mc{A} \setminus \mc{A}_i$.

\begin{lemma}\label{lemma:(co)induction with gluing conditions}
    Let $M$ be an $\mc{A}_i$-module. If \eqref{eq:source condition for gluing theorem} holds, then there are natural isomorphisms
    \[
    M \lten_{\mc{A}_i}\mc{A} \cong M \ten_{\mc{A}_i}\mc{A} \cong \mu_i(M).
    \]
    Similarly, if \eqref{eq:sink condition for gluing theorem} holds, then there are natural isomorphisms
    \[
    \RHom_{\mc{A}_i}(\mc{A},M) \cong \Hom_{\mc{A}_i}(\mc{A},M) \cong \mu_i(M).
    \]
\end{lemma}

\begin{proof}
Suppose \eqref{eq:source condition for gluing theorem} holds. To prove the first isomorphism above, it suffices to show that $\mc{A}(x,-)$ is projective as a left $\mc{A}_i$-module for any $x \in \mc{A}$. Indeed, its restriction to $\mc{A}_i$ vanishes if $x \in \mc{A} \setminus \mc{A}_i$ by our hypothesis, and it is the representable left $\mc{A}_i$-module associated with $x$ if $x \in \mc{A}_i$ since $\mc{A}_i \subseteq \mc{A}$ is full. The same argument also shows that $M \ten_{\mc{A}_i} \mc{A}$ vanishes on $\mc{A} \setminus \mc{A}_i$ and that its restriction to $\mc{A}_i$ is $M \ten_{\mc{A}_i} \mc{A}_i = M$. We immediately conclude that $M \ten_{\mc{A}_i} \mc{A} \cong \mu_i(M)$. The proof of the second statement is analogous.
\end{proof}

\begin{lemma}\label{lemma:global dimension of gluing}
    Suppose \eqref{eq:source condition for gluing theorem} or \eqref{eq:sink condition for gluing theorem} holds. If $\mc{A}_1$ and $\mc{A}_2$ have finite global dimension, then the same holds for $\mc{A}$. In this case, $\gldim(\mc{A}) \leq \max\{d_1,d_2\}$ where $d_i = \gldim(\mc{A}_i)$. 
\end{lemma}

\begin{proof}
    If \eqref{eq:source condition for gluing theorem} holds, then by Lemma \ref{lemma:(co)induction with gluing conditions} the exact functor $\mu_i$ is isomorphic to $- \ten_{\mc{A}_i}\mc{A}$, which preserves projective modules since it is left adjoint to the restriction functor from $\mc{A}$ to $\mc{A}_i$. Similarly, if \eqref{eq:sink condition for gluing theorem} holds, the same lemma implies that $\mu_i$ preserves injective modules. Together with our hypothesis that the projective and injective dimensions of any $\mc{A}_i$-module are at most $d_i$, we deduce that, to show that $\mc{A}$ has global dimension at most $\max\{d_1,d_2\}$, it suffices to prove that any $\mc{A}$-module $M$ is an extension of modules in the image of $\mu_1$ and $\mu_2$. Indeed, defining $M_1 = \mu_1(\restr{M}{\mc{A}_1})$ and $M_2$ on objects by
    \[
    M_2(x) = \begin{cases}
        M(x) &\textrm{if } x \in \mc{A}_2 \setminus \mc{C},\\
        0   &\textrm{if } x \in \mc{A}_1,
    \end{cases}
    \]
    with the action on morphisms coming from $M$, it is clear that $M$ is an extension of $M_2$ by $M_1$ (if \eqref{eq:source condition for gluing theorem} holds) or an extension of $M_1$ by $M_2$ (if \eqref{eq:sink condition for gluing theorem} holds).
\end{proof}

\subsection{Tensor products (of path algebras)}\label{section:tensor products}

Let $\mc{A}_1$, $\mc{A}_2$ and $\mc{B}_2$ be small $k$-categories. Let $\mc{B}_2 \to \mc{A}_2$ be a linear functor. The next result describes the relative inverse dualizing bimodule $\Omega_{\mc{A},\mc{B}}$ for the induced functor $\mc{B} \to \mc{A}$, where 
$$\mc{A} = \mc{A}_1 \ten \mc{A}_2 \text{ and } \mc{B} = \mc{A}_1 \ten \mc{B}_2.$$

\begin{lemma}\label{lemma:inverse dualizing bimodule for tensor product}
   If $\mc{A}_1$, $\mc{A}_2$ and $\mc{B}_2$ are smooth, then we have an isomorphism $\Omega_{\mc{A},\mc{B}} \cong \Omega_{\mc{A}_1} \ten \Omega_{\mc{A}_2,\mc{B}_2}$ in $\mc{D}(\mc{A}^e)$.
\end{lemma}

\begin{proof}
We can compute $\mc{A} \lten_{\mc{B}} \mc{A} = \mc{A} \lten_{\mc{B}} \mc{B} \lten_{\mc{B}} \mc{A}$ by taking a cofibrant resolution of $\mc{B}$ in $\mc{D}(\mc{B}^e)$. Since $\mc{B} = \mc{A}_1 \otimes \mc{B}_2$, such a resolution can be obtained by tensoring a cofibrant resolution of $\mc{A}_1$ in $\mc{D}(\mc{A}_1^e)$ with a cofibrant resolution of $\mc{B}_2$ in $\mc{D}(\mc{B}_2^e)$. Now, for an $\mc{A}_1$-bimodule $M$ and a $\mc{B}_2$-bimodule $N$, we have an evident natural isomorphism
\[
\mc{A} \ten_{\mc{B}} (M \ten N) \ten_{\mc{B}} \mc{A} \cong M \ten (\mc{A}_2 \ten_{\mc{B}_2} N \ten_{\mc{B}_2} \mc{A}_2)
\]
of $\mc{A}$-bimodules. We deduce that
\[
\mc{A} \lten_{\mc{B}} \mc{A} \cong \mc{A}_1 \otimes (\mc{A}_2 \lten_{\mc{B}_2} \mc{A}_2)
\]
in $\mc{D}(\mc{A}^e)$. Under this identification, the canonical morphism $f: \mc{A} \lten_{\mc{B}} \mc{A} \to \mc{A}$ becomes the tensor product of $\mc{A}_1$ with the canonical morphism $f_2: \mc{A}_2 \lten_{\mc{B}_2} \mc{A}_2 \to \mc{A}_2$. This yields the isomorphism $\cone(f) \cong \mc{A}_1 \ten \cone(f_2)$. By the definition of the inverse dualizing bimodules, we need to show that
\[
\RHom_{\mc{A}^e}(\cone(f), \mc{A}^e) \cong \RHom_{\mc{A}_1^e}(\mc{A}_1,\mc{A}_1^e) \ten \RHom_{\mc{A}_2}(\cone(f_2),\mc{A}_2^e).
\]
This follows from the hypothesis and the natural isomorphism
\[
\RHom_{\mc{A}^e}(V \otimes W, \mc{A}^e) \xleftarrow[]{\sim} \RHom_{\mc{A}_1^e}(V,\mc{A}_1^e) \ten \RHom_{\mc{A}_2^e}(W,\mc{A}_2^e)
\]
in $\mc{D}(\mc{A}^e)$ for $V \in \per(\mc{A}_1^e)$ and $W \in \per(\mc{A}_2^e)$.
\end{proof}

Keeping the notations above, we have the following immediate corollary.

\begin{corollary}\label{corollary:Omega is nilpotent for tensor product}
Suppose $\mc{A}_1$, $\mc{A}_2$ and $\mc{B}_2$ are smooth. If $\Omega_{\mc{A}_1}$ or $\Omega_{\mc{A}_2,\mc{B}_2}$ is nilpotent, then the same holds for $\Omega_{\mc{A},\mc{B}}$.  
\end{corollary}

We are interested in the case of tensor products of finite-dimensional path algebras. Let us recall some of their properties. Let $Q$ and $Q'$ be two finite acyclic quivers. The tensor product $A = kQ \ten kQ'$ is a finite-dimensional algebra of global dimension at most $2$. By \cite[Section 5]{Kelannals}, it can be described as a quotient of the path algebra of $Q \otimes Q'$ (as in Section \ref{section:triangle products and green sequences}) by the commutation relations
\begin{equation}\label{eq:commutation relations for tensor products}
(j,\alpha')(\alpha,i') - (\alpha,j')(i,\alpha')
\end{equation}
for all arrows $\alpha:i \to j$ and $\alpha': i' \to j'$ in $Q$ and $Q'$. For $i_0 \in Q'$, consider the subalgebra $B = kQ \otimes e_{i_0}(kQ')e_{i_0} \cong kQ$ of $A$. By Theorem \ref{theorem:Ginzburg algebras as completion}, we deduce that $\bP_3(A,B)$ is quasi-isomorphic to the relative Ginzburg dg algebra of the triangle product $Q \boxtimes Q'$, where we freeze the full subquiver on the vertices of the form $(i,i_0)$ and endow it with the potential
\begin{equation}\label{eq:potential for tensor product}
\sum (j,\alpha')(\alpha,i')(\alpha,\alpha')^{\mathrm{op}} - (\alpha,j')(i,\alpha')(\alpha,\alpha')^{\mathrm{op}},
\end{equation}
where the sum varies over all arrows $\alpha: i \to j$ in $Q$ and $\alpha': i' \to j'$ in $Q'$.

\begin{example}
    Consider the quivers
    \[
    Q = \begin{tikzcd}
	1 & 2 & 3
	\arrow[from=1-1, to=1-2]
	\arrow[from=1-3, to=1-2]
\end{tikzcd} \quad \textrm{and} \quad Q' = \begin{tikzcd}
	\color{blue}\boxed{1} & 2 & 3 & 4.
	\arrow[from=1-2, to=1-1]
	\arrow[from=1-3, to=1-2]
    \arrow[from=1-4, to=1-3]
\end{tikzcd}
    \]
The algebra $A = kQ \otimes kQ'$ is the path algebra of the quiver
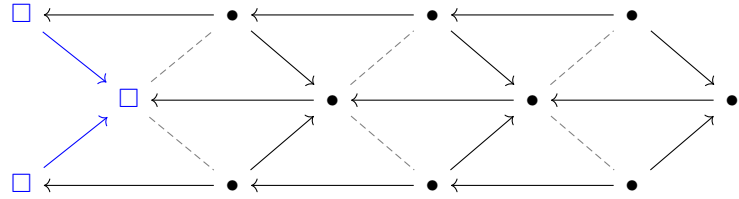
\[\begin{tikzcd}
	{\color{blue}\square} && \bullet && \bullet && \bullet & \\
	& {\color{blue}\square} && \bullet && \bullet && \bullet \\
	{\color{blue}\square} && \bullet && \bullet && \bullet
	\arrow[color=blue, from=1-1, to=2-2]
	\arrow[from=1-3, to=1-1]
	\arrow[from=1-3, to=2-4]
	\arrow[from=1-5, to=1-3]
	\arrow[from=1-5, to=2-6]
	\arrow[from=1-7, to=1-5]
	\arrow[from=1-7, to=2-8]
    \arrow[color=gray, dashed, no head, from=2-2, to=1-3]
    \arrow[color=gray, dashed, no head, from=2-2, to=3-3]
	\arrow[from=2-4, to=2-2]
    \arrow[color=gray, dashed, no head, from=2-4, to=1-5]
    \arrow[color=gray, dashed, no head, from=2-4, to=3-5]
	\arrow[from=2-6, to=2-4]
    \arrow[color=gray, dashed, no head, from=2-6, to=1-7]
    \arrow[color=gray, dashed, no head, from=2-6, to=3-7]
	\arrow[from=2-8, to=2-6]
	\arrow[color=blue, from=3-1, to=2-2]
	\arrow[from=3-3, to=2-4]
	\arrow[from=3-3, to=3-1]
	\arrow[from=3-5, to=2-6]
	\arrow[from=3-5, to=3-3]
	\arrow[from=3-7, to=2-8]
	\arrow[from=3-7, to=3-5]
\end{tikzcd}\]
modulo the relations which make all squares commute. For $i_0 = 1$, the subalgebra $B = kQ \ten e_1(kQ')e_1$ consists of the path algebra of the frozen subquiver depicted in blue. Then $\bP_3(A,B)$ is quasi-isomorphic to the relative Ginzburg dg algebra of the ice quiver
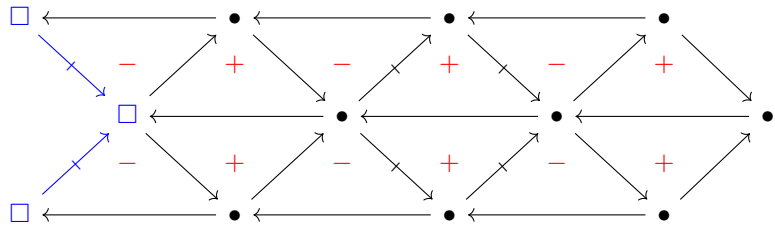
\[\begin{tikzcd}[row sep=tiny]
	{\color{blue}\square} && \bullet && \bullet && \bullet & \\
    & {\color{red}-} & {\color{red}+} & {\color{red}-} & {\color{red}+} & {\color{red}-} & {\color{red}+}\\
	& {\color{blue}\square} && \bullet && \bullet && \bullet \\
    & {\color{red}-} & {\color{red}+} & {\color{red}-} & {\color{red}+} & {\color{red}-} & {\color{red}+}\\
	{\color{blue}\square} && \bullet && \bullet && \bullet
	\arrow["\shortmid"{marking}, color=blue, from=1-1, to=3-2]
	\arrow[from=1-3, to=1-1]
	\arrow[from=1-3, to=3-4]
	\arrow[from=1-5, to=1-3]
	\arrow["\shortmid"{marking}, from=1-5, to=3-6]
	\arrow[from=1-7, to=1-5]
	\arrow[from=1-7, to=3-8]
	\arrow[from=3-2, to=1-3]
	\arrow[from=3-2, to=5-3]
	\arrow["\shortmid"{marking}, from=3-4, to=1-5]
	\arrow[from=3-4, to=3-2]
	\arrow["\shortmid"{marking}, from=3-4, to=5-5]
	\arrow[from=3-6, to=1-7]
	\arrow[from=3-6, to=3-4]
	\arrow[from=3-6, to=5-7]
	\arrow[from=3-8, to=3-6]
	\arrow["\shortmid"{marking}, color=blue, from=5-1, to=3-2]
	\arrow[from=5-3, to=3-4]
	\arrow[from=5-3, to=5-1]
	\arrow["\shortmid"{marking}, from=5-5, to=3-6]
	\arrow[from=5-5, to=5-3]
	\arrow[from=5-7, to=3-8]
	\arrow[from=5-7, to=5-5]
\end{tikzcd}\]
equipped with the potential consisting of the signed sum of all $3$-cycles, where the signs are indicated in red.
\end{example}

\begin{remark}\label{remark:signs of the potential}
    The case we are interested in is when $Q$ is a Dynkin quiver, $Q'$ is a linear orientation of a Dynkin diagram of type $\mathsf{A}_n$, and $i_0$ is the unique sink of $Q'$, as in the example above. In this case, when viewing $\bP_3(A,B)$ as a relative Ginzburg dg algebra, we can (and will) ignore the signs in the potential (\ref{eq:potential for tensor product}). Indeed, we can always find an automorphism of the path algebra of $Q \boxtimes Q'$ that changes the signs of some of the arrows of the form $(\alpha,i')$ or $(\alpha,\alpha')^{\rm{op}}$, and sends the potential (\ref{eq:potential for tensor product}) to itself but where we replaced all minus signs by plus ones. For instance, in the example, it suffices to change the sign of the marked arrows.
\end{remark}

\subsection{Truncated Auslander algebras of Dynkin quivers}\label{section:truncated Auslander algebras}

Let $Q$ be a Dynkin quiver. Let $\ind(kQ) \subseteq \modcat kQ$ be the full subcategory of indecomposable modules. We say that a full subcategory $\mc{A} \subseteq \ind(kQ)$ is \emph{left-closed} if it satisfies the following condition: for any $N \in \mc{A}$ and $M \in \ind(kQ)$ such that $\Hom_{kQ}(M,N) \neq 0$, we have $M \in \mc{A}$. Intuitively, $\mc{A}$ is obtained from $\ind(kQ)$ by removing indecomposable modules at the right side of the Auslander--Reiten quiver of $kQ$. If $X$ is the direct sum of a set of representatives for the isomorphism classes of indecomposable modules in $\mc{A}$, we call $\End_{kQ}(X)$ a \emph{left-truncated Auslander algebra}. Note that $\mc{A}$ and $\End_{kQ}(X)$ are Morita equivalent.

We recall the following description of $\mc{A}$ as a mesh category. Let $\Gamma_{\mc{A}}$ be the Gabriel quiver of $\mc{A}$, which identifies with a subquiver of the Auslander--Reiten quiver $\Gamma$ of $kQ$. Since $\mc{A}$ is left-closed, any mesh in $\Gamma$ of the form
\[\begin{tikzcd}[sep={3em,between origins}]
	& {y_1} & \\
	{\tau(x)} & \vdots & x \\
	& {y_n}
	\arrow["{\alpha_1}", from=1-2, to=2-3]
	\arrow["{\sigma(\alpha_1)}", from=2-1, to=1-2]
	\arrow["{\sigma(\alpha_n)}"', from=2-1, to=3-2]
	\arrow["{\alpha_n}"', from=3-2, to=2-3]
\end{tikzcd}\]
with $x \in \Gamma_{\mc{A}}$ is fully contained in $\Gamma_{\mc{A}}$. By \cite{Riedtmann80}, $\mc{A}$ is equivalent to the quotient of the path category of $\Gamma_{\mc{A}}$ modulo the ideal generated by the \emph{mesh relations}
\begin{equation}\label{eq:mesh relations}
    r_x = \sum_{i=1}^n\sigma(\alpha_i)\alpha_i
\end{equation}
for all $x \in \Gamma_{\mc{A}}$ that is not projective.

We will now present some lemmas concerning the global dimension of $\mc{A}$ and the action of a relative inverse dualizing bimodule on its derived category. Let us first introduce some notation.

For $M \in \mc{A}$, we write $M^{\wedge} = \restr{\Hom_{kQ}(-,M)}{\mc{A}}$ and $M^{\vee} = \restr{D\Hom_{kQ}(M,-)}{\mc{A}}$ for the corresponding representable and corepresentable $\mc{A}$-modules. Recall that $M^{\wedge}$ has a unique simple quotient $S_M$, and any simple $\mc{A}$-module is of this form. We denote by $\tau$ the Auslander--Reiten translation of $\modcat kQ$ (and not $\mc{D}^b(kQ)$).

\begin{lemma}\label{lemma:gldim for Auslander algebra}
    Suppose $\mc{A} \subseteq \ind(kQ)$ is full and left-closed. Then $\mc{A}$ has global dimension at most $2$.
\end{lemma}

\begin{proof}
    For $N \in \mc{A}$, we can naturally view $S_N$ as a simple module over $\ind(kQ)$ and, as such, it admits a projective resolution
    \[\begin{tikzcd}
	0 & {\Hom_{kQ}(-,L)} & {\Hom_{kQ}(-,M)} & {\Hom_{kQ}(-,N)} & {S_N} & 0
	\arrow[from=1-1, to=1-2]
	\arrow[from=1-2, to=1-3]
	\arrow[from=1-3, to=1-4]
	\arrow[from=1-4, to=1-5]
	\arrow[from=1-5, to=1-6]
    \end{tikzcd}\]
    where $M \to N$ is a right minimal almost split map and $L$ is its kernel (see \cite[Theorem 6.11]{AssemSimsonSkowronski}). Since $\mc{A}$ is left-closed, the indecomposable direct summands of $M$ belong to $\mc{A}$, and the same holds for $L$ if it is nonzero. Therefore, the restriction of the resolution above to $\mc{A}$ is still a projective resolution. We obtain that every simple $\mc{A}$-module has projective dimension at most $2$, hence $\mc{A}$ has global dimension at most $2$.
\end{proof}

\begin{lemma}[{\cite[Proposition 8.6]{YilinWuJacobifinite}}]\label{lemma:Omega for Auslander algebra}
    Let $\mc{A} \subseteq \ind(kQ)$ be a full and left-closed subcategory. Denote by $\mc{B}$ the full subcategory of $\mc{A}$ consisting of the indecomposable projective modules in $\mc{A}$. For $M \in \mc{A}$, we have
    \[
    M^{\wedge} \lten_{\mc{A}}\Sigma^2\Omega_{\mc{A},\mc{B}} \cong \begin{cases}
        (\tau^{-1}M)^{\wedge} &\textrm{if } \tau^{-1}M \in \mc{A},\\
        0 &\textrm{otherwise}.
    \end{cases}
    \]
\end{lemma}

\begin{proof}
    We assume first that $\mc{A}$ contains all indecomposable projective $kQ$-modules. By Remark \ref{rmk:computing Omega with Serre functor}, $M^{\wedge} \lten_{\mc{A}}\Sigma^2\Omega_{\mc{A},\mc{B}}$ is the cocone of the canonical map
    \[
     \Sigma^2\mathbb{S}_{\mc{A}}^{-1}(M^{\wedge})\longrightarrow\Sigma^2\mathbb{S}_{\mc{B}}^{-1}(\restr{M^{\wedge}}{\mc{B}}) \lten_{\mc{B}} \mc{A}.
    \]
    We adapt the argument given in \cite[Proposition 8.6]{YilinWuJacobifinite} to compute it. Choose an injective resolution
    \[\begin{tikzcd}
	0 & M & {I_0} & {I_1} & 0
	\arrow[from=1-1, to=1-2]
	\arrow[from=1-2, to=1-3]
	\arrow[from=1-3, to=1-4]
	\arrow[from=1-4, to=1-5]
    \end{tikzcd}\]
    of the $kQ$-module $M$. If we denote by $P_i = \mathbb{S}_{kQ}^{-1}(I_i)$ the projective module corresponding to $I_i$, then \cite{YilinWuJacobifinite} shows that
    \[
    \Sigma^2\mathbb{S}_{\mc{B}}^{-1}(\restr{M^{\wedge}}{\mc{B}}) \lten_{\mc{B}} \mc{A} \quad \cong \quad 0\longrightarrow P_0^{\wedge} \longrightarrow P_1^{\wedge} \longrightarrow 0
    \]
    where the complex on the right has $P_0^{\wedge}$ in degree $-2$. To compute $\Sigma^2\mathbb{S}_{\mc{A}}^{-1}(M^{\wedge})$, we need an injective resolution of $M^{\wedge}$. It is shown in \cite{YilinWuJacobifinite} that the $\ind(kQ)$-module $\Hom_{kQ}(-,M)$ has the injective resolution
    \[\begin{tikzcd}[column sep=small]
	0 & {\Hom(-,M)} & {D\Hom(P_0,-)} & {D\Hom(P_1,-)} & {D\Hom(\tau^{-1}M,-)} & 0.
	\arrow[from=1-1, to=1-2]
	\arrow[from=1-2, to=1-3]
	\arrow[from=1-3, to=1-4]
	\arrow[from=1-4, to=1-5]
	\arrow[from=1-5, to=1-6]
\end{tikzcd}\]
Notice that the restriction of the last term to $\mc{A}$ vanishes if $\tau^{-1}M \not\in \mc{A}$ since $\mc{A}$ is left-closed. Restricting to $\mc{A}$ and forgetting the last term if $\tau^{-1}M \not\in \mc{A}$, we obtain an injective resolution for $M^{\wedge}$, whence
\[
\Sigma^2\mathbb{S}_{\mc{A}}(M^{\wedge}) \quad \cong \quad 0 \longrightarrow P_0^{\wedge} \longrightarrow P_1^{\wedge} \longrightarrow N^{\wedge} \longrightarrow 0
\]
where the complex on the right has $P_0^{\wedge}$ in degree $-2$, and where $N = \tau^{-1}M$ if $\tau^{-1}M \in \mc{A}$ or $N = 0$ otherwise. The canonical map then becomes the evident map between the complexes above, and we conclude that $M^{\wedge} \lten_{\mc{A}}\Sigma^2\Omega_{\mc{A},\mc{B}} \cong N^{\wedge}$, as claimed.

Let us prove the general case. Assume that the left-closed subcategory $\mc{A} \subseteq \ind(kQ)$ does not contain an indecomposable projective module $P$ associated with a vertex $i \in Q_0$. Then any representation $M \in \mc{A}$ must vanish at the vertex $i$, otherwise we would have a nonzero map $P \to M$. Thus, we can view $\mc{A}$ as a subcategory of $\ind(kQ')$ where $Q'$ is the full subquiver of $Q$ whose vertices correspond to the indecomposable projective modules in $\mc{A}$. Note that $Q'$ is not necessarily connected, but its connected components are Dynkin quivers. We deduce that $\mc{A}$ is a disjoint union of subcategories where each one is equivalent to a left-closed subcategory of $\ind(kQ'')$ (for some Dynkin quiver $Q''$) that contains all indecomposable projective $kQ''$-modules. We conclude the proof by applying the previous argument to each component of this disjoint union.
\end{proof}

\begin{lemma}\label{lemma:adjoint action of Omega for Auslander algebra}
Let $\mc{A} \subseteq \ind(kQ)$ be a full and left-closed subcategory. Denote by $\mc{B}$ the full subcategory of $\mc{A}$ consisting of the indecomposable projective modules in $\mc{A}$. For $M \in \mc{A}$, we have
\[
\RHom_{\mc{A}}(\Sigma^2\Omega_{\mc{A},\mc{B}}, S_M) \cong \begin{cases}
    S_{\tau M} &\textrm{if } M \not\in \mc{B},\\
    0 &\textrm{otherwise}.
\end{cases}
\]
\end{lemma}

\begin{proof}
    By Remark \ref{rmk:computing Omega with Serre functor}, $\RHom_{\mc{A}}(\Sigma^2\Omega_{\mc{A},\mc{B}}, S_M)$ is the cone of the canonical map
    \[
    \Sigma^{-2}\RHom_{\mc{B}}(\mc{A},\mathbb{S}_{\mc{B}}(\restr{S_M}{\mc{B}})) \longrightarrow \Sigma^{-2}\mathbb{S}_{\mc{A}}(S_M).
    \]
    If $M \not\in \mc{B}$, that is, if $M$ is not a projective $kQ$-module, then $\restr{S_M}{\mc{B}} = 0$ and we are reduced to computing the term on the right. If
    \[
    0 \longrightarrow \tau M \longrightarrow L_1 \oplus \dotsb \oplus L_m \longrightarrow M \longrightarrow 0
    \]
    is the almost split sequence ending at $M$, where each $L_i$ is indecomposable, then $\tau M, L_i \in \mc{A}$ because $\mc{A}$ is left-closed. Thus, $S_M$ has a projective resolution given by
    \[
    0 \longrightarrow (\tau M)^{\wedge} \longrightarrow L_1^{\wedge} \oplus \dotsb \oplus L_m^{\wedge} \longrightarrow M^{\wedge} \longrightarrow S_M \longrightarrow 0
    \]
    and we deduce that $\Sigma^{-2}\mathbb{S}_{\mc{A}}(S_M)$ is the complex
    \[
    0 \longrightarrow (\tau M)^{\vee} \longrightarrow L_1^{\vee} \oplus \dotsb \oplus L_m^{\vee} \longrightarrow M^{\vee} \longrightarrow 0,
    \]
    where $(\tau M)^{\vee}$ is in degree zero. By \cite[Theorem 6.11]{AssemSimsonSkowronski}, this complex is concentrated in degree zero, where its cohomology is $S_{\tau M}$. We conclude that $\Sigma^{-2}\mathbb{S}_{\mc{A}}(S_M) \cong S_{\tau M}$, as desired.

    Now, assume $M \in \mc{B}$. Let $L_1 \oplus \dotsb \oplus L_m$ be the radical of $M$, where each $L_i$ is indecomposable. By \cite[Theorem 6.11]{AssemSimsonSkowronski} and the fact that $\mc{A}$ is left-closed, a projective resolution of $S_M$ is given by
    \[
    0 \longrightarrow L_1^{\wedge} \oplus \dotsb \oplus L_m^{\wedge} \longrightarrow M^{\wedge} \longrightarrow S_M \longrightarrow 0.
    \]
    Since $kQ$ is hereditary, each $L_i$ is again projective and the restriction of the resolution above to $\mc{B}$ is a projective resolution of $\restr{S_M}{\mc{B}}$. Using that $\RHom_{\mc{B}}(\mc{A},-)$ preserves corepresentable modules, we see that both $\Sigma^{-2}\RHom_{\mc{B}}(\mc{A},\mathbb{S}_{\mc{B}}(\restr{S_M}{\mc{B}}))$ and $\Sigma^{-2}\mathbb{S}_{\mc{A}}(S_M)$ are given by the complex
    \[
    0 \longrightarrow L_1^{\vee} \oplus \dotsb \oplus L_m^{\vee} \longrightarrow M^{\vee} \longrightarrow 0,
    \]
    where $M^{\vee}$ is in degree $2$, and the canonical map between them is the identity. We obtain that $\RHom_{\mc{A}}(\Sigma^2\Omega_{\mc{A},\mc{B}}, S_M) = 0$, finishing the proof.
\end{proof}

\section{Properness of the relative Ginzburg algebra}\label{section:properness of the Ginzburg algebra}

As in Section \ref{section:initial monoidal seed}, let $\underline{w}_0$ be a reduced expression of the longest element and let $[a,b]$ be a finite integer interval. We let
$\bG^{[a, b]}\left(\un{w}_0\right)$ and $\bGh^{[a, b]}\left(\un{w}_0\right)$ denote the non-completed and completed relative Ginzburg dg algebras associated with the ice quiver with potential $(Q^{[a,b]}(\underline{w}_0),F^{[a,b]}(\underline{w}_0),W^{[a,b]}(\underline{w}_0))$. We pose the following conjecture.

\begin{conjecture}\label{conj: ginzburg proper}
    The relative Ginzburg dg algebras $\bG^{[a,b]}(\underline{w}_0)$ and $\bGh^{[a,b]}(\underline{w}_0)$ are proper.
\end{conjecture}

Our main goal in this section is to prove the conjecture for when $\underline{w}_0$ is adapted to an orientation of $\Delta$. The key observation is that, for certain intervals $[a,b]$, $Q^{[a,b]}(\underline{w}_0)$ is a triangle product as in Section \ref{section:triangle products and green sequences}. In this so-called ``regular'' case, we reduce the conjecture to the case when $\Delta = \mathsf{A}_1$, for which we can easily verify it directly. For arbitrary intervals, we exploit the fact that $\underline{w}_0$ is adapted to decompose $Q^{[a,b]}(\underline{w}_0)$ into two subquivers: a regular one and another one isomorphic to the Gabriel quiver of a truncated Auslander algebra as in Section \ref{section:truncated Auslander algebras}. With the aid of Proposition \ref{proposition:gluing of bimodules}, we settle the conjecture in the adapted case. In fact, we prove the stronger Theorem \ref{thm:coincidence of extensions}, which we will need for proving Theorem \ref{thm:coincidence of Lambda matrices}. We conclude this section by exploring what happens when we apply braid and commutation moves to $\underline{w}_0$, following the ideas in \cite[Section 6.1]{ContuMonoidalAdditive}. This allows us to obtain more cases of the conjecture.

\subsection{The regular case}\label{section:regular case is proper}

For $\ell \geq 1$, let $(Q^{\circ}_{\ell},F^{\circ}_{\ell})$ be the following ice quiver:
\[\begin{tikzcd}
	\color{blue}\boxed{1} & 2 & 3 & \dotsb & \ell.
	\arrow[from=1-2, to=1-1]
	\arrow[from=1-3, to=1-2]
	\arrow[from=1-4, to=1-3]
	\arrow[from=1-5, to=1-4]
\end{tikzcd}\]
We remark that it is isomorphic to $(Q^{[a,b]}(\underline{w}_0),F^{[a,b]}(\underline{w}_0))$ if $\Delta = \mathsf{A}_1$ and $\ell = b-a+1$. Define 
\begin{equation}\label{eq:definition Al and Bl}
A^{\circ}_{\ell} = kQ^{\circ}_{\ell} \text{ and } B^{\circ}_{\ell} = e_1A^{\circ}_{\ell}e_1 \cong k
\end{equation}
Note that $B^{\circ}_{\ell}$ is the subalgebra corresponding to the frozen vertex.

For $1 \leq x \leq y \leq \ell$, denote by $V_{[x,y]}$ the indecomposable $A^{\circ}_{\ell}$-module concentrated on the vertices $x$, $x+1$, $\dots$, $y$ of $Q^{\circ}_{\ell}$. We also extend the notation for arbitrary $x,y \leq \ell$, where we interpret $V_{[x,y]}$ as $V_{[1,y]}$ if $x < 1$ and as zero if $y < 1$ or $x > y$. We write $V_{[x]}$ for $V_{[x,x]}$.

\begin{lemma}\label{lemma:action of Omega in type A1}
    For $1 \leq x \leq y \leq \ell$, we have an isomorphism
    \[
    V_{[x,y]} \lten_{A^{\circ}_{\ell}} \Sigma\Omega_{A^{\circ}_{\ell},B^{\circ}_{\ell}} \cong V_{[x-1,y-1]}
    \]
    in $\mc{D}^b(A^{\circ}_{\ell})$.
\end{lemma}

\begin{proof}
    Denote $A = A^{\circ}_{\ell}$ and $B = B^{\circ}_{\ell}$. By Remark \ref{rmk:computing Omega with Serre functor}, $V_{[x,y]} \lten_A \Sigma\Omega_{A,B}$ is the cocone of the canonical map
    \[
    \Sigma\mathbb{S}_A^{-1}(V_{[x,y]}) \longrightarrow \Sigma\mathbb{S}^{-1}_B(\restr{V_{[x,y]}}{B}) \lten_B A. 
    \]
    If $x > 1$, then the restriction of $V_{[x,y]}$ to $B$ is zero, and we are reduced to computing the term on the left. Knowing that $\Sigma\mathbb{S}_A^{-1}$ is the inverse Auslander--Reiten translation on $\mc{D}^b(A)$ and that $A$ is the path algebra of a linearly oriented Dynkin quiver of type $\mathsf{A}_{\ell}$, one easily computes that $\Sigma\mathbb{S}_A^{-1}(V_{[x,y]}) \cong V_{[x-1,y-1]}$. For the case $x = 1$, we use the short exact sequence
    \[
    0 \longrightarrow V_{[2,y]} \longrightarrow V_{[1,y]} \longrightarrow V_{[1]} \longrightarrow 0.
    \]
    By the previous case, $V_{[2,y]} \lten_A \Sigma\Omega_{A,B} \cong V_{[1,y-1]} =  V_{[x-1,y-1]}$, so it suffices to show that $V_{[1]} \lten_A \Sigma\Omega_{A,B} = 0$ to conclude the proof. Indeed, an easy computation shows that the canonical map above identifies with the identity map of $\Sigma V_{[1,\ell]}$, whose cocone is zero.
\end{proof}

\begin{corollary}\label{corollary:type A1 is proper}
The relative inverse dualizing bimodule $\Omega_{A^{\circ}_{\ell},B^{\circ}_{\ell}}$ is nilpotent.
\end{corollary}

\begin{proof}
Denote $A = A^{\circ}_{\ell}$ and $B = B^{\circ}_{\ell}$. We equivalently prove that $\Sigma \Omega_{A,B}$ is nilpotent. Let $G: \mc{D}^b(A) \to \mc{D}^b(A)$ be the functor $G = - \lten_A \Sigma\Omega_{A,B}$. By Lemma \ref{lemma:action of Omega in type A1}, $G^{\ell}$ vanishes on all simple $A$-modules, hence it vanishes on $\mc{D}^b(A)$. We conclude that $(\Sigma\Omega_{A,B})^{\lten_A\ell} \cong G^{\ell}(A) \cong 0$, as desired.
\end{proof}

\begin{corollary}\label{corollary:regular case is proper}
    Let $Q$ be a Dynkin quiver. Write $A = kQ \ten A^{\circ}_{\ell}$ and $B = kQ \ten B^{\circ}_{\ell}$, where $A^{\circ}_{\ell}$ and $B^{\circ}_{\ell}$ are as in \eqref{eq:definition Al and Bl}. Then $\Omega_{A,B}$ is nilpotent. Consequently, $\bP_3(A,B)$ is proper.
\end{corollary}

\begin{proof}
    The first statement follows from Corollaries \ref{corollary:Omega is nilpotent for tensor product} and \ref{corollary:type A1 is proper}. The second follows from Lemma \ref{lemma:Omega nilpotent implies proper}.
\end{proof}

For the next result, we denote by $\tau = \mathbb{S}_{kQ}\Sigma^{-1}$ the Auslander--Reiten translation on the bounded derived category of the Dynkin quiver $Q$.

\begin{corollary}\label{corollary:action of Omega for tensor product}
    Let $Q$ be a Dynkin quiver. Write $A = kQ \ten A^{\circ}_{\ell}$ and $B = kQ \ten B^{\circ}_{\ell}$, where $A^{\circ}_{\ell}$ and $B^{\circ}_{\ell}$ are as in \eqref{eq:definition Al and Bl}. For $M \in \mc{D}^b(kQ)$ and $1 \leq x \leq y \leq \ell$, we have an isomorphism
    \[
    (M \ten V_{[x,y]}) \lten_A \Sigma^2\Omega_{A,B} \cong \tau^{-1}M \ten V_{[x-1,y-1]}.
    \]
    in $\mc{D}^b(A)$.
\end{corollary}

\begin{proof}
    By Lemma \ref{lemma:inverse dualizing bimodule for tensor product}, we have $\Sigma^2\Omega_{A,B} \cong \Sigma\Omega_{kQ} \ten \Sigma\Omega_{A^{\circ}_{\ell},B^{\circ}_{\ell}}$. Therefore,
    \[
    (M \ten V_{[x,y]}) \lten_A \Sigma^2\Omega_{A,B} \cong (M \lten_{kQ} \Sigma\Omega_{kQ}) \ten (V_{[x,y]} \lten_{A^{\circ}_{\ell}} \Sigma\Omega_{A^{\circ}_{\ell},B^{\circ}_{\ell}}).
    \]
    The result then follows by Remark \ref{rmk:absolute Omega gives Serre functor} and Lemma \ref{lemma:action of Omega in type A1}.
\end{proof}

\subsection{Coincidence of extensions in the adapted case} Let $(Q,F,W)$ and $(Q',F',W')$ be two ice quivers with potential. Let $\bG = \bG(Q,F,W)$ and $\bG' = \bG(Q',F',W')$ be the associated relative Ginzburg dg algebras. Let $X \subseteq Q_0$ and $X' \subseteq Q_0'$ be subsets such that there is a bijection $f: X \to X'$. We say that $\bG$ and $\bG'$ \emph{agree on extensions} (along $f$) if we have an isomorphism
\[
\RHom_{\bG}(e_s\bG,e_t\bG) \cong \RHom_{\bG'}(e_{f(s)}\bG',e_{f(t)}\bG')
\]
in the derived category of vector spaces for all $s,t \in X$.

\begin{remark}
    In what follows, if $Q_0$ and $Q_0'$ are both defined as subsets of a common set, we will take $X = X' = Q_0 \cap Q_0'$ and $f = \id$. This applies in particular when we want to compare extensions of $\bG^{[a,b]}(\underline{w}_0)$ and $\bG^{[a',b']}(\underline{w}_0)$ for different values of $a$, $a'$, $b$ and $b'$.
\end{remark}

The following is the main result of this subsection.

\begin{theorem}\label{thm:coincidence of extensions}
    Suppose $\underline{w}_0$ is adapted to an orientation of $\Delta$. For any $a_1, a_2, b \in \Z$ with $a_1,a_2 \leq b$, the relative Ginzburg dg algebras $\bG^{[a_1,b]}(\underline{w}_0)$ and $\bG^{[a_2,b]}(\underline{w}_0)$ agree on extensions.
\end{theorem}

\begin{corollary}\label{cor:adapted words give proper Ginzburgs}
    Suppose $\underline{w}_0$ is adapted to an orientation of $\Delta$. For any $a,b \in \Z$ with $a \leq b$, both $\bG^{[a,b]}(\underline{w}_0)$ and $\bGh^{[a,b]}(\underline{w}_0)$ are proper dg algebras.
\end{corollary}

\begin{proof}[Proof of \Cref{cor:adapted words give proper Ginzburgs}]
    By Lemma \ref{lemma:identification of HL with GLS quivers} (and Remark \ref{rmk:reflecting height function gives all cases}), all cycles in $W^{[a,b]}(\underline{w}_0)$ have length $3$, so it is enough to prove that $\bG  = \bG^{[a,b]}(\underline{w}_0)$ is proper by Lemma \ref{lemma:complete/noncomplete Ginzburg are qiso when Jacobi-finite} and Remark \ref{rmk:defining an Adams grading}. Note that
    \[
    H^n(\bG) \cong \bigoplus_{a\leq s,t\leq b}e_tH^n(\bG)e_s \cong \bigoplus_{a\leq s,t\leq b} \Ext^n_{\bG}(e_s\bG,e_t\bG).
    \]
    for any $n \in \Z$. By Theorem \ref{thm:coincidence of extensions}, $\bG$ agrees on extensions with $\bG^{[a',b]}(\underline{w}_0)$ for any $a' \leq b$. Thus, by the formula above for the cohomology of $\bG$, it suffices to find $a' \leq a$ such that $\bG^{[a',b]}(\underline{w}_0)$ is proper. We take $a' = b - m \cdot  2l(w_0) + 1$ for some $m \geq 0$ such that $a'\leq a$. Since the length of the interval $[a',b]$ is a multiple of $2l(w_0)$ and $\underline{w}_0$ is adapted, all rows of $Q^{[a',b]}(\underline{w}_0)$ have the same number of vertices by \cite[Corollary 2.20]{Bedard}. Denote this number by $\ell$. By Lemma \ref{lemma:identification of HL with GLS quivers}, we conclude that $Q^{[a',b]}(\underline{w}_0)$ is isomorphic to the triangle product between an orientation of $\Delta$ and $Q^{\circ}_{\ell}$. By the discussion in Section \ref{section:tensor products} and Corollary \ref{corollary:regular case is proper}, $\bG^{[a',b]}(\underline{w}_0)$ is proper, as desired.
\end{proof}

The proof of \Cref{thm:coincidence of extensions} consists of three main steps: Lemma \ref{lemma:short word agree on extensions with word for w0}, Proposition \ref{prop:Gamma and its extended version agree}, and Proposition \ref{prop:truncating the extended Ginzburg algebra}. 

{\it Convention:} From now on, we fix an adapted expression $\underline{w}_0$ and omit it from the notation. We also fix $a,b \in \Z$ with $a \leq b$. We use throughout that $Q^{[-\infty,b]}$ is isomorphic to a subquiver of $Q_{\rm{HL}}$ by Lemma \ref{lemma:identification of HL with GLS quivers}.

\begin{lemma}\label{lemma:short word agree on extensions with word for w0}
    Let $N = l(w_0)$ be the length of $w_0$. If $b - a + 1< N$, then $\bG^{[a,b]}$ and $\bG^{[b-N+1,b]}$ agree on extensions.
\end{lemma}

\begin{proof}
   Denote by $\overline{Q}^{[a,b]}$ the quiver obtained from $Q^{[a,b]}$ by removing the horizontal arrows. Since $\underline{w}_0$ is adapted and since the length of the interval $[a,b]$ is at most $N$, $\overline{Q}^{[a,b]}$ is isomorphic to the Gabriel quiver of a left-truncated Auslander algebra $A$ in such a way that each row of $\overline{Q}^{[a,b]}$ corresponds to an orbit of the Auslander--Reiten translation in the associated left-closed subcategory $\mc{A}$ (see \cite{Bedard,OhSuh19a}). Note that $A$ is the quotient of $k\overline{Q}^{[a,b]}$ by the mesh relations. Let $B \subseteq A$ be the subalgebra that corresponds to the indecomposable projective modules in $\mc{A}$. By Theorem \ref{theorem:Ginzburg algebras as completion}, $\bG = \bG^{[a,b]}$ is quasi-isomorphic to $\bP_3(A,B)$. Consequently,
   \[
    \RHom_{\bG}(e_s\bG,e_t\bG) \cong \bigoplus_{r \geq 0}\RHom_A(e_sA,F^r(e_tA))
   \]
    for any $s,t \in [a,b]$, where $F$ is the functor $- \lten_A \Sigma^2\Omega_{A,B}$. By Lemma \ref{lemma:Omega for Auslander algebra}, $F(e_tA)$ is isomorphic to $e_{t^+}A$ if $t^+ \leq b$, or to zero otherwise. Moreover, we have
    \[
    \RHom_A(e_sA,e_{t'}A) = \Hom_A(e_sA,e_{t'}A) \cong e_{t'}Ae_s,
    \]
    where this last space is the quotient of the space of linear combinations of all paths from $s$ to $t'$ in $\overline{Q}^{[a,b]}$ by the mesh relations.

   Now, denote $a' = b-N+1$. Replacing the interval $[a,b]$ by $[a',b]$ in the previous paragraph, we define similarly a quiver $\overline{Q}^{[a',b]}$ and algebras $B' \subseteq A'$ such that $\bG' = \bG^{[a',b]}$ is quasi-isomorphic to $\bP_3(A',B')$. The same argument also shows that $\RHom_{\bG'}(e_s\bG',e_t\bG')$ for $s,t \in [a',b]$ is concentrated in degree zero and its zeroth cohomology is the quotient by the mesh relations of the space of linear combinations of all paths from $s$ to a vertex $t' \in [t,b]$ in the same row of $t$. To conclude, note that $\overline{Q}^{[a,b]}$ is a full subquiver of $\overline{Q}^{[a',b]}$ and any path in $\overline{Q}^{[a',b]}$ whose start and end points are in $\overline{Q}^{[a,b]}$ is completely contained in $\overline{Q}^{[a,b]}$. It follows immediately from these properties and the discussion above that $\bG$ and $\bG'$ agree on extensions.
\end{proof}

For $j \in \Delta_0$, let $m_j$ be the number of indices $s \in [a,b]$ such that $i_s = j$. We define the \emph{regular width} of $Q^{[a,b]}$ as
\[
\ell = \ell(Q^{[a,b]}) = \min\{m_j \mid j \in \Delta_0\}.
\]
Suppose $\ell \geq 1$. For $\tilde{\ell} \geq \ell$, we define the \emph{$\tilde{\ell}$-extension of $Q^{[a,b]}$} as the full subquiver $Q^{[a,b],\tilde{\ell}}$ of $Q^{[-\infty,b]}$ whose vertices are the integers in $[a,b]$ and all $\tilde{\ell} - \ell$ vertices immediately to the left of a frozen vertex of $Q^{[a,b]}$. More precisely, for any $s \in F^{[a,b]}$, the quiver $Q^{[a,b],\tilde{\ell}}$ also includes the vertices $s_1$, $s_2$, $\dots$, $s_{\tilde{\ell} - \ell}$, where $s_1 = s^-$ and $s_{j+1} = s_j^-$ for $1 \leq j < \tilde{\ell} - \ell$. We define the frozen subquiver $F^{[a,b],\tilde{\ell}} \subseteq Q^{[a,b],\tilde{\ell}}$ as the full subquiver of all vertices $s$ such that $s^-$ does not belong to $Q^{[a,b],\tilde{\ell}}$, and we equip $Q^{[a,b],\tilde{\ell}}$ with the potential $W^{[a,b],\tilde{\ell}}$ consisting of the sum of all $3$-cycles. This defines an ice quiver with potential and a relative Ginzburg dg algebra $\bG^{[a,b],\tilde{\ell}} = \bG(Q^{[a,b],\tilde{\ell}},F^{[a,b],\tilde{\ell}},W^{[a,b],\tilde{\ell}})$. We remark that $\bG^{[a,b]} = \bG^{[a,b],\tilde{\ell}}$ for $\tilde{\ell} = \ell$.

The next step in the proof of Theorem \ref{thm:coincidence of extensions} is the following.
\begin{proposition}\label{prop:Gamma and its extended version agree}
    Assume $\ell = \ell(Q^{[a,b]}) \geq 1$ and let $\tilde{\ell}_1,\tilde{\ell}_2 \geq \ell$ be two integers. Then $\bG^{[a,b],\tilde{\ell}_1}$ and $\bG^{[a,b],\tilde{\ell}_2}$ agree on extensions.
\end{proposition}

The main tool to prove the result above will be Proposition \ref{proposition:gluing of bimodules}. To apply it, we need to introduce some categories attached to $Q^{[a,b],\tilde{\ell}}$. We refer the reader to Example \ref{example:decomposition of adapted quiver into regular and Auslander} for a concrete instance of the constructions we now define.

From now on, assume $\ell \geq 1$ and fix $\tilde{\ell} \geq \ell$. We define $Q_{\mathrm{reg}}^{[a,b],\tilde{\ell}}$ to be the full subquiver of $Q^{[a,b],\tilde{\ell}}$ which contains the $\tilde{\ell}$ leftmost vertices of each row. We call it the \emph{regular part} of $Q^{[a,b],\tilde{\ell}}$, and its vertices are called \emph{regular}. We define the \emph{Auslander part} of $Q^{[a,b],\tilde{\ell}}$ to be its full subquiver consisting of the vertices that are not regular and the rightmost regular vertex of each row. We call these vertices the \emph{Auslander vertices}. Note that the Auslander part does not depend on $\tilde{\ell}$, so we will denote it by $Q_{\mathrm{Aus}}^{[a,b]}$. We remark that the intersection of $Q_{\mathrm{reg}}^{[a,b],\tilde{\ell}}$ with $Q_{\mathrm{Aus}}^{[a,b]}$ is isomorphic to $F^{[a,b],\tilde{\ell}}$ and $F^{[a,b]}$.

\begin{lemma}\label{lemma:Auslander part is well defined}
    Let $\overline{Q}_{\mathrm{Aus}}^{[a,b]}$ be the subquiver of $Q_{\mathrm{Aus}}^{[a,b]}$ obtained by removing the horizontal arrows. There is an orientation $\vec{\Delta}^{[a,b]}$ of $\Delta$ and a left-closed full subcategory $\mc{A}_{\mathrm{Aus}}^{[a,b]} \subseteq \ind(k\vec{\Delta}^{[a,b]})$ whose Gabriel quiver is isomorphic to $\overline{Q}_{\mathrm{Aus}}^{[a,b]}$. We can take this isomorphism in such a way that the row of $\overline{Q}_{\mathrm{Aus}}^{[a,b]}$ associated with $j \in \Delta_0$ corresponds to the orbit under the Auslander--Reiten translation of the indecomposable projective $k\vec{\Delta}^{[a,b]}$-module associated with $j$.
\end{lemma}

\begin{proof}
    Let $\ov{Q}_{\rm{HL}}$ be the subquiver of $Q_{\rm{HL}}$ obtained by removing the arrows of the form $(i,p) \to (i,p-2)$. For a height function $\xi: \Delta_0 \to \Z$, the Auslander--Reiten quiver of $kQ_{\xi}$ is isomorphic to the full subquiver of $\ov{Q}_{\rm{HL}}$ with vertices of the form $(i,\xi(i) + 2r)$ for $i \in \Delta_0$ and $r \in \Z$ such that $0 \leq 2r < \xi(i^*) - \xi(i) + h$, where $h$ is the Coxeter number of $\Delta$ (see \cite[Section 2.3]{FujitaOh}). Moreover, by fixing $i$ and varying $r$, we obtain the vertices corresponding to the orbit under the Auslander--Reiten translation of the indecomposable projective $kQ_{\xi}$-module associated with $i$. We remark that $\xi(i^*) - \xi(i) + h$ is an even integer.

    Since $\underline{w}_0$ is adapted, we can identify $\overline{Q}_{\mathrm{Aus}}^{[a,b]}$ as a full subquiver of $\ov{Q}_{\rm{HL}}$ by Lemma \ref{lemma:identification of HL with GLS quivers}. Moreover, the full subquiver consisting of the leftmost vertices of each row in $\overline{Q}_{\mathrm{Aus}}^{[a,b]}$ is connected and determines an orientation $\vec{\Delta}^{[a,b]}$ of $\Delta$. A similar statement holds if we replace ``leftmost'' by ``rightmost''. Choose a height function $\xi$ such that $Q_{\xi} = \vec{\Delta}^{[a,b]}$ and $(i,\xi(i))$ is the leftmost vertex of the $i$-th row of $\overline{Q}_{\mathrm{Aus}}^{[a,b]}$. Let $(i,t_i)$ denote the rightmost vertex of the $i$-th row. To prove the lemma, we need to show that $t_i < \xi(i^*) + h$. By construction, note that there is $i_0 \in \Delta_0$ whose corresponding row in $\overline{Q}_{\mathrm{Aus}}^{[a,b]}$ has precisely one vertex. We may assume that $\xi(i_0) = 0$. If $d(i,i_0)$ denotes the minimal distance in $\Delta$ between $i$ and $i_0$, then notice that $t_i \leq d(i,i_0)$ since the subquiver of $\overline{Q}_{\mathrm{Aus}}^{[a,b]}$ on the vertices $(i,t_i)$ is connected. Similarly, $\xi(i) \geq -d(i,i_0)$. Therefore, we have $t_i - \xi(i^*) \leq d(i,i_0) + d(i^*,i_0)$. By direct inspection, one sees that the inequality $d(i,j) + d(i^*,j) < h$ holds for all vertices $i$ and $j$ in an arbitrary Dynkin diagram $\Delta$ of type $\mathsf{ADE}$. This concludes the proof.
\end{proof}

We note that the frozen subquiver $F^{[a,b],\tilde{\ell}}$ is isomorphic to $\vec{\Delta}^{[a,b]}$ in the lemma above. Since $\underline{w}_0$ is adapted, it follows from Lemma \ref{lemma:identification of HL with GLS quivers} that $Q_{\mathrm{reg}}^{[a,b],\tilde{\ell}}$ and is canonically isomorphic to the triangle product $\vec{\Delta}^{[a,b]} \boxtimes Q^{\circ}_{\tilde{\ell}}$, where $Q^{\circ}_{\tilde{\ell}}$ is the quiver from Section \ref{section:regular case is proper}. This canonical isomorphism sends a row associated with a vertex $j \in \Delta_0$ to the row in the triangle product corresponding to $j$. We denote by $Q_{\mathrm{tens}}^{[a,b],\tilde{\ell}}$ the subquiver of $Q_{\mathrm{reg}}^{[a,b],\tilde{\ell}}$ that corresponds to $\vec{\Delta}^{[a,b]} \ten Q^{\circ}_{\tilde{\ell}}$ under this isomorphism.

Let $Q_{\mc{A}}$ be the subquiver of $Q^{[a,b],\tilde{\ell}}$ that contains all vertices and whose arrows are precisely those in $\overline{Q}^{[a,b]}_{\mathrm{Aus}}$ and $Q^{[a,b],\tilde{\ell}}_{\mathrm{tens}}$. We define $\mc{A}$ as the quotient of the path category of $Q_{\mc{A}}$ by the mesh relations on the subquiver $\overline{Q}^{[a,b]}_{\mathrm{Aus}}$ (coming from Lemma \ref{lemma:Auslander part is well defined} and \eqref{eq:mesh relations}) and by the commutation relations on the subquiver $Q^{[a,b],\tilde{\ell}}_{\mathrm{tens}}$ (coming from (\ref{eq:commutation relations for tensor products})). Let $\mc{A}_1$ and $\mc{A}_2$ be the full subcategories on the regular and Auslander vertices, respectively, and denote by $\mc{C}$ their intersection. Note that
\[
\Hom(\mc{A}_1\setminus \mc{C}, \mc{A}_2) = \Hom(\mc{A}_2\setminus \mc{C}, \mc{A}_1) = 0
\]
because there are no arrows in $Q_{\mc{A}}$ from a vertex in $Q_{\mc{A}} \setminus Q_{\mathrm{tens}}^{[a,b],\tilde{\ell}} \cap \overline{Q}_{\mathrm{Aus}}^{[a,b]}$ to a vertex in $Q_{\mathrm{tens}}^{[a,b],\tilde{\ell}} \cap \overline{Q}_{\mathrm{Aus}}^{[a,b]}$. We also deduce that $\mc{A}_1$ is Morita equivalent to $k\vec{\Delta}^{[a,b]} \ten kQ^{\circ}_{\tilde{\ell}}$ (by Section \ref{section:tensor products}), and $\mc{A}_2$ is equivalent to the category $\mc{A}^{[a,b]}_{\mathrm{Aus}}$ in Lemma \ref{lemma:Auslander part is well defined} (by Section \ref{section:truncated Auslander algebras}). Finally, we define $\mc{B} \subseteq \mc{A}_1$ to be the full subcategory on the vertices of $F^{[a,b],\tilde{\ell}}$. We remark that the hypotheses of Proposition \ref{proposition:gluing of bimodules} are satisfied.

\begin{example}\label{example:decomposition of adapted quiver into regular and Auslander}
   Suppose $\Delta$ is of type $\mathsf{A}_3$ and take $\underline{w}_0 = (3,1,2,3,1,2)$. For $[a,b] = [-1,6]$, the ice quiver $(Q^{[-1,6]},F^{[-1,6]})$ is given by
   
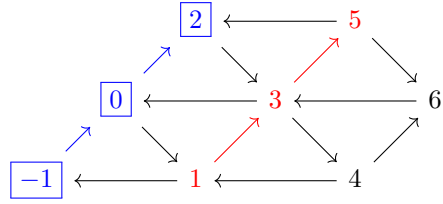
\[\begin{tikzcd}[sep={3em,between origins}]
	&& \color{blue}\boxed{2} && \color{red}5 & \\
	& \color{blue}\boxed{0} && \color{red}3 && 6 \\
	\color{blue}\boxed{-1} && \color{red}1 && 4
	\arrow[from=1-3, to=2-4]
	\arrow[from=1-5, to=1-3]
	\arrow[from=1-5, to=2-6]
	\arrow[color=blue, from=2-2, to=1-3]
	\arrow[from=2-2, to=3-3]
	\arrow[color=red, from=2-4, to=1-5]
	\arrow[from=2-4, to=2-2]
	\arrow[from=2-4, to=3-5]
	\arrow[from=2-6, to=2-4]
	\arrow[color=blue, from=3-1, to=2-2]
	\arrow[color=red, from=3-3, to=2-4]
	\arrow[from=3-3, to=3-1]
	\arrow[from=3-5, to=2-6]
	\arrow[from=3-5, to=3-3]
    \end{tikzcd}\]
    The regular width is $\ell = 2$. The red subquiver is $Q_{\mathrm{reg}}^{[-1,6],\ell} \cap Q_{\mathrm{Aus}}^{[-1,6]}$. On its left we have $Q_{\mathrm{reg}}^{[-1,6],\ell}$, consisting of the vertices $-1$, $0$, $1$, $2$, $3$ and $5$. On its right we have $Q_{\mathrm{Aus}}^{[-1,6]}$, consisting of the vertices $1$, $3$, $4$, $5$ and $6$.  The subquiver $Q_{\mathrm{tens}}^{[-1,6],\ell}$ is obtained from $Q_{\mathrm{reg}}^{[-1,6],\ell}$ by removing the arrows $0 \to 1$ and $2 \to 3$, and $\overline{Q}_{\mathrm{Aus}}^{[-1,6]}$ is obtained from $Q_{\mathrm{Aus}}^{[-1,6]}$ by removing the arrows $4 \to 1$ and $6 \to 3$.
    The quiver $Q_{\mc{A}}$ is given by
    
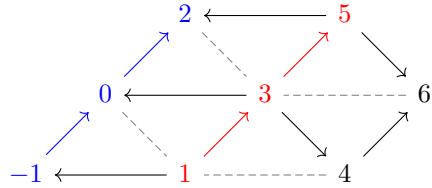
\[\begin{tikzcd}[sep={3em,between origins}]
	&& \color{blue}2 && \color{red}5 & \\
	& \color{blue}0 && \color{red}3 && 6 \\
	\color{blue}-1 && \color{red}1 && 4
    \arrow[color=gray, dashed, no head, from=1-3, to=2-4]
	\arrow[from=1-5, to=1-3]
	\arrow[from=1-5, to=2-6]
	\arrow[color=blue, from=2-2, to=1-3]
    \arrow[color=gray, dashed, no head, from=2-2, to=3-3]
	\arrow[color=red, from=2-4, to=1-5]
	\arrow[from=2-4, to=2-2]
	\arrow[from=2-4, to=3-5]
    \arrow[color=gray, dashed, no head, from=2-6, to=2-4]
	\arrow[color=blue, from=3-1, to=2-2]
	\arrow[color=red, from=3-3, to=2-4]
	\arrow[from=3-3, to=3-1]
	\arrow[from=3-5, to=2-6]
    \arrow[color=gray, dashed, no head, from=3-5, to=3-3]
    \end{tikzcd}\]
    The category $\mc{A}$ is the path category of $Q_{\mc{A}}$ modulo the relations indicated by the dashed lines: the three squares (anti)commute and the composition $1 \to 3 \to 4$ vanishes. The subcategory $\mc{B} \subseteq \mc{A}$ has the blue vertices as objects, while the objects of $\mc{C} \subseteq \mc{A}$ are the red ones.

    Below, we depict the ice quiver $(Q^{[-1,6],\tilde{\ell}},F^{[-1,6],\tilde{\ell}})$ for $\tilde{\ell} = 3$: 
    
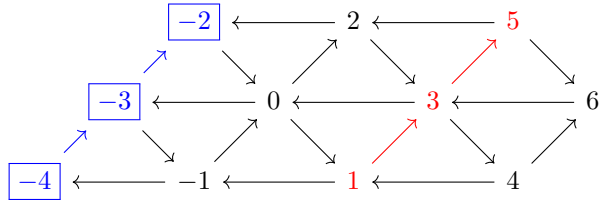
\[\begin{tikzcd}[sep={3em,between origins}]
	&& \color{blue}\boxed{-2} && 2 && \color{red}5 & \\
	& \color{blue}\boxed{-3} && 0 && \color{red}3 && 6 \\
	\color{blue}\boxed{-4} && -1 && \color{red}1 && 4
	\arrow[from=1-3, to=2-4]
	\arrow[from=1-5, to=1-3]
	\arrow[from=1-5, to=2-6]
    \arrow[from=1-7, to=1-5]
    \arrow[from=1-7, to=2-8]
	\arrow[color=blue, from=2-2, to=1-3]
	\arrow[from=2-2, to=3-3]
	\arrow[from=2-4, to=1-5]
	\arrow[from=2-4, to=2-2]
	\arrow[from=2-4, to=3-5]
    \arrow[color=red, from=2-6, to=1-7]
	\arrow[from=2-6, to=2-4]
    \arrow[from=2-6, to=3-7]
    \arrow[from=2-8, to=2-6]
	\arrow[color=blue, from=3-1, to=2-2]
	\arrow[from=3-3, to=2-4]
	\arrow[from=3-3, to=3-1]
	\arrow[from=3-5, to=3-3]
    \arrow[color=red, from=3-5, to=2-6]
    \arrow[from=3-7, to=2-8]
	\arrow[from=3-7, to=3-5]
    \end{tikzcd}\] 
    The associated quivers and categories have a similar description.
\end{example}

\begin{lemma}\label{lemma:construction of Gamma for adapted word}
    The total dg algebra of $\bP_3(\mc{A},\mc{B})$ is quasi-isomorphic to $\bG^{[a,b],\tilde{\ell}}$.
\end{lemma}

\begin{proof}
    By Lemmas \ref{lemma:global dimension of gluing} and \ref{lemma:gldim for Auslander algebra}, $\mc{A}$ has global dimension at most $2$. The statement then follows from Theorem \ref{theorem:Ginzburg algebras as completion} and the description of the categories $\mc{A}$ and $\mc{B}$ in terms of quivers and relations. See also Remark \ref{remark:signs of the potential}.
\end{proof}

We are ready to prove Proposition \ref{prop:Gamma and its extended version agree}.

\begin{proof}[Proof of Proposition \ref{prop:Gamma and its extended version agree}]
    Let $\tilde{\ell} \geq \ell$. Let $\bG = \bG^{[a,b],\tilde{\ell}}$ and $\bP = \bP_3(\mc{A},\mc{B})$ be as above. By Lemma \ref{lemma:construction of Gamma for adapted word} and the construction of $\bP$ as a tensor dg category over $\mc{A}$, we have
    \[
    \RHom_{\bG}(e_s\bG,e_t\bG) \cong \RHom_\bP(s^{\wedge},t^{\wedge}) \cong \bigoplus_{r \geq 0} \RHom_{\mc{A}}(s^{\wedge}, F^r(t^{\wedge}))
    \]
    for any two vertices $s,t \in Q^{[-\infty,b]}_0$ that belong to $Q^{[a,b],\tilde{\ell}}$, where $F$ denotes the functor $-\lten_{\mc{A}}\Sigma^2\Omega_{\mc{A},\mc{B}}$ and we write $x^{\wedge} = \mc{A}(-,x)$ for a vertex $x$ in $Q^{[a,b],\tilde{\ell}}$. To prove the proposition, it suffices to explicitly compute the complex above for fixed $s,t \in Q^{[-\infty,b]}_0$ and show that it does not depend on $\tilde{\ell}$. We will describe $F^r(t^{\wedge})$ for this computation, but first, we need some notation.

    Fix a vertex $t \in Q^{[-\infty,b]}_0$ that belongs to $Q^{[a,b],\tilde{\ell}}$. If $t$ is not Auslander, let $d \geq 2$ be the integer such that $t$ is the $d$-th regular vertex of its row, counted from right to left. If $t$ is Auslander, we set $d = 1$. We define $t_{\mathrm{Aus}}$ as the smallest Auslander vertex such that $t \leq t_{\mathrm{Aus}}$ and $i_t = i_{t_{\mathrm{Aus}}}$. Let $M_t \in \ind(k\vec{\Delta}^{[a,b]})$ be the indecomposable representation which corresponds to $t_{\mathrm{Aus}}$ under the isomorphism in Lemma \ref{lemma:Auslander part is well defined}. For an integer $s$, we will denote $s^+$ by $\tau^{-1}(s)$. Let $r_0 \geq 0$ be the largest integer such that $\tau^{-r_0}(t_{\mathrm{Aus}})$ is in $Q^{[a,b]}_{\mathrm{Aus}}$. We remark that $d$, $t_{\mathrm{Aus}}$, $M_t$ and $r_0$ do not depend on $\tilde{\ell}$.

    Recall that the subcategory $\mc{A}_1 \subseteq \mc{A}$ is Morita equivalent to $k\vec{\Delta}^{[a,b]} \ten kQ^{\circ}_{\tilde{\ell}}$. More precisely, we can view any $k\vec{\Delta}^{[a,b]} \ten kQ^{\circ}_{\tilde{\ell}}$-module as a representation of the quiver $Q^{[a,b]}_{\mathrm{tens}}$, which can in turn be viewed as an $\mc{A}_1$-module by the description of $\mc{A}_1$ as a quotient of a path category. Additionally, we identify $\mc{A}_1$-modules with $\mc{A}$-modules vanishing on $\mc{A} \setminus \mc{A}_1$ via the functor $\mu_1$, as explained in Section \ref{section:gluing bimodules}. In what follows, we use these identifications to view objects in $\mc{D}^b(k\vec{\Delta}^{[a,b]} \ten kQ^{\circ}_{\tilde{\ell}})$ as objects in $\mc{D}^b(\mc{A})$. We also denote by $\tau$ the Auslander--Reiten translation on $\mc{D}^b(k\vec{\Delta}^{[a,b]})$ and by $V_{[x,y]}$ the indecomposable $kQ^{\circ}_{\tilde{\ell}}$-module from Section \ref{section:regular case is proper}.

    We will show by induction on $r \geq 0$ that $F^r(t^{\wedge})$ admits a filtration
    \[\begin{tikzcd}[ampersand replacement=\&, column sep=small]
        	{0 = N_0} \&\& {N_1} \& \dotsb \& {N_{n-1}} \&\& {N_n} \&\& {N_{n+1} = F^r(t^{\wedge})} \\
        	\& {X_1} \&\&\&\& {X_2} \&\& {X_n}
        	\arrow[from=1-1, to=1-3]
        	\arrow[from=1-3, to=1-4]
        	\arrow[from=1-3, to=2-2]
        	\arrow[from=1-4, to=1-5]
        	\arrow[from=1-5, to=1-7]
        	\arrow[from=1-7, to=1-9]
        	\arrow[from=1-7, to=2-6]
        	\arrow[start anchor = {[xshift=-6mm]}, from=1-9, to=2-8]
        	\arrow[dashed, end anchor={[xshift=2mm]}, from=2-2, to=1-1]
        	\arrow[dashed, from=2-6, to=1-5]
        	\arrow[dashed, from=2-8, to=1-7]
    \end{tikzcd}\]
    in $\mc{D}^b(\mc{A})$ whose set $\mc{X}_r = \{X_1,\dots,X_n\}$ of successive quotients is the following.
    If $r \leq r_0$, then $\mc{X}_r$ consists of the objects
    \[
    \tau^{-r}(t_{\mathrm{Aus}})^{\wedge}, \quad \tau^{-r}M_t \ten V_{[\tilde{\ell}-x]} \ (1 \leq x \leq r), \quad \textrm{and} \quad \tau^{-r}M_t \ten V_{[\tilde{\ell}-r-d+1,\tilde{\ell}-r-1]}.
    \]
    Otherwise, if $r > r_0$, it consists of the objects
    \[
    \tau^{-r}M_t \ten V_{[\tilde{\ell}-x]} \ (r-r_0 \leq x \leq r) \quad \textrm{and} \quad \tau^{-r}M_t \ten V_{[\tilde{\ell}-r-d+1,\tilde{\ell}-r-1]}.
    \]
    Let us start with the base case $r = 0$. If $t$ is Auslander, then $t = t_{\mathrm{Aus}}$, $d = 1$ and $M_t \ten V_{[\tilde{\ell}-d+1,\tilde{\ell}-1]} = 0$, from where our claim easily follows. If $t$ is not Auslander, note that $t_{\mathrm{Aus}}$ is also regular (and thus belongs to $\mc{C}$), so $M_t$ is the indecomposable projective $k\vec{\Delta}^{[a,b]}$-module associated with $i_t \in \Delta_0$. Therefore, since there are no morphisms from $\mc{A} \setminus \mc{A}_1$ to $\mc{A}$ and from $\mc{A} \setminus \mc{C}$ to $\mc{C}$, it is not hard to see that $t^{\wedge} \cong M_t \ten V_{[\tilde{\ell}-d+1,\tilde{\ell}]}$ and $t_{\mathrm{Aus}}^{\wedge} \cong M_t \ten V_{[\tilde{\ell}]}$. In particular, we have an exact sequence
    \[
    0 \longrightarrow t_{\mathrm{Aus}}^{\wedge} \longrightarrow t^{\wedge} \longrightarrow M_t \ten V_{[\tilde{\ell}-d+1,\tilde{\ell}-1]} \longrightarrow 0
    \]
    in $\modcat\mc{A}$, whence our claim.

    Now we assume that our claim holds for some $r \geq 0$ and we prove that it also holds for $r+1$. We first compute $F(X)$ for $X \in \mc{X}_r$. We use Proposition \ref{proposition:gluing of bimodules} to obtain a distinguished triangle
    \[
    X \lten_{\mc{A}_2} \Sigma^2\Omega_{\mc{A}_2,\mc{C}} \lten_{\mc{A}_2} \mc{A} \rightarrow F(X) \rightarrow X \lten_{\mc{A}_1} \Sigma^2\Omega_{\mc{A}_1,\mc{B}} \lten_{\mc{A}_1} \mc{A} \rightarrow \Sigma  X \lten_{\mc{A}_2} \Sigma^2\Omega_{\mc{A}_2,\mc{C}} \lten_{\mc{A}_2} \mc{A}
    \]
    in $\mc{D}^b(\mc{A})$. If $X$ is of the form $\tau^{-r}M_t \ten V_{[x,y]}$ with $x,y < \tilde{\ell}$, note that $\restr{X}{\mc{A}_2} = 0$, so the triangle above gives
    \[
    F(X) \cong X \lten_{\mc{A}_1} \Sigma^2\Omega_{\mc{A}_1,\mc{B}} \lten_{\mc{A}_1} \mc{A} \cong \tau^{-(r+1)}M_t \ten V_{[x-1,y-1]},
    \]
    where the second isomorphism follows from Corollary \ref{corollary:action of Omega for tensor product} and Lemma \ref{lemma:(co)induction with gluing conditions}. This shows that $F(X) \in \mc{X}_{r+1}$. This recovers all objects of $\mc{X}_{r+1}$ except for $\tau^{-(r+1)}M_t \ten V_{[\tilde{\ell}-1]}$ (if $r \leq r_0$) and $\tau^{-(r+1)}(t_{\mathrm{Aus}})^{\wedge}$ (if $r < r_0$). When $r \leq r_0$, it remains to compute $F(X)$ for $X = \tau^{-r}(t_{\mathrm{Aus}})^{\wedge}$. In this case, observe that the vertex $\tau^{-r}(t_{\mathrm{Aus}})$ corresponds to $\tau^{-r}M_t \in \ind(k\vec{\Delta}^{[a,b]})$ under the isomorphism in Lemma \ref{lemma:Auslander part is well defined}. We then have $\restr{X}{\mc{A}_1} \cong \tau^{-r}M_t \ten V_{[\tilde{\ell}]}$, hence the distinguished triangle above is isomorphic to
    \[
     X \lten_{\mc{A}_2} \Sigma^2\Omega_{\mc{A}_2,\mc{C}} \lten_{\mc{A}_2} \mc{A} \longrightarrow F(X) \longrightarrow \tau^{-(r+1)}M_t \ten V_{[\tilde{\ell} - 1]} \longrightarrow \Sigma X \lten_{\mc{A}_2} \Sigma^2\Omega_{\mc{A}_2,\mc{C}} \lten_{\mc{A}_2} \mc{A}
    \]
    where we computed the third term by Corollary \ref{corollary:action of Omega for tensor product} and Lemma \ref{lemma:(co)induction with gluing conditions} again. Finally, Lemma \ref{lemma:Omega for Auslander algebra} implies that
    \[
     X \lten_{\mc{A}_2} \Sigma^2\Omega_{\mc{A}_2,\mc{C}} \lten_{\mc{A}_2} \mc{A} \cong \begin{cases}
         \tau^{-(r+1)}(t_{\mathrm{Aus}})^{\wedge} &\textrm{if } r < r_0,\\
         0 &\textrm{if } r = r_0.
     \end{cases}
    \]
    This gives a filtration of $F(X)$ by the remaining objects of $\mc{X}_{r+1}$ that we did not recover before. Since $F$ is a triangle functor, we conclude that the image under $F$ of the filtration of $F^r(t^{\wedge})$ by $\mc{X}_r$ gives rise to a filtration of $F^{r+1}(t^{\wedge})$ by $\mc{X}_{r+1}$, as desired. This finishes the induction.

    For $s \in Q^{[a,b],\tilde{\ell}}_0$, we can finally compute $C^r_{s,t} = \RHom_{\mc{A}}(s^{\wedge},F^r(t^{\wedge}))$. As a first step, note that all objects in $\mc{X}_r$ are concentrated in a single cohomological degree $m$, where $m \leq 0$ is the integer such that $\Sigma^m\tau^{-r}M_t$ is concentrated in degree zero. Observe that $m = 0$ if $r \leq r_0$. Therefore, $C^r_{s,t}$ is concentrated in degree $m$ and we have an isomorphism
    \[
    C^r_{s,t} \cong \bigoplus_{X \in \mc{X}_r}\RHom_{\mc{A}}(s^{\wedge},X)
    \]
    in the derived category of vector spaces. Viewing $X$ as a functor from the opposite category of $\mc{A}$ to the category of complexes of vector spaces, each term in the direct sum is given by the evaluation $X(s)$. By the definition of $\mc{X}_r$, we thus get the following description of $C^r_{s,t}$. First, assume $s$ is Auslander. If $r > r_0$, then $C^r_{s,t} = 0$. If $r \leq r_0$, then $C^r_{s,t}$ is concentrated in degree $0$ and
    \[
    \dim H^0(C^r_{s,t}) = \dim \mc{A}_2(s,\tau^{-r}(t_{\mathrm{Aus}})).
    \]
    Now, suppose $s$ is not Auslander. Let $d_s \geq 2$ be the integer such that $s$ is the $d_s$-th regular vertex of its row, counted from right to left. Then $C^r_{s,t}$ can only be nonzero if 
    \[
    \tilde{\ell}-r-d+1 \leq \tilde{\ell} - d_s + 1 \leq \tilde{\ell} - \max\{1,r-r_0\}\iff \max\{1,r-r_0\} + 1 \leq d_s \leq r + d,
    \]
    and, in this case, it is concentrated in degree $m$ and the dimension of $H^m(C^r_{s,t})$ is the multiplicity of the simple associated with $i_s \in \Delta_0$ in the $k\vec{\Delta}^{[a,b]}$-module $\Sigma^m\tau^{-r}M_t$. As we can see, in any case, $C^r_{s,t}$ admits a uniform description independent of $\tilde{\ell}$, concluding the proof.
\end{proof}

The final step for the proof of Theorem \ref{thm:coincidence of extensions} is to truncate $Q^{[a,b],\tilde{\ell}}$ into a certain triangle product. To do so, we need to define another decomposition of $Q^{[a,b],\tilde{\ell}}$ similar to the one we used for the proof of Proposition \ref{prop:Gamma and its extended version agree}, but we will swap ``left'' and ``right''.

As before, assume $\ell \geq 1$ and fix $\tilde{\ell} \geq \ell$. We define $Q_{\mathrm{t-reg}}^{[a,b],\tilde{\ell}}$ to be the full subquiver of $Q^{[a,b],\tilde{\ell}}$ which contains the $\tilde{\ell}$ \emph{rightmost} vertices of each row. We call it the \emph{trans-regular part} of $Q^{[a,b],\tilde{\ell}}$, and its vertices are called \emph{trans-regular}. We let $F^{[a,b],\tilde{\ell}}_{\mathrm{t-reg}} \subseteq Q^{[a,b],\tilde{\ell}}_{\mathrm{t-reg}}$ be the full subquiver consisting of the leftmost vertex of each row, and consider the potential $W^{[a,b],\tilde{\ell}}_{\mathrm{t-reg}}$ consisting of the sum of all $3$-cycles. This defines a relative Ginzburg dg algebra $\bG^{[a,b],\tilde{\ell}}_{\mathrm{t-reg}} = (Q^{[a,b],\tilde{\ell}}_{\mathrm{t-reg}},F^{[a,b],\tilde{\ell}}_{\mathrm{t-reg}},W^{[a,b],\tilde{\ell}}_{\mathrm{t-reg}})$. Our next goal will be to prove the following result.

\begin{proposition}\label{prop:truncating the extended Ginzburg algebra}
    Assume $\ell = \ell(Q^{[a,b]}) \geq 1$ and take $\tilde{\ell} \geq \ell$. Then $\bG^{[a,b],\tilde{\ell}}$ and $\bG^{[a,b],\tilde{\ell}}_{\mathrm{t-reg}}$ agree on extensions.
\end{proposition}

Similar to the previous case, $Q_{\mathrm{t-reg}}^{[a,b],\tilde{\ell}}$ is canonically isomorphic to the triangle product $F^{[a,b],\tilde{\ell}}_{\mathrm{t-reg}} \boxtimes Q^{\circ}_{\tilde{\ell}}$ in such a way that the row associated with a vertex $j \in \Delta_0$ corresponds to the row in the triangle product associated to $j$. We denote by $Q_{\mathrm{t-tens}}^{[a,b],\tilde{\ell}}$ the subquiver of $Q_{\mathrm{t-reg}}^{[a,b],\tilde{\ell}}$ that corresponds to $F^{[a,b],\tilde{\ell}}_{\mathrm{t-reg}} \ten Q^{\circ}_{\tilde{\ell}}$ under this isomorphism.

We define the \emph{trans-Auslander part} $Q_{\mathrm{t-Aus}}^{[a,b],\tilde{\ell}}$ of $Q^{[a,b],\tilde{\ell}}$ to be its full subquiver consisting of the vertices that are not trans-regular and the leftmost trans-regular vertex of each row. We call these vertices the \emph{trans-Auslander vertices}. We remark that $Q_{\mathrm{t-Aus}}^{[a,b],\tilde{\ell}}$ is isomorphic to the Auslander part $Q^{[a,b],\tilde{\ell}}$. In particular, if we define $\overline{Q}_{\mathrm{t-Aus}}^{[a,b],\tilde{\ell}}$ as the subquiver of the trans-Auslander part obtained by removing the horizontal arrows, then $\overline{Q}_{\mathrm{t-Aus}}^{[a,b],\tilde{\ell}}$ is isomorphic to the quiver $\overline{Q}_{\mathrm{Aus}}^{[a,b]}$ in Lemma \ref{lemma:Auslander part is well defined}.

Let $Q_{\mc{A^{\rm{t}}}}$ be the subquiver of $Q^{[a,b],\tilde{\ell}}$ that contains all vertices and whose arrows are precisely those in $\overline{Q}^{[a,b],\tilde{\ell}}_{\mathrm{t-Aus}}$ and $Q^{[a,b],\tilde{\ell}}_{\mathrm{t-tens}}$. We define $\mc{A}^{\rm{t}}$ as the quotient of the path category of $Q_{\mc{A^{\rm{t}}}}$ by the mesh relations on the subquiver $\overline{Q}^{[a,b],\tilde{\ell}}_{\mathrm{t-Aus}}$ (coming from Lemma \ref{lemma:Auslander part is well defined} and \eqref{eq:mesh relations}) and by the commutation relations on the subquiver $Q^{[a,b],\tilde{\ell}}_{\mathrm{t-tens}}$ (coming from (\ref{eq:commutation relations for tensor products})). Let $\mc{A}_1^{\rm{t}}$ and $\mc{A}_2^{\rm{t}}$ be the full subcategories on the trans-Auslander and trans-regular vertices, respectively, and denote by $\mc{C}^{\rm{t}}$ their intersection. Note that
\[
\Hom(\mc{A}_2^{\rm{t}}, \mc{A}_1^{\rm{t}}\setminus \mc{C}^{\rm{t}}) = \Hom(\mc{A}_1^{\rm{t}}, \mc{A}_2^{\rm{t}}\setminus \mc{C}^{\rm{t}}) = 0.
\]
because there are no arrows in $Q_{\mc{A}^{\rm{t}}}$ from a vertex in $Q_{\mathrm{t-tens}}^{[a,b],\tilde{\ell}} \cap \overline{Q}_{\mathrm{t-Aus}}^{[a,b],\tilde{\ell}}$ to a vertex in $Q_{\mc{A}^{\rm{t}}} \setminus Q_{\mathrm{t-tens}}^{[a,b],\tilde{\ell}} \cap \overline{Q}_{\mathrm{t-Aus}}^{[a,b],\tilde{\ell}}$. Similarly to the previous case (but with a different index), $\mc{A}_1^{\rm{t}}$ is equivalent to the category $\mc{A}^{[a,b]}_{\mathrm{Aus}}$ in Lemma \ref{lemma:Auslander part is well defined}. Finally, we define $\mc{B}^{\rm{t}} \subseteq \mc{A}_1^{\rm{t}}$ to be the full subcategory on the vertices of $F^{[a,b],\tilde{\ell}}$. We remark that the hypotheses of Proposition \ref{proposition:gluing of bimodules} are satisfied, but this time \eqref{eq:sink condition for gluing theorem} is the condition that holds and not \eqref{eq:source condition for gluing theorem}.

We have an analogue of Lemma \ref{lemma:construction of Gamma for adapted word}, which can be proved similarly.

\begin{lemma}\label{lemma:alternative construction of Gamma for adapted word}
    The total dg algebras of $\bP_3(\mc{A}^{\rm{t}},\mc{B}^{\rm{t}})$ and $\bP_3(\mc{A}^{\rm{t}}_2,\mc{C}^{\rm{t}})$ are quasi-isomorphic to $\bG^{[a,b],\tilde{\ell}}$ and $\bG^{[a,b],\tilde{\ell}}_{\rm{t-reg}}$, respectively.
\end{lemma}

We can now prove Proposition \ref{prop:truncating the extended Ginzburg algebra}.

\begin{proof}[Proof of Proposition \ref{prop:truncating the extended Ginzburg algebra}]
    Following the constructions above, we denote $\bG = \bG^{[a,b],\tilde{\ell}}$ and $\bP = \bP_3(\mc{A}^{\rm{t}},\mc{B}^{\rm{t}})$. By Lemma \ref{lemma:alternative construction of Gamma for adapted word} and the construction of $\bP$ as a tensor dg category over $\mc{A}^{\rm{t}}$, we have
    \[
    \RHom_{\bG}(e_s\bG,e_t\bG) \cong \RHom_\bP(s^{\wedge},t^{\wedge}) \cong \bigoplus_{r \geq 0} \RHom_{\mc{A}^{\rm{t}}}(s^{\wedge}, F^r(t^{\wedge}))
    \]
    for any two vertices $s,t \in Q^{[a,b],\tilde{\ell}}$, where $F$ denotes the functor $-\lten_{\mc{A}^{\rm{t}}}\Sigma^2\Omega_{\mc{A}^{\rm{t}},\mc{B}^{\rm{t}}}$ and we write $x^{\wedge} = \mc{A}^{\rm{t}}(-,x)$ for a vertex $x$. By Serre duality on $\mc{D}^b(\mc{A}^{\rm{t}})$, we have
    \[
    \RHom_{\mc{A}^{\rm{t}}}(s^{\wedge}, F^r(t^{\wedge})) \cong D\RHom_{\mc{A}^{\rm{t}}}(F^r(t^{\wedge}),s^{\vee}) \cong D\RHom_{\mc{A}^{\rm{t}}}(t^{\wedge},G^r(s^{\vee})),
    \]
    where $D$ denotes duality over $k$, $s^{\vee}$ is the corepresentable module $D\mc{A}^{\rm{t}}(s,-)$, and $G = \RHom_{\mc{A}^{\rm{t}}}(\Sigma^2\Omega_{\mc{A}^{\rm{t}},\mc{B}^{\rm{t}}},-)$ is the right adjoint of $F$. When $t$ is trans-regular, observe that the complex above only depends on the restriction of $G^r(s^{\vee})$ to $\mc{A}_2^{\rm{t}}$. Similarly, if $\bG' = \bG^{[a,b],\tilde{\ell}}_{\rm{t-reg}}$, then we obtain that
    \[
    \RHom_{\bG'}(e_s\bG',e_t\bG') \cong \bigoplus_{r \geq 0} D\RHom_{\mc{A}_2^{\rm{t}}}(\restr{t^{\wedge}}{\mc{A}_2^{\rm{t}}},G_2^r(\restr{s^{\vee}}{\mc{A}_2^{\rm{t}}}))
    \]
    for any $s,t \in Q^{[a,b],\tilde{\ell}}_{\rm{t-reg}}$, where $G_2 = \RHom_{\mc{A}_2^{\rm{t}}}(\Sigma^2\Omega_{\mc{A}_2^{\rm{t}},\mc{C}^{\rm{t}}},-)$. Therefore, to prove that $\bG$ and $\bG'$ agree on extensions, it suffices to show that
    \[
    \restr{G^r(s^{\vee})}{\mc{A}_2^{\rm{t}}} \cong G_2^r(\restr{s^{\vee}}{\mc{A}_2^{\rm{t}}})
    \]
    for any trans-regular vertex $s$ and $r \geq 0$. In fact, we will now prove a stronger result:
    \[
    \restr{G(M)}{\mc{A}_2^{\rm{t}}} \cong G_2(\restr{M}{\mc{A}_2^{\rm{t}}}).
    \]
    for any $M \in \mc{D}^b(\mc{A}^{\rm{t}})$.

    Applying $\RHom_{\mc{A}^{\rm{t}}}(-,M)$ to the distinguished triangle in Proposition \ref{proposition:gluing of bimodules}, we get the distinguished triangle
    \[
    \mu_1^{\bf{R}}(G_1(\restr{M}{\mc{A}_1^{\rm{t}}}))
     \longrightarrow  G(M) \longrightarrow \mu_2^{\bf{R}}(G_2(\restr{M}{\mc{A}_2^{\rm{t}}})) \longrightarrow \Sigma\mu_1^{\bf{R}}(G_1(\restr{M}{\mc{A}_1^{\rm{t}}}))
    \]
    in $\mc{D}^b(\mc{A}^{\rm{t}})$, where $G_1 = \RHom_{\mc{A}_1^{\rm{t}}}(\Sigma^2\Omega_{\mc{A}_1^{\rm{t}},\mc{B}^{\rm{t}}},-)$ and $\mu_i^{\bf{R}} = \RHom_{\mc{A}_i^{\rm{t}}}(\mc{A}^{\rm{t}},-)$ for $i=1,2$. Let us show that
    \begin{equation}\label{eq:restriction that vanishes}
    \restr{\mu_1^{\bf{R}}(G_1(\restr{M}{\mc{A}_1^{\rm{t}}}))}{\mc{A}_2^{\rm{t}}} = 0,
    \end{equation}
    which will imply our claim. For a simple $\mc{A}_1^{\rm{t}}$-module $S$, Lemma \ref{lemma:adjoint action of Omega for Auslander algebra} implies that $G_1(S)$ is still concentrated in degree zero and $\restr{G_1(S)}{\mc{C}^{\rm{t}}} = 0$. By Lemma \ref{lemma:(co)induction with gluing conditions}, $\mu_1^{\bf{R}}(G_1(S)) = \mu_1(G_1(S))$ where $\mu_1: \Modcat\mc{A}_1^{\rm{t}} \to \Modcat\mc{A}^{\rm{t}}$ is the functor defined in Section \ref{section:gluing bimodules}. But $\mu_1(G_1(S))$ vanishes on $\mc{A}_2^{\rm{t}}$ since $G_1(S)$ vanishes on $\mc{C}^{\rm{t}}$. Since $\restr{M}{\mc{A}_1^{\rm{t}}}$ belongs to $\mc{D}^b(\mc{A}_1^{\rm{t}})$, which is generated by the simple $\mc{A}_1^{\rm{t}}$-modules, we obtain \eqref{eq:restriction that vanishes}, as desired. This concludes the proof.
\end{proof}

We can finally give a proof for Theorem \ref{thm:coincidence of extensions}.

\begin{proof}[Proof of Theorem \ref{thm:coincidence of extensions}]
    Let $a_1,a_2,b \in \Z$ with $a_1,a_2 \leq b$. To prove that $\bG^{[a_1,b]}$ and $\bG^{[a_2,b]}$ agree on extensions, we may assume that $a_1,a_2 \leq b - l(w_0) + 1$ by Lemma \ref{lemma:short word agree on extensions with word for w0}. In particular, the regular lengths $\ell_1 = \ell(Q^{[a_1,b]})$ and $\ell_2 = \ell(Q^{[a_2,b]})$ satisfy $\ell_1,\ell_2 \geq 1$. Take $\tilde{\ell} \geq \ell_1,\ell_2$ such that $Q^{[a_i,b],\tilde{\ell}}_{\rm{t-reg}}$ contains $Q^{[a_i,b]}$ for $i=1,2$. By Propositions \ref{prop:Gamma and its extended version agree} and \ref{prop:truncating the extended Ginzburg algebra}, $\bG^{[a_i,b]}$ and $\bG^{[a_i,b],\tilde{\ell}}_{\rm{t-reg}}$ agree on extensions. But $Q^{[a_1,b],\tilde{\ell}}_{\rm{t-reg}} = Q^{[a_2,b],\tilde{\ell}}_{\rm{t-reg}}$, so $\bG^{[a_1,b],\tilde{\ell}}_{\rm{t-reg}} = \bG^{[a_2,b],\tilde{\ell}}_{\rm{t-reg}}$, concluding the proof.
\end{proof}

\subsection{Commutation and braid moves} In this section, we describe how the ice quiver with potential $(Q^{[a,b]}(\underline{w}_0),F^{[a,b]}(\underline{w}_0),W^{[a,b]}(\underline{w}_0))$ changes when we apply a commutation move or a braid move to the reduced expression $\underline{w}_0$. We will follow the arguments in \cite[Section 6.1]{ContuMonoidalAdditive} closely.

Let $\mf{i} = (i_s)_{s\in\Z}$ be an infinite sequence of indices $i_s \in \Delta_0$ satisfying the following:
\begin{equation}\label{eq:condition for infinite sequences}
    \textrm{for any } s \in \Z, \textrm{ there are } s' < s < s'' \textrm{ such that } i_{s'} = i_s = i_{s''}.
\end{equation}
Note that the sequence $\widehat{\underline{w}}_0$ satisfies this condition. It allows us to define $s^-$ and $s^+$ for any $s \in \Z$ as in Section \ref{section:initial monoidal seed}. Therefore, for a finite interval $[a,b]$, we can define an ice quiver with potential $(Q^{[a,b]}(\mf{i}),F^{[a,b]}(\mf{i}),W^{[a,b]}(\mf{i}))$ as in Section \ref{section:initial monoidal seed} by replacing $\widehat{\underline{w}}_0$ by $\mf{i}$.

Let $\mf{i} = (i_s)_{s\in\Z}$ and $\mf{i}' = (i_s')_{s\in\Z}$ be two sequences of indices in $\Delta_0$ satisfying (\ref{eq:condition for infinite sequences}). We say that $\mf{i}'$ is obtained from $\mf{i}$ by a
\begin{itemize}
    \item \emph{commutation move} at position $s \in \Z$ if $i_s \not\sim i_{s+1}$ and we have $i_s' = i_{s+1}$, $i_{s+1}' = i_s$ and $i_t' = i_t$ for $t \neq s,s+1$.
    \item \emph{braid move} at position $s \in \Z$ if $i_s \sim i_{s+1}$ and we have $i'_{s-1} = i_{s+1}' = i_s$, $i_s' = i_{s-1} = i_{s+1}$ and $i_t' = i_t$ for $t \neq s-1,s,s+1$.
\end{itemize}
These moves have the following effect on the associated ice quivers with potential.
 
\begin{lemma}\label{lemma:effect of commutation move}
    Let $\mf{i} = (i_s)_{s\in\Z}$ and $\mf{i}' = (i_s')_{s\in\Z}$ be two sequences of indices in $\Delta_0$ satisfying (\ref{eq:condition for infinite sequences}). Let $[a,b]$ be a finite interval. If $\mf{i}'$ is obtained from $\mf{i}$ by a commutation move at position $a \leq s < b$, then there exists an isomorphism of ice quivers with potential
    \[
    \varphi: (Q^{[a,b]}(\mf{i}),F^{[a,b]}(\mf{i}),W^{[a,b]}(\mf{i})) \longrightarrow (Q^{[a,b]}(\mf{i}'),F^{[a,b]}(\mf{i}'),W^{[a,b]}(\mf{i}')).
    \]
\end{lemma}

\begin{proof}
    We define $\varphi$ on vertices by
    \[
    \varphi(t) = \begin{cases}
        t, &\textrm{if } t \neq s,s+1,\\
        s+1 &\textrm{if } t = s,\\
        s &\textrm{if } t = s+1,
    \end{cases}
    \]
    for $t \in [a,b]$. Since $i_s \not\sim i_{s+1}$, it is immediate from the definition of these ice quivers with potential that $\varphi$ yields the desired isomorphism.
\end{proof}

\begin{lemma}\label{lemma:effect of braid move}
    Let $\mf{i} = (i_s)_{s\in\Z}$ and $\mf{i}' = (i_s')_{s\in\Z}$ be two sequences of indices in $\Delta_0$ satisfying (\ref{eq:condition for infinite sequences}). Let $[a,b]$ be a finite interval. If $\mf{i}'$ is obtained from $\mf{i}$ by a braid move at position $a < s < b$, then there exists an isomorphism of ice quivers with potential
    \[
    \varphi: \mu_{s+1}(Q^{[a,b]}(\mf{i}),F^{[a,b]}(\mf{i}),W^{[a,b]}(\mf{i})) \longrightarrow (Q^{[a,b]}(\mf{i}'),F^{[a,b]}(\mf{i}'),W^{[a,b]}(\mf{i}')).
    \]
\end{lemma}

\begin{proof}
    To ease the reading, we will omit the index $[a,b]$ in the notation of the ice quivers and the potentials. We compute the mutation above following the proof of \cite[Lemma 6.2]{ContuMonoidalAdditive}.

    By the definition of $Q(\mf{i})$, its full subquiver containing all vertices connected to $s+1$ is of the following form:
    
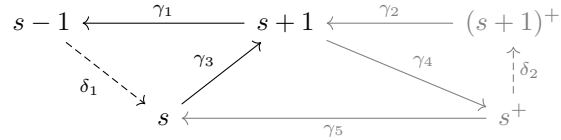
\[\begin{tikzcd}
	{s-1} && {s+1} && \textcolor{rgb,255:red,128;green,128;blue,128}{{(s+1)^+}} \\
	& s &&& \textcolor{rgb,255:red,128;green,128;blue,128}{{s^+}}
	\arrow["{\delta_1}"', dashed, from=1-1, to=2-2]
	\arrow["{\gamma_1}"', from=1-3, to=1-1]
	\arrow["{\gamma_4}", color={rgb,255:red,128;green,128;blue,128}, from=1-3, to=2-5]
	\arrow["{\gamma_2}"', color={rgb,255:red,128;green,128;blue,128}, from=1-5, to=1-3]
	\arrow["{\gamma_3}", from=2-2, to=1-3]
	\arrow["{\delta_2}"', color={rgb,255:red,128;green,128;blue,128}, dashed, from=2-5, to=1-5]
	\arrow["{\gamma_5}", color={rgb,255:red,128;green,128;blue,128}, from=2-5, to=2-2]
    \end{tikzcd}\]
    Here, the arrow $\delta_1$ appears if $(s-1)^- < s^-$ or if both $s-1$ and $s$ are frozen vertices. The vertex $s^+$ appears if $ s^+ \leq b$, and similarly for $(s+1)^+$. If both appear, the arrow $\delta_2$ is present precisely when $s^+ < (s+1)^+$. If $\delta_1$ does not appear or is not frozen, the proof reduces to that in \cite{ContuMonoidalAdditive}, so we will assume from now on that $s-1$ and $s$ are frozen. We will also suppose that $s^+$ and $(s+1)^+$ appear but not $\delta_2$. The other cases can be treated similarly. With these assumptions, we have the following picture:
    
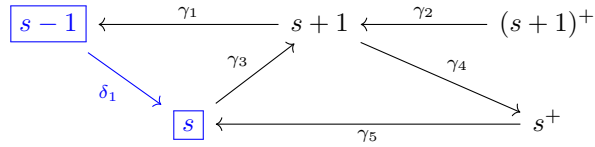
\[\begin{tikzcd}
	\color{blue}\boxed{s-1} && {s+1} && {(s+1)^+} \\
	& \color{blue}\boxed{s} &&& {s^+}
	\arrow[color=blue, "{\delta_1}"', from=1-1, to=2-2]
	\arrow["{\gamma_1}"', from=1-3, to=1-1]
	\arrow["{\gamma_4}", from=1-3, to=2-5]
	\arrow["{\gamma_2}"', from=1-5, to=1-3]
	\arrow["{\gamma_3}", from=2-2, to=1-3]
	\arrow["{\gamma_5}", from=2-5, to=2-2]
    \end{tikzcd}\]
    We can write
    \[
    W(\mf{i}) = \gamma_5\gamma_4\gamma_3 + \gamma_1\gamma_3\delta_1 + c_1 + c_2 + c_3 + W_0
    \]
    where:
    \begin{enumerate}[(1)]
        \item $c_1$ is the sum of all simple cordless cycles containing the path $\gamma_1\gamma_2$;
        \item $c_2$ is zero or the unique simple cordless cycle containing the path $\gamma_4\gamma_2$ if it exists;
        \item $c_3$ is the sum of all simple cordless cycles containing the arrow $\gamma_5$ except for $\gamma_5\gamma_4\gamma_3$;
        \item $W_0$ is the sum of all simple cordless cycles that do not contain any arrow in the picture above.
    \end{enumerate}
    Let us compute $\widetilde{\mu}_{s+1}(Q(\mf{i}),F(\mf{i}),W(\mf{i})) = (\widetilde{\mu}_{s+1}Q(\mf{i}),F(\mf{i}),\widetilde{\mu}_{s+1}W(\mf{i}))$. We see that $Q'= \widetilde{\mu}_{s+1}Q(\mf{i})$ differs from $Q(\mf{i})$ only around $s+1$, where it is given by
    
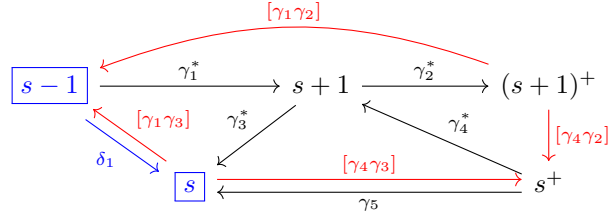
\[\begin{tikzcd}
	\color{blue}\boxed{s-1} && {s+1} && {(s+1)^+} \\
	& \color{blue}\boxed{s} &&& {s^+}
	\arrow["{\gamma_1^*}", from=1-1, to=1-3]
	\arrow[color=blue, "{\delta_1}"', shift right, from=1-1, to=2-2]
	\arrow["{\gamma_2^*}", from=1-3, to=1-5]
	\arrow["{\gamma_3^*}"', from=1-3, to=2-2]
	\arrow[color=red, "{[\gamma_1\gamma_2]}"', bend right=20, from=1-5, to=1-1]
	\arrow[color=red, "{[\gamma_4\gamma_2]}", from=1-5, to=2-5]
	\arrow[color=red, "{[\gamma_1\gamma_3]}"', shift right, from=2-2, to=1-1]
	\arrow[color=red, "{[\gamma_4\gamma_3]}", shift left, from=2-2, to=2-5]
	\arrow["{\gamma_4^*}"', from=2-5, to=1-3]
	\arrow["{\gamma_5}", shift left, from=2-5, to=2-2]
    \end{tikzcd}\]
    where the new red arrows are not frozen. Moreover, we have
    \[
    W' = \widetilde{\mu}_{s+1}W(\mf{i}) = \gamma_5[\gamma_4\gamma_3] + [\gamma_1\gamma_3]\delta_1 + \widetilde{c}_1 + \widetilde{c}_2 + c_3 + W^* + W_0
    \]
    where $\widetilde{c}_1$ and $\widetilde{c}_2$ are obtained from $c_1$ and $c_2$ by replacing all instances of the paths $\gamma_1\gamma_2$ and $\gamma_4\gamma_2$ by $[\gamma_1\gamma_2]$ and $[\gamma_4\gamma_2]$, and
    \[
    W^* = [\gamma_1\gamma_2]\gamma_2^*\gamma_1^* + [\gamma_1\gamma_3]\gamma_3^*\gamma_1^* + [\gamma_4\gamma_2]\gamma_2^*\gamma_4^* + [\gamma_4\gamma_3]\gamma_3^*\gamma_4^*.
    \]
    Observe that $(Q',F(\mf{i}),W')$ is not reduced, and the $2$-cycles contained in the new potential are precisely $\gamma_5[\gamma_4\gamma_3]$ and $[\gamma_1\gamma_3]\delta_1$. We now follow the algorithm in the proof of \cite[Lemma 3.14]{PresslandMutation} (which is based on \cite[Lemma 4.7]{DerksenWeymanZelevinskyI}). Consider the algebra endomorphism $f$ of the completed path algebra of $Q'$ that fixes the vertices and is defined on arrows as follows:
    \begin{enumerate}[(1)]
        \item $f(\gamma_5) = \gamma_5 + \gamma_3^*\gamma_4^*$;
        \item $f([\gamma_4\gamma_3]) = [\gamma_4\gamma_3] - \widetilde{c}_3$;
        \item $f(\alpha) = -\alpha$ for $\alpha = \delta_1, \gamma_1^*,\gamma_2^*,\gamma_4^*,[\gamma_1\gamma_3]$;
        \item $f(\beta) = \beta$ for any arrow $\beta \neq \gamma_5,[\gamma_4\gamma_3],\delta_1, \gamma_1^*,\gamma_2^*,\gamma_4^*,[\gamma_1\gamma_3]$ in $Q'$.
    \end{enumerate}
    Here, $\widetilde{c}_3$ is the sum of paths such that $c_3$ is cyclically equivalent to $\gamma_5\widetilde{c}_3$. Note that no path in $\widetilde{c}_3$ contains the arrow $\gamma_5$ by the definition of $W(\mf{i})$. Clearly, $f$ is an automorphism by \cite[Proposition 2.4]{DerksenWeymanZelevinskyI}. Moreover, it preserves the completed path algebra of the frozen subquiver, so it induces a right equivalence of $(Q',F(\mf{i}),W')$ and $(Q',F(\mf{i}),f(W'))$. One computes that, up to cyclic equivalence, we have
    \[
    f(W') = \gamma_5[\gamma_4\gamma_3] + [\gamma_1\gamma_3]\delta_1 + \widetilde{c}_1 + \widetilde{c}_2 + \gamma_3^*\gamma_4^*\widetilde{c}_3 + (W^* - [\gamma_4\gamma_3]\gamma_3^*\gamma_4^*) + W_0.
    \]
    The potential above has the form described in \cite[Lemma 3.14]{PresslandMutation}. Thus, by the proof of Theorem 3.6 in loc.\ cit., the ice quiver of the reduction of $(Q', F(\mf{i}), W')$ is obtained by deleting the arrows $\gamma_5$, $[\gamma_4\gamma_3]$ and $\delta_1$, and freezing the arrow $[\gamma_1\gamma_3]$. Graphically, after adjusting the positions of the vertices $s-1$ and $s$, it is given locally at $s+1$ by
    
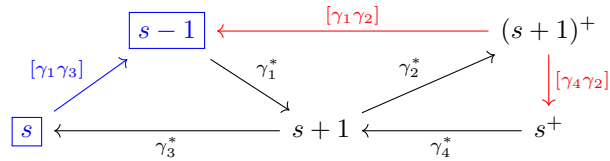
\[\begin{tikzcd}
	& \color{blue}\boxed{s-1} &&& {(s+1)^+} \\
	\color{blue}\boxed{s} && {s+1} && {s^+}
	\arrow["{\gamma_1^*}", from=1-2, to=2-3]
	\arrow[color=red, "{[\gamma_1\gamma_2]}"', from=1-5, to=1-2]
	\arrow[color=red, "{[\gamma_4\gamma_2]}", from=1-5, to=2-5]
	\arrow[color=blue, "{[\gamma_1\gamma_3]}", from=2-1, to=1-2]
	\arrow["{\gamma_2^*}", from=2-3, to=1-5]
	\arrow["{\gamma_3^*}", from=2-3, to=2-1]
	\arrow["{\gamma_4^*}", from=2-5, to=2-3]
    \end{tikzcd}\]
    We can take the potential of the reduction to be
    \begin{align*}
    W_{\mathrm{red}} &= \widetilde{c}_1 + \widetilde{c}_2 + \gamma_3^*\gamma_4^*\widetilde{c}_3 + (W^* - [\gamma_4\gamma_3]\gamma_3^*\gamma_4^*) + W_0\\
    &= [\gamma_1\gamma_2]\gamma_2^*\gamma_1^* + [\gamma_1\gamma_3]\gamma_3^*\gamma_1^* + [\gamma_4\gamma_2]\gamma_2^*\gamma_4^* + \widetilde{c}_1 + \widetilde{c}_2 + \gamma_3^*\gamma_4^*\widetilde{c}_3 + W_0.
    \end{align*}
    This reduced ice quiver with potential is $\mu_{s+1}(Q(\mf{i}),F(\mf{i}),W(\mf{i}))$ by definition. Comparing it with $(Q(\mf{i}'),F(\mf{i}'),W(\mf{i}'))$, it is clear that we have an isomorphism $\varphi$ as in the statement which is given on vertices by
    \[
    \varphi(t) = \begin{cases}
        t, &\textrm{if } t \neq s-1,s,\\
        s &\textrm{if } t = s-1,\\
        s-1 &\textrm{if } t = s,
    \end{cases}
    \]
    for $t \in [a,b]$.
\end{proof}

For the next result, we say that a sequence of commutation and braid moves is \emph{compatible} with an interval $[a,b]$ if the position $s$ of each of its moves satisfies $a \leq s < b$ (if it is a commutation move) or $a < s < b$ (if it is a braid move).

\begin{proposition}\label{prop:words related by braid/commutation moves agree on properness}
Let $\mf{i} = (i_s)_{s\in\Z}$ and $\mf{i}' = (i_s')_{s\in\Z}$ be two sequences of indices in $\Delta_0$ satisfying (\ref{eq:condition for infinite sequences}). Let $[a,b]$ be a finite interval. Suppose $\mf{i}'$ is obtained from $\mf{i}$ by a sequence of commutation and braid moves compatible with $[a,b]$. Then one of 
\[
\bGh(Q^{[a,b]}(\mf{i}),F^{[a,b]}(\mf{i}),W^{[a,b]}(\mf{i})) \quad \textrm{and} \quad \bGh(Q^{[a,b]}(\mf{i}'),F^{[a,b]}(\mf{i}'),W^{[a,b]}(\mf{i}'))
\]
is proper if and only if the other is proper.
\end{proposition}

\begin{proof}
    This is a combination of Lemmas \ref{lemma:mutation preserves properness}, \ref{lemma:effect of commutation move}, and \ref{lemma:effect of braid move}.
\end{proof}

We can now return to the case when $\mf{i} = \widehat{\un{w}}_0$. Let $[a,b]$ be a finite integer interval. Let $r \geq 0$ be the largest integer such that $a' = a + r \cdot l(w_0)$ satisfies $a'\leq b$. We define the \emph{$[a,b]$-residue} of $\widehat{\un{w}}_0$ to be the element
\[
w = s_{i_{a'}}s_{i_{a'+1}}\dotsb s_{i_b}
\]
of the Weyl group of $\Delta$, where $s_j$ denotes the simple reflection associated with $j \in \Delta_0$.

\begin{proposition}\label{prop:words that we can prove give a proper Ginzburg}
    Let $\un{w}_0$ be a reduced expression of $w_0$ and take a finite integer interval $[a,b]$. If the $[a,b]$-residue of $\widehat{\un{w}}_0$ has a reduced expression adapted to an orientation of $\Delta$, then $\bGh^{[a,b]}(\un{w}_0)$ is proper.
\end{proposition}

\begin{proof}
    Assume without loss of generality that $a = 1$. Let $w = s_{i_{a'}}s_{i_{a'+1}}\dotsb s_{i_b}$ denote the $[1,b]$-residue of $\widehat{\un{w}}_0$, where $a'= 1 + r \cdot l(w_0)$ as above. Let $(j_1,\dots,j_t)$ be a reduced expression of $w$ adapted to an orientation $Q$ of $\Delta$. By \cite{Bedard} (see also \cite{OhSuh19a}), there is a compatible reading of the Auslander--Reiten quiver of $Q$ starting with $(j_1,\dots,j_t)$. Consequently, we can extend this sequence of indices to a reduced word $(j_1,\dots,j_{l(w_0)})$ for $w_0$ that is adapted to $Q$. By performing a sequence $\mf{s}$ of commutation and braid moves, we can transform $(i_{a'},i_{a'+1},\dots,i_b)$ into $(j_1,\dots,j_t)$, since both represent the same element $w$. Similarly, we can transform $(i_{b+1},i_{b+2},\dots,i_{a'+l(w_0) - 1})$ into $(j_{t+1},j_{t+2},\dots,j_{l(w_0)})$ by a sequence $\mf{s}'$ of moves, since both represent the element $w^{-1}w_0$. Now we define a reduced word $\un{w}_0'$ for $w_0$ as
    \[
    \un{w}_0' = (j_1,\dots,j_{l(w_0)}) \quad \textrm{or} \quad \un{w}_0' = (j_1^*,\dots,j_{l(w_0)}^*),
    \]
    depending on whether $r$ is even or odd, respectively. This ensures that the subsequence of $\widehat{\un{w}}_0'$ from entries $a'$ to $b$ is precisely $(j_1,\dots,j_t)$. Note that $\un{w}_0'$ is adapted to an orientation of $\Delta$. The next step is to extend $\mf{s}$ and $\mf{s}'$ to a sequence of braid and commutation moves on $\widehat{\un{w}}_0$ as follows. We can divide the integers from $1$ to $b$ into blocks of consecutive numbers whose sizes alternate between $t$ and $l(w_0) - t$. On each block of size $t$, we apply the sequence $\mf{s}$, and on each block of size $l(w_0)-t$, we apply the sequence $\mf{s}'$ (with the positions of the moves appropriately shifted). It is not hard to see that this indeed defines a sequence of commutation and braid moves starting from $\widehat{\un{w}}_0$ and that the entries between positions $1$ and $b$ of the resulting infinite sequence match the corresponding entries in $\widehat{\un{w}}_0'$. But $\bGh^{[1,b]}(\un{w}_0')$ is proper by Corollary \ref{cor:adapted words give proper Ginzburgs}, so the same holds for $\bGh^{[1,b]}(\un{w}_0)$ by Proposition \ref{prop:words related by braid/commutation moves agree on properness}.
\end{proof}

\begin{corollary}\label{cor:properness if length of w0 divides length of [a,b]}
Let $\un{w}_0$ be a reduced expression of $w_0$ and take a finite integer interval $[a,b]$. If $l(w_0)$ divides $b-a+1$, then $\bGh^{[a,b]}(\un{w}_0)$ is proper.
\end{corollary}

\begin{proof}
    The hypothesis forces the $[a,b]$-residue of $\widehat{\un{w}}_0$ to be the longest element, which has a reduced expression adapted to an orientation of $\Delta$. The result then follows from Proposition \ref{prop:words that we can prove give a proper Ginzburg}.
\end{proof}

\section{Additive categorification of the monoidal \texorpdfstring{$\Lambda$}{}-matrix}\label{section:additive categorification of the Lambda-matrix}

In this section, we suppose that $\mf{g}$ is untwisted of simply-laced type and that the reduced word $\un{w}_0$ is adapted to an orientation of $\Delta$. We fix $a,b \in \Z$ with $a \leq b$.

\begin{lemma}
    The ice quiver with potential $(Q^{[a,b]}(\un{w}_0),F^{[a,b]}(\un{w}_0),W^{[a,b]}(\un{w}_0))$ is rigid.
\end{lemma}

\begin{proof}
    Since $\un{w}_0$ is adapted, we can view $Q^{[a,b]}(\un{w}_0)$ as a subquiver of $Q_{\rm{HL}}$ by Lemma \ref{lemma:identification of HL with GLS quivers}. Thus, the argument in \cite[Proposition 4.17]{HLcluster} shows that the quiver with potential $(Q^{[a,b]}(\un{w}_0),W^{[a,b]}(\un{w}_0))$ is rigid. They do not use the relation $\partial_{\alpha}(W^{[a,b]}(\un{w}_0)) = 0$ for $\alpha \in F_1^{[a,b]}(\un{w_0})$, which is not present in the relative Jacobian algebra, so their proof also applies to the ice quiver with potential above.
\end{proof}

In particular, the ice quiver with potential above is non-degenerate. It is also Jacobi-finite since the completed relative Ginzburg dg algebra $\bGh^{[a,b]}(\un{w}_0)$ is proper by Corollary \ref{cor:adapted words give proper Ginzburgs}. Therefore, by Section \ref{section:categorification after Yilin}, we can use the associated Higgs category $\mc{H}^{[a,b]}(\un{w}_0)$ to study its cluster algebra $\mc{A}^{[a,b]}(\un{w}_0)$. Denote by $\varphi: K_0(\mathscr{C}_{\mf{g}}^{[a,b],\mc{D},\un{w}_0}) \to \mc{A}^{[a,b]}(\un{w}_0)$ the isomorphism arising from Theorem \ref{thm:KKOP monoidal categorification} and Proposition \ref{prop:quiver of initial monoidal seed of KKOP}. The main result of this section shows how to recover the $\Lambda$-invariant in this $\Uplambda$-monoidal categorification using $\mc{H}^{[a,b]}(\un{w}_0)$.

\begin{theorem}\label{thm:coincidence of Lambda matrices}
    Suppose $\mf{g}$ is untwisted of simply-laced type. Let $(\mc{D},\un{w}_0)$ be a complete PBW-pair adapted to an orientation of $\Delta$, and let $a,b \in \Z$ with $a \leq b$. For reachable simple objects $V,W \in \mathscr{C}_{\mf{g}}^{[a,b],\mc{D},\un{w}_0}$, we have
    \[
    \Lambda(V,W) = \dim\Ext^1_{\mc{H}}(M,N) + [M,N]_{\mc{H}},
    \]
    where $\mc{H} = \mc{H}^{[a,b]}(\un{w}_0)$ and $M,N \in \mc{H}$ are the corresponding reachable rigid objects such that $\varphi(V) = CC(M)$ and $\varphi(W) = CC(N)$ in $\mc{A}^{[a,b]}(\un{w}_0)$.
\end{theorem}

\begin{proof}
Theorems \ref{thm:canonical quantum structure} and \ref{thm:KKOP monoidal categorification} endow $\mc{A}^{[a,b]}(\un{w}_0)$ with two $\Uplambda$-cluster algebra structures. If we show that they coincide, then the theorem follows by Proposition \ref{prop:additive interpretation of the tropical and F invariants} and Theorem \ref{thm:Lambda invariant as a tropical invariant}, since each side of the equality above computes the same tropical invariant in $\mc{A}^{[a,b]}(\un{w}_0)$. Therefore, it suffices to show that the Poisson coefficient matrices of some $\Uplambda$-seed are equal.

We assume without loss of generality that $b = 0$. If $\xi$ denotes the height function giving rise to $(\mc{D},\un{w}_0)$, let $\mathscr{S}$ be the (reachable) monoidal seed in $\mathscr{C}_{\mf{g}}^{[a,0],\mc{D},\un{w}_0}$ that is the restriction of the monoidal seed $\mathscr{S}^{<\xi}$ from Theorem \ref{thm:initial monoidal seed for HL category} to $Q^{[a,0]}(\un{w}_0)$, as explained in Remark \ref{rmk:initial monoidal seed of finite interval}. We assume $V,W$ belong to $\mathscr{S}$. Since these objects correspond to the same seed in $\mc{A}^{[a,0]}(\un{w}_0)$, we have $\dd(V,W) = 0$ and $\Ext^1_{\mc{H}}(M,N) = 0$. By Proposition \ref{prop: lambda as sum of delta}, we have
\[
\Lambda(V,W) = \sum_{n \geq 1} (-1)^{n-1} \left[\dd(V, \DD^{-n}(W))  - \dd(W, \DD^{-n}(V)) \right].
\]
On the other hand,
\[
[M,N]_{\mc{H}} = \sum_{n \geq 0}(-1)^n(\dim\Ext^{-n}_{\mc{H}}(M,N) - \dim\Ext^{-n}_{\mc{H}}(N,M))
\]
by definition. By Proposition \ref{prop:d-invariant with dual as a negative extension} below, the terms of these sums coincide, so we have $\Lambda(V,W) = [M,N]_{\mc{H}}$. This shows that the Poisson coefficient matrices of the initial $\Uplambda$-seeds coincide, concluding the proof.
\end{proof}

\begin{proposition}\label{prop:d-invariant with dual as a negative extension}
    Suppose $\mf{g}$ is untwisted of simply-laced type. Let $(\mc{D},\un{w}_0)$ be a complete PBW-pair adapted to an orientation of $\Delta$, and denote by $\xi: \Delta_0 \to \Z$ the corresponding height function. Take an integer $a \leq 0$ and consider the initial monoidal seed $\mathscr{S}$ of $\mathscr{C}_{\mf{g}}^{[a,0],\mc{D},\un{w}_0}$ obtained by restricting the monoidal seed $\mathscr{S}^{<\xi}$ from Theorem \ref{thm:initial monoidal seed for HL category} to $Q^{[a,0]}(\un{w}_0)$. For simple objects $V$ and $W$ in $\mathscr{S}$, we have
    \begin{equation}\label{eq:shifted d-invariant is a negative extension group}
    \dd(V,\DD^{-n}(W)) = \dim\Ext^{1-n}_{\mc{H}}(M,N)
    \end{equation}
    for any $n \geq 1$, where $\mc{H} = \mc{H}^{[a,0]}(\un{w}_0)$ and $M,N \in \mc{H}$ are the corresponding reachable rigid indecomposable objects such that $\varphi(V) = CC(M)$ and $\varphi(W) = CC(N)$ in $\mc{A}^{[a,0]}(\un{w}_0)$.
\end{proposition}

\begin{proof}
The following observation is essential: for fixed $V,W \in \mathscr{C}_{\mf{g}}$, the left-hand side of \eqref{eq:shifted d-invariant is a negative extension group} does not depend on $a$ by Remark \ref{rmk:initial monoidal seed of finite interval}, and neither does the right-hand side by Proposition \ref{proposition:negative extensions can be computed in per Gamma} and Theorem \ref{thm:coincidence of extensions} (see also Lemma \ref{lemma:complete/noncomplete Ginzburg are qiso when Jacobi-finite} and Remark \ref{rmk:defining an Adams grading}). By taking $a$ such that the length of the interval $[a,0]$ is a multiple of $2l(w_0)$, we can assume every row of $Q^{[a,0]}(\un{w}_0)$ has exactly $r+1$ vertices for some $r \geq 0$ (see \cite[Corollary 2.20]{Bedard}). By Proposition \ref{proposition:mutation sequence the gives left dual}, if we identify $Q^{[a,0]}(\un{w}_0)$ with a subquiver of $Q^{<\xi}_{\rm{HL}}$ via Lemma \ref{lemma:identification of HL with GLS quivers}, we can suppose $r$ is large enough so that $\DD^{-n}(W)$ is a reachable simple object of $\mathscr{C}_{\mf{g}}^{[a,0],\mc{D},\un{w}_0}$ that can be reached by applying the sequence of mutations ${\bf v}^{\leq r}_{\DD^{-n}}$ from Remark \ref{rmk:truncation of mutation sequence that gives dual} to $\mathscr{S}$. In particular, $\varphi(\DD^{-n}(W))$ is a cluster variable of $\mc{A}^{[a,0]}(\un{w}_0)$ and
\[
\dd(V,\DD^{-n}(W)) = \frac{1}{2}(\varphi(V)\mid\mid \varphi(\DD^{-n}(M)))_F
\]
by Theorem \ref{thm:Lambda invariant as a tropical invariant}. Let $\wt{N} \in \mc{H}$ be the reachable rigid indecomposable object such that $CC(\widetilde{N}) = \varphi(\DD^{-n}(W))$. We claim that $\wt{N} \cong \Sigma^{-n}N$ (up to increasing $r$), where $\Sigma$ is the suspension functor of the relative cluster category. This is enough to conclude the proof. Indeed, if this is the case, then the number above equals
\begin{align*}
    \frac{1}{2}(CC(M)\mid\mid CC(\wt{N})))_F &= \dim\Ext^1_{\mc{H}}(M,\wt{N})\\
    &= \dim\Ext^1_{\mc{H}}(M,\Sigma^{-n}N)\\
    &= \dim\Ext^{1-n}_{\mc{H}}(M,N),
\end{align*}
where the first equality follows from Proposition \ref{prop:additive interpretation of the tropical and F invariants}. We crucially use here the fact that the $F$-invariant does not depend on the choice of the Poisson coefficient matrices.

To prove our claim, let us first assume that $\Sigma^{-n}N \in \mc{H}$. In particular, observe that $N$ is not projective-injective in $\mc{H}$ by the definition of the Higgs category. Let $(\ov{Q},\ov{W})$ be the quiver with potential obtained from $(Q^{[a,0]}(\un{w}_0),F^{[a,0]}(\un{w}_0),W^{[a,0]}(\un{w}_0))$ by deleting the frozen vertices. As explained in Remark \ref{rmk:sequence that gives dual is indeed a maximal green sequence}, Theorem \ref{thm:maximal green sequence for triangle product} implies that the mutation sequence ${\bf v}^{\leq r}_{\DD^{-n}}$ is a composition of $n$ consecutive maximal green sequences for $\ov{Q}$ and the induced permutation on the vertices of $\ov{Q}$ coincides with the one induced by Proposition \ref{proposition:mutation sequence the gives left dual}. We deduce from Theorem \ref{thm:maximal green sequence gives inverse shift} that $p^*(\widetilde{N}) \cong p^*(\Sigma^{-n}N)$, where $p^*: \mc{H} \longrightarrow \mc{C}(\ov{Q},\ov{W})$ is the functor from Section \ref{section:relative cluster and Higgs categories}. Since $\widetilde{N}$ and $\Sigma^{-n}N$ are indecomposable objects that are not projective-injective in $\mc{H}$, the equivalence \eqref{eq:stable Higgs is absolute cluster} together with the previous isomorphism imply that $\wt{N} \cong \Sigma^{-n}N$, as claimed.

It remains to show that $\Sigma^{-n}N \in \mc{H}$ when $r$ is large enough. By the definition of the Higgs category, we must prove that
\[
\Ext^p_{\mc{C}}(\Sigma^{-n}N,e_s\bGh) = \Ext^p_{\mc{C}}(e_s\bGh,\Sigma^{-n}N) = 0
\]
for any $p \geq 0$ and $s \in F^{[a,0]}_0(\un{w}_0)$, where $\mc{C} = \mc{C}(Q^{[a,0]}(\un{w}_0),F^{[a,0]}(\un{w}_0),W^{[a,0]}(\un{w}_0))$ and $\bGh = \bGh^{[a,0]}(\un{w}_0)$. These equalities are equivalent to
\[
\Ext^{p+n}_{\mc{C}}(N,e_s\bGh) = \Ext^{p-n}_{\mc{C}}(e_s\bGh,N) = 0.
\]
Since $N \in \mc{H}$ and $n \geq 1$, the vanishing of the extension group on the left is immediate. We similarly deduce that the extension group on the right is zero if $p \geq n$. For $0 \leq p < n$, since $N = e_t\bGh$ for some $t \in [a,0]$, it suffices to show that
\begin{equation}\label{eq:vanishing of Ext in main theorem}
\Ext^{p-n}_{\bGh}(e_s\bGh,e_t\bGh) = 0
\end{equation}
by Proposition \ref{proposition:negative extensions can be computed in per Gamma}. To do so, we identify $Q^{[a,0]}(\un{w}_0)$ with the triangle product $Q_{\xi} \boxtimes Q^{\circ}_{r+1}$. Writing $A = kQ_{\xi} \ten A^{\circ}_{r+1}$ and $B = kQ_{\xi} \ten B^{\circ}_{r+1}$, where $A^{\circ}_{r+1}$ and $B^{\circ}_{r+1}$ are as in \eqref{eq:definition Al and Bl}, then $\bGh$ is quasi-isomorphic to $\bP_3(A,B)$ by Lemma \ref{lemma:complete/noncomplete Ginzburg are qiso when Jacobi-finite} and Section \ref{section:tensor products}. Therefore,
\begin{equation}\label{eq:RHom complex in main theorem}
\RHom_{\bGh}(e_s\bGh,e_t\bGh) \cong \bigoplus_{q \geq 0}\RHom_A(e_sA,F^q(e_tA)) \cong \bigoplus_{q \geq 0}F^q(e_tA)e_s,
\end{equation}
where $F = -\lten_A\Sigma^2\Omega_{A,B}$. But $e_tA \cong P \ten V_{[r+2-d,r+1]}$, where $P$ is the indecomposable projective $kQ_{\xi}$-module associated with $i_t \in \Delta_0$ and $d \geq 1$ denotes the integer such that $t$ is the $d$-th vertex of its row in $Q^{[a,0]}(\un{w}_0)$, counted from right to left. Note that $d$ does not change if we increase $r$. By Corollary \ref{corollary:action of Omega for tensor product}, we have
\[
F^q(e_tA) \cong \tau^{-q}P \ten V_{[r+2-d-q,r+1-q]}.
\]
Since $s$ is frozen, for $F^q(e_tA)e_s$ not to be zero, we must have $r+ 2-d-q \leq 1$, that is, $q \geq r+1-d$. If $r$ is large enough, this last inequality ensures that $\tau^{-q}P$ is a complex concentrated in a degree $m$ satisfying $m < -n$, so the same holds for $F^q(e_tA)e_s$. Hence, the complex \eqref{eq:RHom complex in main theorem} is concentrated in degrees strictly lower than $-n$ and \eqref{eq:vanishing of Ext in main theorem} holds for $0 \leq p < n$. This finishes the proof.
\end{proof}

\renewcommand*{\bibfont}{\small}
\sloppy
\printbibliography

\end{document}